\documentclass[11pt,a4paper]{article}
\usepackage{multirow}
\usepackage{epsfig}
\usepackage{epstopdf}
\usepackage[T1]{fontenc}
\usepackage{geometry}
\usepackage{amsbsy,amsmath,latexsym,amsfonts, color, authblk, amssymb, graphics, bm}
\usepackage{epsf,slidesec,epic,eepic}
\usepackage{fancybox}
\usepackage{fancyhdr}
\usepackage{setspace}
\usepackage{cases}
\usepackage{subfigure}
\usepackage{epsfig}
\usepackage{nccmath}
\usepackage{algorithm}
\usepackage{algorithmic}
\usepackage{mathrsfs}
\usepackage[colorlinks, citecolor=blue]{hyperref}
\usepackage{cleveref}
\usepackage{ulem}
\newtheorem{theorem}{Theorem}[section]
\newtheorem{lemma}{Lemma}[section]
\newtheorem{definition}{Definition}[section]

\newtheorem{proposition}{Proposition}[section]

\newtheorem{assumption}{Assumption}
\newtheorem{remark}{Remark}[section]

\newtheorem{alemma}{Lemma}

\newenvironment{proof}{{\noindent \bf Proof:}}{\hfill$\Box$\medskip}

\definecolor{lred}{rgb}{1,0.8,0.8}
\definecolor{lblue}{rgb}{0.8,0.8,1}
\definecolor{dred}{rgb}{0.6,0,0}
\definecolor{dblue}{rgb}{0,0,0.5}
\definecolor{dgreen}{rgb}{0,0.5,0.5}

\title{An inexact regularized proximal Newton method for nonconvex and nonsmooth optimization}
\author{
 Ruyu Liu\footnote{School of Mathematics, South China University of Technology, Guangzhou (\href{mailto:maruyuliu@mail.scut.edu.cn}{maruyuliu@mail.scut.edu.cn}).},\ \
 Shaohua Pan\footnote{School of Mathematics, South China University of Technology, Guangzhou(\href{mailto:shhpan@scut.edu.cn}{shhpan@scut.edu.cn}).},\ \
 Yuqia Wu\footnote{Department of Applied Mathematics, The Hong Kong Polytechnic University, 
 	Hong Kong(\href{mailto:yuqia.wu@connect.polyu.hk}{yuqia.wu@connect.polyu.hk}).},\ \ 
 Xiaoqi Yang\footnote{Department of Applied Mathematics, The Hong Kong Polytechnic University, 
 	Hong Kong(\href{mailto:mayangxq@polyu.edu.hk}{mayangxq@polyu.edu.hk}).}.
 }
\date{}
\begin{document}
 
\maketitle

%

\begin{abstract}
	This paper focuses on the minimization of a sum of a twice continuously differentiable function $f$ and a nonsmooth convex function. An inexact regularized proximal Newton method is proposed by an approximation to the Hessian of $f$ involving the $\varrho$th power of the KKT residual. For $\varrho=0$, we justify the global convergence of the iterate sequence for the KL objective function and its R-linear convergence rate for the KL objective function of exponent $1/2$. For $\varrho\in(0,1)$, by assuming that cluster points satisfy a locally H\"{o}lderian error bound of order $q$ on a second-order stationary point set and a local error bound of order $q>1\!+\!\varrho$ on the common stationary point set, respectively, we establish the global convergence of the iterate sequence and its superlinear convergence rate with order depending on $q$ and $\varrho$. A dual semismooth Newton augmented Lagrangian method is also developed for seeking an inexact minimizer of subproblems. Numerical comparisons with two state-of-the-art methods on $\ell_1$-regularized Student's $t$-regressions, group penalized Student's $t$-regressions, and nonconvex image restoration confirm the efficiency of the proposed method.
\end{abstract}

\section{Introduction}\label{sec1}

Given a data matrix $A\in\mathbb{R}^{m\times n}$ and a vector $b\in\mathbb{R}^{m}$, we are interested in the following nonconvex and nonsmooth composite optimization problem
\begin{equation}\label{prob}
	\min_{x\in\mathbb{R}^n}F(x):=f(x)+g(x)\ \ {\rm with}\ f(x):=\psi(Ax-b),
\end{equation}
where $\psi\!:\mathbb{R}^m\to\overline{\mathbb{R}}$ and $g\!:\mathbb{R}^n\to\overline{\mathbb{R}}$ with 
$\overline{\mathbb{R}}\!:=(-\infty,\infty]$ are proper lower semicontinuous (lsc) functions and satisfy the following basic assumption: 
\begin{assumption}\label{ass-1}
	\begin{description}
		\item[(i)] $\psi$ is twice continuously differentiable on an open set containing $A(\mathcal{O})-b$, where $\mathcal{O}\subset\mathbb{R}^n$ is an open set covering the domain ${\rm dom}\,g$ of $g$; 
		
		\item[(ii)] $g$ is convex and continuous relative to ${\rm dom}\,g$; 
		
		\item[(iii)] $\inf_{x\in\mathbb{R}^n}F(x)>-\infty$.
	\end{description} 
\end{assumption}

Assumption \ref{ass-1} (ii) means that model \eqref{prob} allows $g$ to be an indicator function of a closed convex set in $\mathbb{R}^n$, and   it also covers the case that $g$ is a weakly convex function. Indeed, by recalling that $g$ is $\alpha$-weakly convex if $g(\cdot)\!+\!(\alpha/2)\|\cdot\|^2$ is convex for some $\alpha\ge0$, $F$ can be rewritten as $F=\overline{f}\!+\!\overline{g}$ with $\overline{f}(\cdot)\!=f(\cdot)\!-\!(\alpha/2)\|\cdot\|^2$ and $\overline{g}(\cdot)\!=g(\cdot)\!+\!(\alpha/2)\|\cdot\|^2$. Note that $\overline{f}$ can be reformulated as $\overline{\psi}(\overline{A}\cdot-\overline{b})$ for suitable $\overline{A}$ and $\overline{b}$ with $\overline{\psi}(\cdot)=\psi(\cdot)-(\alpha/2)\|\cdot\|^2$. Hence, $\overline{f}$ and $\overline{g}$ conform to Assumption \ref{ass-1}. 

Model \eqref{prob} has a host of applications in statistics, signal and image processing, machine learning, financial engineering, and so on. 
For example, the popular lasso \cite{Tibshirani96} and sparse inverse covariance estimation \cite{Yuan06} in statistics are the special instances of \eqref{prob} with a convex $\psi$. In some inverse problems, non-Gaussianity of noise or nonlinear relation between measurements 
and unknowns often leads to \eqref{prob} with a nonconvex $\psi$ (see
\cite{Bonettini17}). In addition, the higher moment portfolio selection problem (see \cite{Dihn11,Zhou21}) also takes
the form of \eqref{prob} with a nonconvex $\psi$.
\subsection{Related works}\label{sec1.1}

For problem \eqref{prob}, many types of methods have been developed. Fukushima and Mine \cite{Fukushima81} introduced originally the proximal gradient (PG) method;  Tseng and Yun \cite{Tseng09} proposed a block coordinate decent method and obtained the subsequence convergence of the iterate sequence and its R-linear convergence rate under the Luo-Tseng error bound; Milzarek \cite{Milzarek14} developed a class of methods by virtue of a combination of semismooth Newton steps, a filter globalization, and thresholding steps for \eqref{prob} with $g(x)=\mu\|x\|_1$, and achieved subsequence convergence and local $q$-superlinear convergence properties for $q\in(1,2]$; Bonettini et al. \cite{Bonettini20} extended their variable metric inexact line-search (VMILA) method \cite{Bonettini16} by incorporating a forward-backward step, and verified the global convergence of the iterate sequence and the linear convergence rate of the objective value sequence under the uniformly bounded positive definiteness of the scaled matrix and the KL property of exponent $\theta\in(0,1/2]$ of the forward-backward envelope (FBE) of $F$; and by using the FBE of $F$, initially introduced in \cite{Patrinos13}, Stella et al. studied a combination of PG step and quasi-Newton step of FBE of $F$ with a line search at the iterate, verified the global convergence for a KL function $F$ and the superlinear convergence rate under the nonsingularity of the Hessian of the FBE in \cite{Stella17}, and obtained the same properties as in \cite{Stella17} for \eqref{prob} but with a nonconvex $g$ by using an Armijo line search at the PG output of iterate in \cite{Themelis18}. 

Next we mainly review inexact proximal Newton methods that are closely related to this work. This class of methods, also called an inexact successive quadratic approximation method, is finding in each iterate an approximate minimizer $y^k$ satisying a certain inexactness criterion for
a subproblem of the following form 
\begin{equation}\label{subprobx}
	\mathop{\min}_{x\in\mathbb{R}^n}
	\Theta_k(x)\!:=\!f(x^k)+\langle\nabla\!f(x^k),x-\!x^k\rangle
	+\frac{1}{2}(x\!-x^k)^{\top}G_k(x\!-\!x^k)+g(x),
\end{equation}
where $x^k$ is the current iterate, and $G_k$ is a symmetric positive definite matrix that represents a suitable approximation of the Hessian $\nabla^2\!f(x^k)$. The proximal Newton method can be viewed as a special variable metric one, and it will reduce to the PG method if $G_k\!=\!\gamma_k I$ with $\gamma_k>0$ related to the Lipschitz constant of $\nabla\!f$. Note that subproblem \eqref{subprobx} is seeking a root of $0\in \nabla\!f(x^k)+G_k(x-x^k)+\partial g(x)$, the partially linearized version at $x^k$ of the stationary point equation $0\in\partial F(x)$, where $\partial F(x)$ denotes the basic (limiting or Mordukhovich) subdifferential of $F$ at $x$. The proximal Newton method belongs to the quite general iterative framework proposed by Fischer \cite{Fischer02} if the inexactness criterion there is used, but it is not implementable due to the involved unknown stationary point set. As pointed out in \cite{Yue19}, the proximal Newton method depends more or less on three key ingredients: the approximation matrix $G_k$, the inner solver to subproblem \eqref{subprobx}, and the inexactness criterion on $y^k$ (i.e., the stopping criterion of the inner solver to control the inexactness of $y^k$). Since \eqref{subprobx} takes into account the second-order information of $f$, the proximal Newton method has a remarkable superiority to the PG method, i.e., a faster convergence rate. 

Early proximal Newton methods were tailored for special instances of convex $\psi$ and $g$ in problem \eqref{prob} such as GLMNET \cite{Friedman07}, newGLMNET \cite{Yuan12}, QUIC \cite{Hsieh11} and the Newton-Lasso method \cite{Oztoprak12}. Lee et al. presented a generic version of the exact proximal Newton method in \cite{Lee14}, achieved a global convergence result by requiring the uniform positive definiteness of $G_k$, and established a local linear or superlinear convergence rate depending on the forcing term on a stopping criterion for an inexact proximal Newton method with unit step-size. Li et al. \cite{Li16} extended the exact proximal Newton method proposed in \cite{Dihn15} for self-concordant functions $f$ to the proximal Newton method with inexact steps, and achieved the local linear, superlinear or quadratic convergence rate resting with the parameter in the inexact criterion under the positive definiteness assumption of $\nabla^2\!f$ on ${\rm dom}\,f$. Yue et al. \cite{Yue19} proposed an inexact proximal Newton method with a regularized Hessian by an inexactness condition depending on the KKT residual of the original problem, and established the superlinear and quadratic convergence rates under the Luo-Tseng error bound. As far as we know, their work is the first to achieve the superlinear convergence without the strong convexity of $F$ for an implementable proximal Newton method, though Fisher \cite{Fischer02} ever got the superlinear convergence rate for the proposed iterative framework, covering the proximal Newton method with an impractical inexactness criterion, under the calmness of the mapping $(\partial F)^{-1}$. Mordukhovich et al. \cite{Mordu20} also studied a similar inexact regularized proximal Newton method, and achieved the $R$-superlinear convergence rate under the metric $q$-subregularity of $\partial F$ for $q\in(\frac{1}{2},1)$, and the quadratic convergence rate under the metric subregularity of $\partial F$ that is equivalent to the calmness of $(\partial F)^{-1}$. Their metric $q$-subregularity condition is weaker than the Luo-Tseng error bound. 

For problem \eqref{prob} with $g(x)=\mu\|x\|_1$, Byrd et al. \cite{Byrd16} studied an inexact proximal Newton method with an approximate solution criterion determined by the KKT residual of \eqref{prob}. They showed that the KKT residual sequence converges to zero under the uniformly bounded positive definiteness of $G_k$ and obtained local superlinear
and quadratic convergence rates under the positive definiteness and the Lipschitz continuity of $\nabla^2\!f$ at stationary points. For \eqref{prob} with an optimal-set-strongly-convex $F$, Lee and Wright \cite{Lee19} investigated an inexact proximal Newton method with an approximate solution criterion relying on the difference between the objective function value of \eqref{subprobx} at the iterate and its optimal value 
and established a global linear convergence rate of the objective value sequence under the uniformly bounded positive definiteness of $G_k$. Recently, Kanzow and Lechner \cite{Kanzow21} proposed a globalized inexact proximal Newton-type (GIPN) method by switching from a Newton step to a PG step when the proximal Newton direction does not satisfy a sufficient decrease condition, and established the global convergence and superlinear convergence rate under the uniformly bounded positive definiteness of $G_k$ and the local strong convexity of $F$. 

From the above discussions, we see that for the nonconvex problem \eqref{prob}, the existing global convergence results of the proximal Newton methods require the uniform positive definiteness of $G_k$, while the local superlinear (or quadratic) convergence results assume that $F$ is locally strongly convex in a neighborhood of cluster points of the iterate sequence. The local strong convexity of $F$ in a neighborhood of a stationary point implies the isolatedness of this stationary point, and then the Luo-Tseng error bound and subsequently the metric subregularity of the mapping $\partial F$. Inspired by the works \cite{Yue19,Mordu20}, it is natural to ask whether an inexact proximal Newton method can be designed for problem \eqref{prob} to possess a superlinear convergence rate without the local strong convexity of $F$. In addition, we observe that when the power $\varrho = 0$ in the regularized Hessian \eqref{def-G} below, the global convergence of the iterate sequence in \cite{Yue19} requires the Luo-Tseng's error bound as does for its linear convergence rate, and in addition their linear convergence rate result depends upon that the parameter of the method is upper bounded by the unknown constant of the error bound. Then, it is natural to ask whether it is possible to achieve the global convergence and linear convergence rate of the iterate sequence for \eqref{prob} under a weaker condition. This work aims to resolve these two questions for the nonconvex and nonsmooth problem \eqref{prob}.

\subsection{Our contributions}\label{sec1.2}

Motivated by the structure of $f$ and the work \cite{Ueda10} for a smooth unconstrained problem, we adopt the following regularized version of the Hessian $\nabla^2f(x^k)$:
\begin{equation}\label{def-G}
	G_k=\!\nabla^2\!f(x^k)+a_1[-\lambda_{\rm min}(\nabla^2\psi(Ax^k\!-\!b))]_{+}A^{\top}\!A+a_2[r(x^k)]^{\varrho}I
\end{equation}
to construct a strongly convex approximation of $F$ at iterate $x^k$, and propose an inexact regularized proximal Newton method (IRPNM) for 
\eqref{prob}, where $a_{+}\!:=\max(0,a)$ for $a\in\mathbb{R}$, $\lambda_{\rm min}(H)$ denotes the smallest eigenvalue of $H$,  
$r(x^k)$ is the KKT residual of \eqref{prob} at $x^k$ (see \eqref{Rmap} for its definition), and $a_1\!\ge 1,a_2\!>0$ and $\varrho\in[0,1)$ are the constants. Different from the regularized Hessian in \cite{Ueda10}, we here use $a_1[-\lambda_{\rm min}(\nabla^2\psi(Ax^k\!-\!b))]_{+}A^{\top}\!A$ to replace $a_1[-\lambda_{\rm min}(\nabla^2\!f(x^k))]_{+}I$ in order to avoid the computation of the smallest eigenvalue when $\psi$ is separable, i.e., $\psi(y)\!:=\sum_{i=1}^m\psi_i(y_i)$ with each $\psi_i:\mathbb{R}\to\overline{\mathbb{R}}$ being twice continuously differentiable on a suitable set. The matrix $G_k$ in \eqref{def-G} is uniformly positive definite when $\varrho = 0$ but not when $\varrho \in (0,1)$ because $r(x^k)\rightarrow 0$ as $k\rightarrow\infty$ as will be shown later. Our inexactness criterion on $y^k$ depends on the nonincreasing of the objective value of \eqref{subprobx}, along with the KKT residual $r(x^k)$ when $\varrho\in(0,1)$ and the approximate optimality of $y^k$ when $\varrho=0$; see criteria \eqref{inexact-conda} and \eqref{inexact-condb} below. In addition, the Armijo line search is imposed on the direction $d^k\!:=y^k-x^k$ to achieve a sufficient decrease of $F$. Our contributions are summarized as follows.

For $\varrho=0$, we achieve a global convergence of the iterate sequence 
for the KL function $F$ and its $R$-linear convergence rate for the KL function $F$ of exponent $1/2$, which is weaker than the Luo-Tseng error bound. In this case, our regularized proximal Newton method is similar to the VMILA in \cite{Bonettini17,Bonettini20} except that a different inexactness criterion and a scaled matrix involving the Hessian of $f$ are used. Compared with the convergence results in \cite{Bonettini20}, which removes the restriction condition (see \cite[Eq. (23)]{Bonettini17}) on the iterate sequence in the convergence analysis of \cite{Bonettini17}, our $R$-linear convergence rate is obtained for the iterate sequence instead of the objective value sequence, and the required KL property of exponent $1/2$ is for $F$ itself rather than its FBE. Though by \cite[Remark 5.1(ii)]{YuLiPong21} the KL property of $F$ with exponent $1/2$ implies that of its FBE, this requires the global Lipschitz continuity of $\nabla\!f$, which does not fit in our setting.

For $\varrho\in(0,1)$, we establish the global convergence of the iterate sequence and its superlinear convergence rate with order $q(1\!+\!\varrho)$, under an assumption that cluster points satisfy a locally H\"{o}lderian error bound of order $q\in(\max\{\varrho,(1\!+\!\varrho)^{-1}\},1]$ on a second-order stationary point set.  This result not only extends the conclusion of \cite[Theorem 5.1]{Mordu20} to problem \eqref{prob}, but also discards the local strong convexity required in  the convergence analysis of the (regularized) proximal Newton methods for this class of nonconvex and nonsmooth problems \cite{Byrd16,Kanzow21}. When cluster points satisfy a local error bound of order $q>1\!+\!\varrho$ on the common stationary point set, we also achieve the global convergence of the iterate sequence and its superlinear convergence rate of order ${(q-\varrho)^2}/{q}$ for $q>[\varrho+\!1/2+\!\sqrt{\varrho+\!1/4}]$, which bridges the gap that the second-order stationary point set may be empty. Compared with the superlinear convergence results in \cite{Themelis18} for the hybrid method of the PG steps and quasi-Newton steps of FBE of $F$, ours do not require twice epi-differentiability of $g$ and the strong local optimality of the limit (which is actually an isolated local minimizer), though there is no direct implication between our local error bound condition and their Dennis-Mor\'e condition.

In addition, inspired by the structure of $G_k$, we also develop a dual semismooth Newton augmented Lagrangian method (SNALM) to compute the approximate minimizer $y^k$ of \eqref{subprobx} satisying the inexactness criterion \eqref{inexact-conda}, and compare the performance of our IRPNM armed with SNALM with that of GIPN \cite{Kanzow21} and ZeroFPR \cite{Themelis18} on $\ell_1$-regularized Student's $t$-regressions, group penalized Student's $t$-regressions and nonconvex image restoration. Numerical comparisons indicate that IRPNM is superior to GIPN in the objective value and the running time, and it is comparable with ZeroFPR in terms of the objective value, but requires much less running time than ZeroFPR if the obtained stationary point is a second-order one (such as for $\ell_1$-regularized and group penalized Student's $t$-regressions), otherwise requires more running time than the latter (such as for nonconvex image restoration). Such a numerical performance is entirely consistent with the theoretical results. 
\subsection{Notation}\label{sec1.3}

Throughout this paper, $\mathbb{S}^n$ represents the set of all $n\times n$ real 
symmetric matrices, $\mathbb{S}_{+}^n$ denotes the cone consisting of all positive
semidefinite matrices in $\mathbb{S}^n$, and $I$ denotes an identity matrix
whose dimension is known from the context. For a real symmetric matrix $H$,
$\|H\|$ denotes the spectral norm of $H$, and $H\succeq 0$ means that $H\in\mathbb{S}_{+}^n$. For a closed set $C\subset\mathbb{R}^n$, 
$\Pi_{C}$ denotes the projection operator onto the set $C$, 
${\rm dist}(x,C)$ means the Euclidean distance of a vector $x\in\mathbb{R}^n$ 
to the set $C$, and $\delta_{C}$ denotes the indicator function of $C$. 
For a vector $x\in\mathbb{R}^n$, $\mathbb{B}(x,\delta)$ denotes the closed ball 
centered at $x$ with radius $\delta>0$. For a
multifunction $\mathcal{F}\!:\mathbb{R}^n\rightrightarrows\mathbb{R}^n$, its graph is defined as ${\rm gph}\,\mathcal{F}:=\{(x,y) \ | \ y \in \mathcal{F}(x)\}$. A function $h\!:\mathbb{R}^n\to\overline{\mathbb{R}}$ is said to be proper if its domain ${\rm dom}\,h:=\{x\in\mathbb{R}^n\ |\ h(x)<\infty\}$ is nonempty.  For a proper $h\!:\mathbb{R}^n\to\overline{\mathbb{R}}$ and a point $x\in{\rm dom}\,h$,  $h'(x;d)\!:=\!\lim_{\tau\downarrow0}(h(x+\tau d)-h(x))/\tau$ denotes its directional derivative at $x$ along a direction $d\in\mathbb{R}^n$, $\partial h(x)$ denotes its basic (or limiting) subdifferential at $x$, and we write $[\alpha<h<\beta]\!:=\{x\in\mathbb{R}^n\,|\,\alpha<h(x)<\beta\}$ for real numbers $\alpha<\beta$.

\section{Preliminaries}\label{sec2} 
    For a closed proper function $h\!:\mathbb{R}^n\to\overline{\mathbb{R}}$,
	its proximal mapping $\mathcal{P}_{\!\gamma h}$ and Moreau envelope 
	$e_{\gamma h}$ associated to a parameter $\gamma>0$ are respectively defined as 
	\[
	\mathcal{P}_{\!\gamma h}(x)\!:=\!\mathop{\arg\min}_{z\in\mathbb{R}^n}
	\Big\{\frac{1}{2\gamma}\|z-x\|^2+h(z)\Big\}\ {\rm and}\  
	e_{\gamma h}(x)\!:=\!\min_{z\in\mathbb{R}^n}\Big\{\frac{1}{2\gamma}\|z-x\|^2+h(z)\Big\}.
	\] 
	By \cite[Theorem 12.12]{RW98}, when $h$ is convex, the mapping $\mathcal{P}_{\!\gamma h}$ is nonexpansive, i.e., $\|\mathcal{P}_{\!\gamma h}(x)-\mathcal{P}_{\!\gamma h}(y)\|\le\|x-y\|$ for any $x,y\in\mathbb{R}^n$. We also need the strict continuity of a function at a point relative to a set containing this point. By \cite[Definition 9.1]{RW98}, a function $h\!:\mathbb{R}^n\to\overline{\mathbb{R}}$ is strictly continuous at a point $\overline{x}$ relative to a set $D\subset {\rm dom}\,h$ if $\overline{x}\in D$ and the Lipschitz modulus of $h$ at $\overline{x}$, denoted by ${\rm lip}_{D}h(\overline{x})$, is finite. That is, to say that $h$ is strictly continuous at $\overline{x}$ relative to $D$ is to assert the existence of a neighborhood $V$ of $\overline{x}$ such that $h$ is Lipschitz continuous on $D\cap V$.

Next we recall the concept of stationary point and clarify its relation with $L$-stationary point, and also introduce a kind of second-order stationary point for \eqref{prob}. 
\subsection{Stationary points of problem \eqref{prob}}\label{sec2.1}

	Recall that a vector $x\in{\rm dom}\,g$ is a stationary point of \eqref{prob} if $x$ is a critical point of $F$, i.e.,
	$0\in\partial F(x)=\nabla\!f(x)+\partial g(x)$, and we denote by $\mathcal{S}^*$ the set of stationary points. Define the KKT residual mapping and residual function of \eqref{prob} respectively by
	\begin{equation}\label{Rmap}
		R(x)\!:=x\!-\!\mathcal{P}_g(x\!-\!\nabla\!f(x))\ \ {\rm and}\ \ 
		r(x):=\|R(x)\|\quad{\rm for}\ x\in\mathbb{R}^n.
	\end{equation}
	By the convexity of $g$, it is easy to check that $0\in\partial F(x)$
	if and only if $r(x)=0$. By \cite[Definition 4.1]{Beck19}, for a vector $x\in\mathbb{R}^n$, if there exists a constant $L>0$ such that $x=\mathcal{P}_{\!L^{-1}g}(x\!-\!L^{-1}\nabla\!f(x))$, then it is called an $L$-stationary point of \eqref{prob}. By the convexity of $g$, one can check that $x$ is a stationary point of \eqref{prob} if and only if it is an $L$-stationary point of \eqref{prob}, and if $x$ is an $L$-stationary point of \eqref{prob} for some $L>0$, then it is also $L$-stationary for any $L>0$. We call $x^*\in\mathcal{S}^*$ a second-order stationary point if $\nabla^2\psi(Ax^*\!-\!b)\succeq 0$, and we denote by $\mathcal{X}^*$ the set of second-order stationary points of \eqref{prob}. By the expression of $f$ and Assumption \ref{ass-1} (i), $\mathcal{X}^*\subset\big\{x\in\mathcal{S}^*\,|\,\nabla^2\!f(x)\succeq 0\big\}$, and the inclusion will become an equality if $A$ has a full row rank. By Assumption \ref{ass-1} (i) and the outer semicontinuity of $\partial g$, both $\mathcal{S}^*$ and  $\mathcal{X}^*$ are closed. Note that $\mathcal{X}^*$ may be empty even if $\mathcal{S}^*$ is nonempty. As will be shown by Proposition \ref{prop-xk} (iv) later, $\mathcal{S}^*$ is nonempty under the boundedness of a level set of $F$. A local minimizer of $F$ may not be a second-order stationary point, and the converse is not necessarily true either. 

Let $\vartheta\!:\mathcal{O}\to\mathbb{R}$ be a continuously differentiable function.
Consider the problem
\begin{equation}\label{prob1}
	\min_{x\in\mathbb{R}^n}\big\{\vartheta(x)+g(x)\big\}
\end{equation}
and its canonical perturbation problem induced by a parameter vector $u\in\mathbb{R}^n$:
\begin{equation}\label{u-prob}
	\min_{x\in\mathbb{R}^n}\big\{\vartheta(x)+g(x)-\langle u,x\rangle\big\}.
\end{equation}
The following proposition states the relation between an approximate stationary point
of problem \eqref{prob1} and a stationary point of its canonical perturbation \eqref{u-prob}.
\begin{proposition}\label{prop-stationary}
	Let $R_{\vartheta}(x)\!:=x\!-\!\mathcal{P}_g(x\!-\!\nabla\vartheta(x))$ for  $x\in\mathbb{R}^n$. Fix any $\epsilon>0$. If problem \eqref{prob1} has a nonempty stationary point set, then it has an $\epsilon$-approximate stationary point $\widehat{x}$ (i.e., $\widehat{x}\in{\rm dom}\,g$ and $\|R_{\vartheta}(\widehat{x})\!+\!\nabla\vartheta(\mathcal{P}_g(\widehat{x}\!-\!  \nabla\vartheta(\widehat{x})))\!-\!\nabla\vartheta(\widehat{x})\|\le\epsilon$) such that   $\overline{x}_{u}\!:=\mathcal{P}_g(\widehat{x}\!-\!\nabla\vartheta(\widehat{x}))$ is a stationary point of \eqref{u-prob} with $u=R_{\vartheta}(\widehat{x})+\!\nabla\vartheta(\overline{x}_{u})	-\!\nabla\vartheta(\widehat{x})$.
\end{proposition}
\begin{proof}
	Let $\{x^k\}\subset{\rm dom}\,g$ be a sequence converging to a stationary point
	of \eqref{prob1}. For each $k$, define $u^k:=R_{\vartheta}(x^k)\!
	+\nabla\vartheta(\mathcal{P}_g(x^k-\!\nabla\vartheta(x^k)))-\!\nabla\vartheta(x^k)$.
	From the continuity, $\lim_{k\to\infty}\|u^k\|=0$, so there exists
	$\widehat{x}\in{\rm dom}\,g$ with 
	$\|R_{\vartheta}(\widehat{x})+\!\nabla\vartheta(\mathcal{P}_g(\widehat{x}-\! \nabla\vartheta(\widehat{x})))-\!\nabla\vartheta(\widehat{x})\|\le\epsilon$. 
	Then, $u=R_{\vartheta}(\widehat{x})+\!\nabla\vartheta(\mathcal{P}_g(\widehat{x}-\!  \nabla\vartheta(\widehat{x})))-\!\nabla\vartheta(\widehat{x})$ is well defined.
	From the definition of $\overline{x}_{u}$ and the expression of
	$R_{\vartheta}(\widehat{x})$, $R_{\vartheta}(\widehat{x})-\nabla\vartheta(\widehat{x})
	\in\partial g(\overline{x}_{u})$. Consequently, 
	$u\in\nabla\vartheta(\overline{x}_u)+\partial g(\overline{x}_u)$,
	and $\overline{x}_u$ is a stationary point of \eqref{u-prob} associated to $u$. 
\end{proof}

The Kurdyka-{\L}ojasiewicz (KL) property plays a crucial role in the convergence analysis of algorithms for nonconvex and nonsmooth optimization problems \cite{Attouch10,Attouch13}, while the metric $q$-subregularity of a multifunction has been used to analyze the local superlinear and quadratic convergence rates of the proximal Newton-type method for nonsmooth composite convex optimization \cite{Mordu20}. Next we explore the relation between the KL property of $F$ and the metric $q$-subregularity of the mapping $\partial F$. These two kinds of regularity are used in the convergence analysis of our algorithm in section \ref{sec4}.   

\subsection{KL property and metric $q$-subregularity}\label{sec2.3}

First we recall the KL property of an extended real-valued function. For each $\varpi\in(0,\infty]$, denote by $\Upsilon_{\!\varpi}$ the family of continuous concave functions $\varphi\!:[0,\varpi)\to\mathbb{R}_{+}$ with $\varphi(0)=0$ that are continuously differentiable on $(0,\varpi)$ with $\varphi'(s)\!>0$ for $s\in(0,\varpi)$.
\begin{definition}\label{KL-def}
	A proper function $h\!:\mathbb{R}^n\to\overline{\mathbb{R}}$ is said to have the KL property at a point $\overline{x}\in{\rm dom}\,\partial h$ if there exist $\delta>0,\varpi\in(0,\infty]$ and $\varphi\in\Upsilon_{\!\varpi}$ such that for all $x\in\mathbb{B}(\overline{x},\delta)\cap\big[h(\overline{x})<h<h(\overline{x})+\varpi\big]$, $\varphi'(h(x)\!-\!h(\overline{x})){\rm dist}(0,\partial h(x))\ge 1$. If $\varphi$ can be chosen to be the function $t\mapsto ct^{1-\theta}$ with $\theta\in[0,1)$ for some $c>0$, then $h$ is said to have the KL property of exponent $\theta$ at $\overline{x}$. If $h$ has the KL property (of exponent $\theta$) at each point of ${\rm dom}\,\partial h$, it is called a KL function (of exponent $\theta$).
\end{definition}

By \cite[Lemma 2.1]{Attouch10}, a proper lsc function $h\!:\mathbb{R}^n\to\overline{\mathbb{R}}$ has the KL property of exponent $0$ at every noncritical point (i.e., the point at which the limiting subdifferential of $h$ does not contain $0$). Thus, to show that a proper lsc function is a KL function of exponent $\theta\in[0,1)$, it suffices to check its KL property of exponent $\theta\in[0,1)$ at all critical points. On the calculation of the KL exponent, we refer the readers to \cite{LiPong18,YuLiPong21,WuPanBi21}. As illustrated in \cite[Section 4]{Attouch10}, KL functions are rather extensive and cover semialgebraic functions, global subanalytic functions, and functions definable in an o-minimal structure over the real field $(\mathbb{R},+,\cdot)$. 

Next we give the formal definition of the metric $q$-subregularity of a multifunction. 
\begin{definition}\label{Def2.2}(see \cite[Definition 3.1]{Li12})
	Let $\mathcal{F}\!:\mathbb{R}^n\rightrightarrows\mathbb{R}^n$ be a multifunction. Consider any point $(\overline{x},\overline{y})\in{\rm gph}\,\mathcal{F}$. For a given $q>0$, we say that $\mathcal{F}$ is (metrically) $q$-subregular at $\overline{x}$ for $\overline{y}$ if there exist $\kappa>0$ and $\delta>0$ such that for all $x\in\mathbb{B}(\overline{x},\delta)$,
	\begin{equation}\label{q-subregular}
		{\rm dist}(x,\mathcal{F}^{-1}(\overline{y}))\le\kappa[{\rm dist}(\overline{y},\mathcal{F}(x))]^{q}.
	\end{equation}
	When $q=1$, this property is called the (metric) subregularity of $\mathcal{F}$ at $\overline{x}$ for $\overline{y}$.
\end{definition}

By Definition \ref{Def2.2}, if $\overline{x}\in{\rm dom}\,\mathcal{F}$ is an isolated point, $\mathcal{F}$ is subregular at $\overline{x}$ for any $\overline{y}\in\mathcal{F}(\overline{x})$; and if $\mathcal{F}(\overline{x})$ is closed, the subregularity of $\mathcal{F}$ at $\overline{x}$ for $\overline{y}\in\mathcal{F}(\overline{x})$ implies its $q$-subregularity at $\overline{x}$ for $\overline{y}$ for any $q\in(0,1)$ (now also known as the $q$-order H\"{o}lderian subregularity). The following lemma shows that the $q$-subregularity of the mapping $\partial F$ is equivalent to that of the mapping $R$ defined in \eqref{Rmap}. It is worth pointing out that such an equivalence was ever obtained in \cite{Dong09} and \cite[page 21]{Drusvyatskiy18} for $q=1$.  
\begin{lemma}\label{sregular-relation}
	Consider any $\overline{x}\in\mathcal{S}^*$ and $q>0$. If the mapping $\partial F$ is $q$-subregular at $\overline{x}$ for $0$, then the residual mapping $R$ is $\min\{q,1\}$-subregular at $\overline{x}$ for $0$. Conversely, if the mapping $R$ is $q$-subregular at $\overline{x}$ for $0$, so is the mapping $\partial F$ at $\overline{x}$ for $0$. 
\end{lemma}
\begin{proof}
	Suppose that $\partial F$ is $q$-subregular at $\overline{x}$ for $0$.
	There exist $\varepsilon>0,\kappa>0$ such that
	\begin{equation}\label{subregular-ineq0}
		{\rm dist}(z,\mathcal{S}^*)\le\kappa[{\rm dist}(0,\partial F(z))]^{q}
		\quad{\rm for\ all}\ z\in\mathbb{B}(\overline{x},\varepsilon).
	\end{equation}
	Since $\nabla\!f$ is strictly continuous on $\mathcal{O}$ by Assumption \ref{ass-1} (i), there exist $\widetilde{\varepsilon}\in(0,{1}/{2})$ and $L'>0$ such that for all $z,z'\in\mathbb{B}(\overline{x},\widetilde{\varepsilon})$, \begin{equation}\label{fgrad-ineq1}
		\|\nabla\!f(z')-\nabla\!f(z)\|\le L'\|z'-z\|.
	\end{equation}
	From $R(\overline{x})=0$ and the continuity of $R$ at $\overline{x}$, we have $\|R(z)\|\le 1$ for all $z\in\mathbb{B}(\overline{x},\widetilde{\varepsilon})$ (if necessary by shrinking $\widetilde{\varepsilon}$).
	Pick any $x\in\mathbb{B}(\overline{x},\delta)$ with $\delta=\min\{\varepsilon,\widetilde{\varepsilon}\}/(1+\!L')$.
	Write $u\!:=R(x)=x-\mathcal{P}_g(x-\!\nabla f(x))$. Note that $\overline{x}=\mathcal{P}_g(\overline{x}-\!\nabla f(\overline{x}))$. Then, 
	\begin{align*}
		\|x-u-\overline{x}\|
		&=\|\mathcal{P}_g(x-\nabla\!f(x))-\mathcal{P}_g(\overline{x}-\nabla\!f(\overline{x}))\|\\
		&\le\|x-\overline{x}+\nabla\!f(\overline{x})-\nabla\!f(x)\|
		\le(1+L')\delta=\min\{\varepsilon,\widetilde{\varepsilon}\}, 
	\end{align*}
	so that $x-u\in\mathbb{B}(\overline{x},\min\{\varepsilon,\widetilde{\varepsilon}\})$, where the first inequality is due to nonexpansiveness of $\mathcal{P}_g$, and the second one is using \eqref{fgrad-ineq1}. In addition, from $u=x-\mathcal{P}_g(x-\nabla\!f(x))$, we deduce that $\nabla\!f(x-\!u)+u-\nabla\!f(x)\in\partial F(x-\!u)$.
	Now using \eqref{subregular-ineq0} with $z=x-u$ and \eqref{fgrad-ineq1} with 
	$z'=x-u$ and $z=x$ yields that
	\[
	{\rm dist}(x\!-\!u,\mathcal{S}^*)\le\kappa\|\nabla\!f(x\!-\!u)+u-\nabla\!f(x)\|^q\le \kappa(1\!+\!L')^q\|u\|^q.
	\]
	Then, ${\rm dist}(x,\mathcal{S}^*)\le\|u\|+\kappa(1\!+L')^q\|u\|^q
	\le[1\!+\kappa(1\!+\!L')^q]\|R(x)\|^{\min\{q,1\}}$, where the second inequality is using $\|R(x)\|\le 1$. By the arbitrariness of $x$ in  
	$\mathbb{B}(\overline{x},\delta)$, the mapping $R$ is $\min\{q,1\}$-subregular at $\overline{x}$ for $0$. Conversely, suppose that the mapping $R$ is $q$-subregular at $\overline{x}$ for $0$. Then, there exist $\delta>0$ and $\nu>0$ such that 
	\begin{equation}\label{subregular-ineq-1}
		{\rm dist}(z,\mathcal{S}^*)\le\nu\|R(z)\|^q\quad\ {\rm for\ all}\ z\in\mathbb{B}(\overline{x},\delta).
	\end{equation}
	Pick any $x\in\mathbb{B}(\overline{x},\delta)\cap{\rm dom}\,\partial F$.
	 By the closedness of $\partial F$, there exists $\xi\in\partial F(x)$ such that $\|\xi\|={\rm dist}(0,\partial F(x))$. 
	From $\xi\in\partial F(x)$ and $\mathcal{P}_g(z)=(I\!+\!\partial g)^{-1}(z)$ for any $z\in\mathbb{R}^n$, we derive $x=\mathcal{P}_g(x+\xi\!-\!\nabla\!f(x))$ and $R(x)=\mathcal{P}_g(x+\xi\!-\!\nabla\!f(x))\!-\!\mathcal{P}_g(x\!-\!\nabla\!f(x))$. The latter, 
	 by the nonexpansiveness of $\mathcal{P}_g$, implies that $\|R(x)\|\le \|\xi\|={\rm dist}(0,\partial F(x))$.
	Together with \eqref{subregular-ineq-1}, we obtain
	${\rm dist}(x,\mathcal{S}^*)\le\nu[{\rm dist}(0,\partial F(x))]^q$, which holds trivially if $x\in\mathbb{B}(\overline{x},\delta)\backslash {\rm dom}\,\partial F$. Thus, the mapping $\partial F$ is $q$-subregular at $\overline{x}$ for $0$ 
\end{proof}

To achieve the relationship between the $q$-subregularity of $\partial F$ for $q>0$ and the KL property of $F$ with exponent $\theta\in(0,1)$, we need the following assumption.
\begin{assumption}\label{ass0}
	For any given $\overline{x}\!\in\mathcal{S}^*$, there exists $\epsilon>0$ such that $F(y)\le F(\overline{x})$ for all $y\in\mathcal{S}^*\cap\mathbb{B}(\overline{x},\epsilon)$.	
\end{assumption}
\begin{remark}\label{KL-assump}
	{\bf(a)} Obviously, Assumption \ref{ass0} is implied by \cite[Assumption 4.1]{LiPong18}, which can be regarded as a local version of \cite[Assumption B]{Luo93}.  In addition, we can provide an example for which Assumption \ref{ass0} holds, but \cite[Assumption 4.1]{LiPong18} does not hold. Let $F\equiv f$ with $f(x):=-e^{-\frac{1}{x^2}}(\sin\frac{1}{x})^2$ for $x \neq 0$ and $f(0):= 0$. It is easy to check that $F$ is smooth and $F'(0)=0$. Fix any $\epsilon\in(0,1/2)$. Pick any $y\in\mathbb{B}(0,\epsilon)\cap\mathcal{S}^*$. Clearly, $F(y)=f(y)\le 0=F(0)$, i.e., Assumption \ref{ass0} holds. Now let $y^1:=\frac{1}{(k+1)\pi}$ and $y^2:= \frac{1}{k\pi}$ with $k=\lceil \frac{1}{\epsilon \pi}\rceil+1$. Obviously, $y^1, y^2\in\mathbb{B}(0,\epsilon)$ with $f(y^1)=f(y^2)=0$. By Rolle's theorem, there must exist $y_0\in (y_1, y_2)$ such that $f'(y_0) =0$. Note that $f(y)<0$ for any $y\in (y^1, y^2)$, so that $f(y_0)<0$, which shows that \cite[Assumption 4.1]{LiPong18} does not hold. Thus, we conclude that Assumption \ref{ass0} is weaker than \cite[Assumption 4.1]{LiPong18}.  
	
	\noindent
	{\bf(b)} When $F$ has the KL property at $\overline{x}\in\mathcal{S}^*$, Assumption \ref{ass0} necessarily holds at $\overline{x}$. Indeed, suppose on the contradiction that Assumption \ref{ass0} does not hold at $\overline{x}$. Then, there exists a sequence $\{x^k\}\subset\mathcal{S}^*$ with $x^k\to \overline{x}$ such that $F(x^k)>F(\overline{x})$ for each $k$. Since $F$ has the KL property at $\overline{x}$, there exist $\delta>0,\varpi\in(0,+\infty]$ and $\varphi\in\Upsilon_{\!\varpi}$ such that for all $z\in\mathbb{B}(\overline{x},\delta)\cap[F(\overline{x})<F<F(\overline{x})+\varpi]$,
	$\varphi'(F(z)-F(\overline{x})){\rm dist}(0,\partial F(z))\ge 1$. By Assumption \ref{ass-1} (i)-(ii), $F$ is continuous at $\overline{x}$ relative to $\mathcal{S}^*$, so there exists $\delta'\in(0,\delta)$ such that for all $z\in\mathbb{B}(\overline{x},\delta')\cap\mathcal{S}^*$, 	$F(z)<F(\overline{x})+\varpi/2$. Clearly, for all sufficiently large $k$, $x^k\in\mathbb{B}(\overline{x},\delta')\cap[F(\overline{x})<F<F(\overline{x})+\varpi]$. Then, for all $k$ large enough,
	\[
	\varphi'(F(x^k)-F(\overline{x})){\rm dist}(0,\partial F(x^k))\ge 1, 
	\]
	which is impossible because ${\rm dist}(0,\partial F(x^k))=0$ is implied by $x^k\in\mathcal{S}^*$.
\end{remark}

The following proposition improves greatly the results of \cite[Propositions 3.1 \& 3.2]{LiuPan19}, an unpublished paper. Since its proof is a little long, we put it in Appendix \ref{secA2}.
\begin{proposition}\label{KL-subregular}
	Consider any $\overline{x}\in\mathcal{S}^*$ and $q>0$. The following assertions hold.
	\begin{description}
		\item [(i)] Under Assumption \ref{ass0}, the $q$-subregularity of the mapping $\partial F$
		at $\overline{x}$ for $0$ implies that the function $F$ has the KL property of exponent $\max\{\frac{1}{2q},\frac{1}{1+q}\}$ at $\overline{x}$.
		
		\item[(ii)] If $F$ has the KL property of exponent $\frac{1}{2q}$ for 
		$q\in(1/2,1]$ at $\overline{x}$	and $\overline{x}$ is a local minimizer of \eqref{prob},
		then $\partial F$ is $(2q\!-\!1)$-subregular at $\overline{x}$ for $0$.
	\end{description}
\end{proposition}
\begin{remark}\label{remark-qsubregular}
	{\bf(a)} The local optimality of $\overline{x}$ in Proposition \ref{KL-subregular} (ii) is sufficient but not necessary. For example, consider problem \eqref{prob} with  $f(x)=\psi(x)=\frac{1}{2}x_1^2+\frac{1}{4}x_2^4-\frac{1}{2}x_2^2$ and $g\equiv 0$ (see \cite[Section 1.2.3]{Nesterov04}). One can verify that $\mathcal{S}^*\!=\{x^{1,*},x^{2,*},x^{3,*}\}$ with $x^{1,*}=(0,0)^{\top},x^{2,*}=(0,-1)^{\top}$ and $x^{3,*}=(0,1)^{\top}$. Since the set  $\mathcal{S}^*$ is finite, Assumption \ref{ass0} holds at each $x^{i,*}$, and $\partial F=\nabla\!f$ is subregular at each $(x^{i,*},0)$. By Proposition \ref{KL-subregular}, $F$ has the KL property of exponent $\frac{1}{2}$ at $x^{1,*}$, but it is not a local minimizer of $F$.
	
	\noindent
	{\bf(b)} Under the assumption of Proposition \ref{KL-subregular} (ii), $F$ admits the growth at $\overline{x}$ as in \eqref{Fgrowth}. By using the example in part (a), one can verify that such a growth does not necessarily hold if the local optimality assumption on $\overline{x}$ is replaced by Assumption \ref{ass0}. The growth of $F$ in \eqref{Fgrowth} has the same order as the one obtained in \cite{Mordu15} under the $q$-subregularity of $\partial F$ with modulus $\kappa$ at $\overline{x}$ for $0$ and a lower calm-type assumption on $F$ at $\overline{x}$. 
	
	\noindent
	{\bf(c)} When $F$ is locally strong convex in a neighborhood of $\overline{x}\in\mathcal{S}^*$, there exist $\delta>0$ and $\widehat{c}>0$ such that for all $x\in\mathbb{B}(\overline{x},\delta)$, $F(x)-F(\overline{x})\ge\frac{\widehat{c}}{2}\|x-\overline{x}\|^2$, which by \cite[Theorem 5 (ii)]{Bolte17} means that $F$ has the KL property of exponent $1/2$ at $\overline{x}$. In fact, in this case, ${\rm dist}(0,\partial F(x))\ge\sqrt{\widehat{c}(F(x)\!-\!F(\overline{x}))}$ holds for all $x\in\mathbb{B}(\overline{x},\delta)$, which by Proposition \ref{KL-subregular} (ii) means that the mapping $\partial F$ is also subregular at $\overline{x}$ for $0$.
\end{remark}

To close this section, we present a condition to ensure that the local Lipschitz error bound on $\mathcal{X}^*$ holds, which will be used in the convergence analysis of section \ref{sec4.2}.
\begin{lemma}\label{alemma-ebound}
	Fix any $\overline{x}\in\mathcal{X}^*$. If $F$ has the KL property of exponent $1/2$ at $\overline{x}$, and there are $\delta>0,\alpha>0$ such that for all $x\in\mathbb{B}(\overline{x},\delta)$, $F(x)\ge F(\overline{x})+(\alpha/2)[{\rm dist}(x,\mathcal{X}^*)]^2$, then there exist $\delta'>0$ and $\kappa'>0$ such that ${\rm dist}(x,\mathcal{X}^*)\le \kappa' r(x)$ for all $x\in\mathbb{B}(\overline{x},\delta')\cap{\rm dom}\,g$.
\end{lemma}
\begin{proof}
	Since $F$ has the KL property of exponent $1/2$ at $\overline{x}$, there exist $\varepsilon'>0,\varpi>0$ and $c>0$ such that for all $z\in\mathbb{B}(\overline{x},\varepsilon')\cap[F(\overline{x})<F<F(\overline{x})+\varpi]$,
	\begin{equation}\label{KL1/2}
		{\rm dist}(0,\partial F(z))\ge2/c\big(F(z)-F(\overline{x})\big)^{1/2}.
	\end{equation}
	Since $F$ is continuous relative to ${\rm dom}\,g$, for all $x\in\mathbb{B}(\overline{x},\delta)\cap{\rm dom}\,g$ (if necessary by shrinking $\delta$), $F(x)\le F(\overline{x})+\varpi/2$.  
	Let $\varepsilon=\min\{\varepsilon',\delta\}$. Then, we claim that 
	\begin{equation}\label{ebound0-appendix}
		{\rm dist}(0,\partial F(x))\ge(\sqrt{2\alpha}/c){\rm dist}(x,\mathcal{X}^*)
		\quad{\rm for\ all}\ x\in\mathbb{B}(\overline{x},\varepsilon)\cap{\rm dom}\,g. 
	\end{equation}
	Pick any $x\in\mathbb{B}(\overline{x},\varepsilon)\cap{\rm dom}\,g$. If $F(x)=F(\overline{x})$, from the given quadratic growth condition, we deduce that $x\in\mathcal{X}^*$, and the claimed inequality holds for any $\kappa'>0$, so it suffices to consider that $F(x)\ne F(\overline{x})$. This along with the quadratic growth condition means that $F(x)>F(\overline{x})$. Together with $F(x)\le F(\overline{x})+\varpi/2$, we have $x\in[F(\overline{x})<F<F(\overline{x})+\varpi]\cap\mathbb{B}(\overline{x},\varepsilon')$. Then, from \eqref{KL1/2}, ${\rm dist}(0,\partial F(x))\ge(\sqrt{2\alpha}/c){\rm dist}(x,\mathcal{X}^*)$. 
	Thus, the claimed fact in \eqref{ebound0-appendix} holds. Let $\delta'=\min\{\varepsilon,\widetilde{\varepsilon}\}/(1\!+\!L')$ where $\widetilde{\varepsilon}$ and $L'$ are the same as in the proof of Lemma \ref{sregular-relation}. Now fix any $x\in\mathbb{B}(\overline{x},\delta')\cap{\rm dom}\,g$ and write $u:=R(x)=x-\mathcal{P}_g(x-\nabla\!f(x))$. Then, by noting that $x-u\in{\rm dom}\,g$ and following the same arguments as those for the first part of Lemma \ref{sregular-relation} with $q=1$, we obtain that ${\rm dist}(x,\mathcal{X}^*)\le \kappa'r(x)$ with $\kappa'=[1+(\sqrt{2\alpha}/c)(1+L')]$. The proof is completed.
\end{proof}
\section{Inexact regularized proximal Newton method}\label{sec3}

Now we describe our inexact regularized proximal Newton method (IRPNM) for solving \eqref{prob}. Let $x^k$ be the current iterate. As mentioned previously, we adopt $G_k$ defined in \eqref{def-G} to construct the quadratic approximation model \eqref{subprobx} at $x^k$. When $\psi$ is convex, this quadratic model reduces to the one used in \cite{Yue19,Mordu20}. Since $G_k$ is positive definite, the objective function of \eqref{subprobx} is strongly convex, so it has a unique minimizer, denoted by $\overline{x}^k$. For model \eqref{subprobx}, one may use the coordinate gradient descent method \cite{Tseng09} as in \cite{Yue19} or the line search PG method as in \cite{Kanzow21} to seek an approximate minimizer $y^k$. Inspired by the structure of $G_k$, we develop in Section \ref{sec5} a dual semismooth Newton augmented Lagrangian (SNALM) method to seek an approximate minimizer $y^k$.

To ensure that our IRPNM has a desirable convergence, we require $y^k$ to satisfy
\begin{subnumcases}{}\label{inexact-conda}
	r_k(y^k)\le\eta\min\big\{r(x^k),[r(x^k)]^{1+\tau}\big\}\ \ {\rm and}\ \
	\Theta_k(y^k)\le\Theta_k(x^k)\ \ {\rm if}\ \varrho\in(0,1);\\
	\label{inexact-condb}
	{\rm dist}(0,\partial \Theta_k(y^k))\le\eta r(x^k)
	\ \ {\rm and}\ \ \Theta_k(y^k)\le \Theta_k(x^k)\ \ {\rm if}\ \varrho=0,
\end{subnumcases}
where $\eta\in\!(0,1)$ and $\tau\!\ge\!\varrho$ are constants, 
and $r_k$ is the KKT residual of \eqref{subprobx} given by
\begin{equation}\label{rk-fun}
	r_k(y)\!:=\!\|R_k(y)\|\ \ {\rm with}\
	R_k(y)\!:=y\!-\!\mathcal{P}_g(y\!-\!\nabla\!f(x^k)\!-\!G_k(y\!-\!x^k))
	\ \ {\rm for}\ y\in\mathbb{R}^n.
\end{equation}
Though the following inequality holds for $r_k(\cdot)$ and ${\rm dist}(0,\partial\Theta_k(\cdot))$ by \cite[Lemma 4.1]{LiPong18}
\begin{equation}\label{rdist} 
	r_k(y)\le{\rm dist}(0,\partial\Theta_k(y))\quad{\rm for\ any}\ y\in{\rm dom}\,g,
\end{equation} 
there is no direct relation between the first inequality of \eqref{inexact-conda} and that of \eqref{inexact-condb}. Criterion \eqref{inexact-conda} is the same as the one used in \cite{Mordu20}, but is weaker than the one used in \cite{Yue19}. Indeed, let $\ell_k$ be the partial first-order approximation function of $F$ at $x^k$:
\begin{equation}\label{lk-fun}
	\ell_k(x):=f(x^k)+\langle\nabla\!f(x^k),x\!-\!x^k\rangle+g(x)\quad\ \forall x\in\mathbb{R}^n.
\end{equation}
One can verify that $\Theta_k(y^k)-\Theta_k(x^k)\le 0$
if $\Theta_k(y^k)-\Theta_k(x^k)\le\zeta(\ell_k(y^k)-\ell_k(x^k))$
for some $\zeta\in(0,1)$ by using the positive definiteness of $G_k$
and the following relation
\begin{equation}\label{Thetak-ellk}
	\Theta_k(x)-\Theta_k(x^k)=\ell_k(x)-\ell_k(x^k)+\frac{1}{2}(x-x^k)^{\top}G_k(x-x^k)\quad\ \forall x\in\mathbb{R}^n.
\end{equation} 
As will be shown in Lemma \ref{yk-define} below, the vector $y^k$ satisfying
\eqref{inexact-condb} is actually an exact minimizer of the canonical perturbation of \eqref{subprobx}. It seems that the distance involved in \eqref{inexact-condb} is difficult to compute in practice, but when choosing SNALM as the inner solver, one can easily achieve an element $\omega^k\in\partial\Theta_k(y^k)$ (see Section \ref{sec5.1.1}) and then an upper bound $\|\omega^k\|$ for ${\rm dist}(0,\partial \Theta_k(y^k))$. Thus, the first inequality of criterion \eqref{inexact-condb} is guaranteed to hold by requiring $\|\omega^k\|\le\eta r(x^k)$. It is worth pointing out that one can replace the first inequality in \eqref{inexact-condb} with $r_k(y^k)\le\eta r(x^k)$, but such $y^k$ become more inaccurate by \eqref{rdist}, which will lead to more iterations and running time.

With an inexact minimizer $y^k$ of subproblem \eqref{subprobx}, we perform the Armijo line search along the direction $d^k:=y^k-x^k$ to capture a step-size $\alpha_k>0$ so that the objective value of problem \eqref{prob} can gain a sufficient decrease. The algorithm steps into the next iterate with $x^{k+1}\!:=x^k\!+\!\alpha_kd^k$ if $F(x^k\!+\!\alpha_kd^k)<F(y^k)$, otherwise with $x^{k+1}\!:=y^k$. Now we are ready to summarize the iterate steps of our IRPNM.
\begin{algorithm}[h]
	\caption{\label{IRPNM}{\bf (Inexact regularized proximal Newton method)}}
	\textbf{Input:}\ $\epsilon_0>0$, $a_1\ge 1,a_2>0,\tau\ge\varrho\in[0,1),  	\eta,\beta\in(0,1),\sigma\in(0,1/2)$ and $x^0\in{\rm dom}\,g$.
	
	\medskip
	\noindent
	\textbf{For} $k=0,1,2,\ldots$ \textbf{do}
	\begin{enumerate}
		\item Compute the residual $r(x^k)$. If $r(x^k)\le\epsilon_0$, then stop; else go to Step 2.
		
		\item Seek an inexact minimizer $y^k$ of \eqref{subprobx} with $G_k$ in \eqref{def-G} so that it satisfies \eqref{inexact-conda} or \eqref{inexact-condb}.
		
		\item \label{stop}Set $d^k:=y^k\!-\!x^k$. If $\|d^k\|\le \epsilon_0$, then stop; 
		else go to Step 4.
		
		\item \label{ls-step} Seek the smallest $m_k$ among all nonnegative integers $m$ such that
		\begin{equation}\label{lsearch}
			F(x^k)-F(x^k\!+\!\beta^md^k)\ge\sigma\beta^m\mu_k\|d^k\|^2
			\ \ {\rm with}\ \mu_k:=a_2[r(x^k)]^{\varrho}.
		\end{equation}
		Set $\alpha_k:=\beta^{m_k}$ and
		\begin{equation}\label{new-iter}
			x^{k+1}:=\left\{\begin{array}{cl}
				y^k &{\rm if}\ F(y^k)<F(x^k\!+\!\alpha_kd^k),\\
				x^k\!+\!\alpha_kd^k &{\rm  otherwise}.
			\end{array}\right.
		\end{equation}
	\end{enumerate}
	\textbf{end (for)}
\end{algorithm}
\begin{remark}\label{remark-IRPNM}
	{\bf(a)} Differently from the proximal Newton methods in \cite{Yue19,Mordu20}, the next iterate $x^{k+1}$ of Algorithm \ref{IRPNM} may take $x^{k}+\alpha_kd^k$ or $y^k$, determined by their objective values. Note that a standard abstract convergence scheme adopted in the KL framework \cite{Attouch13} usually requires a relative error condition at the iterates, while such a selection allows us to employ the relative error condition at $y^k$ (which might not be the next iterate) in the proof of convergence, which is crucial to achieve the global convergence of the iterate sequence $\{x^k\}_{k\in\mathbb{N}}$ for $\varrho=0$ under the KL property of $F$ (see Theorem \ref{gconverge-KL}). To the best of our knowledge, such a technique first appears in \cite[Algorithm 1]{Bonettini17}.
	
	\noindent
	{\bf(b)} The line search criterion in \cite[Eq (7)]{Yue19} implies the one in equation \eqref{lsearch}. Indeed, equality \eqref{Thetak-ellk} and $\Theta_k(y^k)\le \Theta_k(x^k)$ in criterion \eqref{inexact-conda} or \eqref{inexact-condb} implies that
	\begin{equation}\label{ineq-lk}
		\ell_k(x^k)-\ell_k(y^k)\ge\frac{1}{2}\langle y^k-x^k,G_k(y^k\!-\!x^k)\rangle
		\ge\frac{1}{2}\mu_k\|y^k\!-\!x^k\|^2.
	\end{equation}
	Since $g$ is convex and $d^k=y^k\!-\!x^k$, for any $\tau\in[0,1]$ we have $x^k\!+\!\tau d^k\in{\rm dom}\,g$ and $\ell_k(x^k)-\ell_k(x^k\!+\!\tau d^k)\ge\tau[\ell_k(x^k)-\ell_k(x^k\!+\!d^k)]	=\tau[\ell_k(x^k)-\ell_k(y^k)]$, which by \eqref{ineq-lk} implies that $\ell_k(x^k)-\ell_k(x^k\!+\!\tau d^k)\ge\frac{1}{2}\tau \mu_k\|y^k\!-\!x^k\|^2$. Thus \eqref{lsearch} holds. This implication suggests that \eqref{lsearch} may need less computation time than the one in \cite{Yue19}.
	
	\noindent
	{\bf(c)} From \eqref{inexact-conda} or \eqref{inexact-condb} along with \eqref{rdist}, for each $k\in\mathbb{N}$, $r_k(y^k)\le\eta r(x^k)$, and then
	\begin{equation}\label{rk-ineq}
		(1\!-\!\eta)r(x^k)\le r(x^k)\!-\!r_k(y^k)\le 2\|d^k\|\!+\!\|G_k\|\|d^k\|,
	\end{equation}
	where the last inequality is using the expressions of $r(x^k)$ and $r_k(y^k)$. By the boundedness of $\{x^k\}_{k\in\mathbb{N}}$ (see Proposition \ref{prop-xk} below) and the expression of $G_k$ in \eqref{def-G}, we have $r(x^k)\le c_0\|d^k\|$ for some $c_0>0$. Thus, $d^k=0$ implies that $x^k$ is a stationary point of \eqref{prob}. This interprets why Algorithm \ref{IRPNM} also adopts $\|d^k\|\le\epsilon_0$ as a termination condition.
	
	\noindent
	{\bf(d)} When $H_k$ and $\eta_k$ in (S.1) of \cite[Algorithm 3.1]{Kanzow21} take $G_k$ and $\eta\min\big\{1,[r(x^k)]^{\tau}\big\}$, respectively, under the unit step-size and the existence of $\widehat{\kappa}>0$ and $\widetilde{k}_0\in\mathbb{N}$ such that  
	\begin{equation}\label{dk-rk}
		\|d^k\|\le\widehat{\kappa}[r(x^k)]^{\widehat{q}}\ \ (\widehat{q}>0)\quad{\rm for\ all}\ k\ge\widetilde{k}_0,
	\end{equation}
	the sequence $\{x^k\}_{k\ge\widetilde{k}_0}$ generated by \cite[Algorithm 3.1]{Kanzow21} is the same as the one yielded by Algorithm \ref{IRPNM} with $\varrho\in(0,1)$. Indeed, using the same arguments as those for Lemma \ref{yk-define} below, one can show that the inexactness criterion in \cite[Eq\,(13)]{Kanzow21} is well defined. In addition, by using equation \eqref{Thetak-ellk}, $G_k\succeq a_2[r(x^k)]^{\varrho}I$ and condition \eqref{dk-rk}, it follows that 
	\begin{align*}
		0\ge\Theta_k(y^k)-\Theta_k(x^k)
		&\ge\ell_k(y^k)-\ell_k(x^k) +(a_2/2)[r(x^k)]^{\varrho}\|d^k\|^2\\
		&\ge \ell_k(y^k)-\ell_k(x^k)+(a_2/2)\widehat{\kappa}^{-\frac{\varrho}{\widehat{q}}}\|d^k\|^{2+\frac{\varrho}{\widehat{q}}}\quad{\rm for\ all}\ k\ge\widetilde{k}_0.
	\end{align*}
	This implies that \cite[condition (14)]{Kanzow21} is satisfied with $\rho=\frac{1}{2}a_2\widehat{\kappa}^{-\frac{\varrho}{\widehat{q}}}$ and $p=2+{\varrho}/{\widehat{q}}$ when $k\ge\widetilde{k}_0$, so the iterates generated by \cite[Algorithm 3.1]{Kanzow21} for $k\ge\widetilde{k}_0$ are the same as those yielded by Algorithm \ref{IRPNM} with $\varrho\in(0,1)$.  
    This fact, along with the convergence analysis in Section \ref{sec4.2}, shows that the global convergence and superlinear rate results of \cite[Algorithm 3.1]{Kanzow21} there can be achieved under weaker conditions (see Remark \ref{remark-converge}).
	
	\noindent
	 {\bf(e)} A stepsize $t>0$ is generally introduced into the definition of KKT residual of \eqref{prob} as $R_{t}(x)\!:=t^{-1}[x\!-\!\mathcal{P}_{tg}(x\!-\!t\nabla\!f(x))]$, and similarly a stepsize $t_k>0$ is done for the KKT residual of \eqref{subprobx} as $R_{t_k}(x)\!:=t_k^{-1}[x\!-\!\mathcal{P}_{t_kg}(x\!-t_k(\nabla\!f(x^k)+G_k(x-\!x^k)))]$. In practice, one can search for $t_k$ via backtracking by the descent lemma (see Lemma \ref{lemma-descent}). Concretely, with an initial lower estimate $L\!>0$ for the Lipschitz modulus of $\nabla\!f$ at $x^k$ and a ratio $\alpha>1$, the following if-end sentence can be added in step 1 before calculating $r(x^k)$:
		
		{\bf If} $f(z)\!>\!f(x^k)+\langle\nabla\!f(x^k),z-\!x^k\rangle\!+\!\frac{L}{2}\|z\!-\!x^k\|^2$ for $z\!=\mathcal{P}_{L^{-1}g}(x^k\!-\!L^{-1}\nabla\!f(x^k))$, {\bf then}
		
		\quad $L\leftarrow\alpha L$,
		
		{\bf end if} \\
		and set $t_k:=1/L$. After testing Algorithm \ref{IRPNM} with such KKT residuals, we found that it cannot improve the performance of Algorithm \ref{IRPNM}, and even requires more running time for some test examples. This phenomenon is reasonable by noting that our approximation matrix $G_k$ actually also plays the role of variable metric. Taking into account this, we simply take the unit stepsize for the functions $r$ and $r_k$ throughout this paper.
	
\end{remark}

Before analyzing the convergence of Algorithm \ref{IRPNM}, we need to verify that it is well defined, i.e., arguing that the inexactness conditions in \eqref{inexact-conda} and \eqref{inexact-condb} are feasible and the line search in \eqref{lsearch} will terminate in a finite number of steps.
\begin{lemma}\label{yk-define}
	For each iterate of Algorithm \ref{IRPNM} with $\epsilon_0=0$, the criterion \eqref{inexact-conda} is satisfied by any $y\in{\rm dom}\,g$ sufficiently close to the exact solution $\overline{x}^k$ of \eqref{subprobx}, the criterion \eqref{inexact-condb} is satisfied by $y_u:=\mathcal{P}_g(y\!-\!\nabla f(x^k)-G_k(y\!-\!x^k))$ for any $y$ sufficiently close to $\overline{x}^k$, and there exists an integer $m_k\ge 0$ such that the descent condition in \eqref{lsearch} is satisfied.
\end{lemma}
\begin{proof}
	Consider the iterate $x^k$. Assume that $x^k$ is not a stationary point of \eqref{prob}, i.e., $r_k(x^k)>0$. From the continuity of $r_k$ and $r_k(\overline{x}^k)=0$, for any $y$ sufficiently close to $\overline{x}^k$, $r_k(y)\le\eta\min\{r(x^k),[r(x^k)]^{1+\tau}\}$. In addition, since $\overline{x}^k$ is the unique optimal solution of \eqref{subprobx}, $G_k(x^k\!-\!\overline{x}^k)\in\partial\ell_k(\overline{x}^k)$.
	This implies that $x^k\ne\overline{x}^k$ (if not, $0\in\partial\ell_k(x^k)$ 
	and $x^k$ is a stationary point of \eqref{prob}). By the convexity of $\ell_k$,
	$\ell_k(x^k)\ge\ell_k(\overline{x}^k) +\langle G_k(x^k\!-\!\overline{x}^k),x^k\!-\!\overline{x}^k\rangle$, which along with \eqref{Thetak-ellk} implies that
	\begin{align*}
		\Theta_k(\overline{x}^k)-\Theta_k(x^k)&=\ell_k(\overline{x}^k)-\ell_k(x^k)
		+(1/2)\langle x^k\!-\!\overline{x}^k,G_k(x^k\!-\!\overline{x}^k)\rangle\\
		&\le -(1/2)\langle x^k\!-\!\overline{x}^k,G_k(x^k\!-\!\overline{x}^k)\rangle\le
		-(1/2)\mu_k\|x^k\!-\!\overline{x}^k\|^2<0.
	\end{align*}
	Since $\Theta_k$ is continuous relative to ${\rm dom}\,g$ by Assumption \ref{ass-1} (ii), the last inequality means that for any $y\in{\rm dom}\,g$ sufficiently close to $\overline{x}^k$, $\Theta_k(y)\le \Theta_k(x^k)$. The two sides show that criterion \eqref{inexact-conda} is satisfied by any $y\in{\rm dom}\,g$ sufficiently close to $\overline{x}^k$.
	
	For the first condition in \eqref{inexact-condb}, consider problems \eqref{prob1} and \eqref{u-prob} with $\vartheta$ given by  
	\[
	\vartheta(x)\!:=f(x^k)+\langle\nabla\!f(x^k),x-x^k\rangle+\frac{1}{2}(x-x^k)^{\top}G_k(x-x^k)
	\ \ {\rm for}\ x\in\mathbb{R}^n.
	\]
    Let $R_{\vartheta}(x)\!=x-\!\mathcal{P}_g(x-\!\nabla\vartheta(x))$ be as in Proposition \ref{prop-stationary}.
	Since $\overline{x}^k\!=\mathcal{P}_g(\overline{x}^k\!-\!\nabla\vartheta(\overline{x}^k))$,
	by the continuity of $\nabla\vartheta$, for any $y$ sufficiently close to $\overline{x}^k$, 
	\[
	\|R_{\vartheta}(y)+\!\nabla\vartheta(\mathcal{P}_g(y-\!\nabla\vartheta(y)))
	-\!\nabla\vartheta(y)\|\le\eta\min\big\{r(x^k),[r(x^k)]^{1+\tau}\big\}.
	\]
	With such $y$, by letting $y_{u}=\mathcal{P}_g(y\!-\!\nabla\vartheta(y))$ and using Proposition \ref{prop-stationary}, it follows that
	${\rm dist}(0,\partial\Theta_k(y_{u}))\!\le\!\eta\min\{r(x^k),[r(x^k)]^{1+\tau}\}$. 
	In addition, from the above discussions, the second condition in \eqref{inexact-condb} hold for any $z\in{\rm dom}\,g$ sufficiently close to $\overline{x}^k$. Note that $y_{u}\in{\rm dom}\,g$ and $y_u$ is close to $\overline{x}^k$ as $y$ is sufficiently close to $\overline{x}^k$. These two sides demonstrate that $y_u$ satisfies the criterion \eqref{inexact-condb} for any $y$ sufficiently close to $\overline{x}^k$.
	
	For the last part, we only need to consider that $d^k\ne 0$. From the convexity of $g$ and the definition of directional derivative, it follows that
	\begin{align*}
		F'(x^k;d^k)&=\langle\nabla\!f(x^k),d^k\rangle+g'(x^k;d^k)
		\le\langle\nabla\!f(x^k),d^k\rangle+g(y^k)-g(x^k)\\
		&=\ell_k(y^k)-\ell_k(x^k)\le-({\mu_k}/{2})\|y^k\!-\!x^k\|^2
		=-({\mu_k}/{2})\|d^k\|^2<0,
	\end{align*}
	where the first inequality is using \cite[Theorem 23.1]{Roc70} and $d^k=y^k-x^k$, and the second inequality is using \eqref{ineq-lk}. Together with the definition
	of directional derivative,
	\[
	F(x^k\!+\!t d^k)-F(x^k)\le-t\big[({\mu_k}/{2})\|d^k\|^2-o(t)\big].
	\]
	This implies that the line search step in \eqref{lsearch} is well defined. 
\end{proof}

To close this section, we summarize some properties of the sequence $\{x^k\}_{k\in\mathbb{N}}$ generated by Algorithm \ref{IRPNM} under the following assumption on a level set of $F$.
	\begin{assumption}\label{ass1}
		The level set $\mathcal{L}_{F}(x^0)\!:=\!\{x\in{\rm dom}\,g\ |\ F(x)\le F(x^0)\}$ is bounded.
\end{assumption}

\begin{proposition}\label{prop-xk}
	Let $\{x^k\}_{k\in\mathbb{N}}$ be a sequence generated by Algorithm \ref{IRPNM} with $\epsilon_0=0$, and denote by $\omega(x^0)$ its cluster point set.  Then, under Assumption \ref{ass1},
	\begin{description}
		\item [(i)] the sequence $\{F(x^k)\}_{k\in\mathbb{N}}$ is nonincreasing and convergent;
		
		\item [(ii)] the sequence $\{x^k\}_{k\in\mathbb{N}}$ is bounded;  
		
		\item[(iii)] $\lim_{k\to\infty}r(x^k)=0$ and $\lim_{k\to\infty}\|d^k\|=0$;
		
		\item [(iv)] $\omega(x^0)\subset\mathcal{S}^*$ is a nonempty compact set and  $F\equiv \overline{F}\!:={\displaystyle\lim_{k\to\infty}}F(x^k)$ on $\omega(x^0)$.		
	\end{description}
\end{proposition}
\begin{proof}
	{\bf(i)-(ii)} For each $k\in\mathbb{N}$, from \eqref{lsearch} and \eqref{new-iter}, we have $F(x^{k+1})<F(x^k)$ and then $\{x^k\}_{k\in\mathbb{N}}\subset\mathcal{L}_{F}(x^0)$. Along with the boundedness of $\mathcal{L}_{F}(x^0)$, we get parts (i) and (ii). 
	
	\noindent
	{\bf(iii)} For each $k\in\mathbb{N}$, let $u^k\!=x^k\!-\mathcal{P}_g(x^k\!-\nabla\!f(x^k))$. 
	Then $u^k\!-\!\nabla\!f(x^k)\in\partial g(x^k\!-\!u^k)$.
	By the expression of $R_k$, 
	$R_k(y^k)\!-\!\nabla\!f(x^k)\!-\!G_k(y^k\!-\!x^k)\in\partial g(y^k\!-\!R_k(y^k))$.
	Using the monotonicity of $\partial g$, we have 
	$\langle R_k(y^k)-G_kd^k-u^k,d^k-R_k(y^k)+u^k\rangle\ge0$ or 
	\[
	\langle d^k,G_kd^k\rangle
	\le\langle R_k(y^k)\!-\!u^k,d^k\!-\!R_k(y^k)+u^k\!+\!G_kd^k\rangle
	\le\langle(I+G_k)d^k,R_k(y^k)\!-\!u^k\rangle.
	\]
	From the expression of $G_k$ in \eqref{def-G} and $\mu_k=a_2[r(x^k)]^{\varrho}$, $G_k\succeq \mu_kI$, which together with the last inequality and the expression of $R_k(y^k)$ implies that  
	\[
	a_2r(x^k)^{\varrho}\|d^k\|^2\le(r_k(y^k)+r(x^k))(1\!+\!\|G_k\|)\|d^k\|
	\le(1\!+\!\eta)(1\!+\!\|G_k\|)r(x^k)\|d^k\|,  
	\]
	where the last inequality is using $r_k(y^k)\le\eta r(x^k)$ obtained in Remark \ref{remark-IRPNM} (c). Then,  
	\begin{equation}\label{dk-bound} 
		\|d^k\|\le a_2^{-1}(1+\eta)(1+\|G_k\|)r(x^k)^{1-\varrho}
		\quad{\rm for\ each}\ k\in\mathbb{N}.
	\end{equation}
	By part (ii) and the continuity of $r$, there exists 
	a constant $\widehat{\tau}>0$ such that $\|d^k\|\le\widehat{\tau}$ for all $k\in\mathbb{N}$. Let $\mathbb{B}$ denote the unit ball of $\mathbb{R}^n$ centered at the origin. 
	Assumption \ref{ass-1} (i) implies that $\nabla\!f$ is Lipschitz continuous on the compact set $\mathcal{L}_F(x^0)+\widehat{\tau}\mathbb{B}$ with Lipschitz constant, say, $L_{\nabla\!f}$. 
	Then, for any $x',x\in\mathcal{L}_F(x^0)+\widehat{\tau}\mathbb{B}$,  
	\begin{equation}\label{grad-Lip}
		\|\nabla\!f(x)-\nabla\!f(x')\|\le L_{\nabla\!f}\|x-x'\|.
	\end{equation} 
	Let $\mathcal{K}\!:=\{k\in\mathbb{N}\ |\ \alpha_k<1\}$. Fix any $k\in\mathcal{K}$.  Since \eqref{lsearch} is violated for the stepsize $t_k:=\alpha_k/\beta$, from the convexity of $g$ and the definition of $\ell_k$, it follows that
	\begin{align}\label{temp-dkineq0}
		\sigma\mu_k t_k\|d^k\|^2
		&>f(x^k)-f(x^k+t_kd^k)+g(x^k)-g(x^k+t_kd^k) \nonumber\\
		&\ge f(x^k)-f(x^k+t_kd^k)+t_k(g(x^k)-g(y^k)) \nonumber\\
		&= f(x^k)-f(x^k+t_kd^k)+t_k(\langle\nabla\!f(x^k),
		d^k\rangle+\ell_k(x^k)-\ell_k(y^k))\nonumber\\
		&= t_k\langle\nabla\!f(x^k)-\nabla\!f(\xi^k),d^k\rangle + t_k(\ell_k(x^k)-\ell_k(y^k)) \nonumber\\
		&\ge t_k\langle\nabla\!f(x^k)-\nabla\!f(\xi^k),d^k\rangle + (t_k\mu_k/2)\|d^k\|^2
	\end{align}
	for some $\xi^k\in(x^k,x^k+t_kd^k)$, where the last equality is due to the mean-value theorem, and the last inequality comes from \eqref{ineq-lk}. 
	Combining \eqref{temp-dkineq0} and \eqref{grad-Lip} leads to 
	\[  
	(1/2\!-\!\sigma)t_k\mu_k\|d^k\|^2\le t_k\|\nabla f(\xi^k)\!-\!\nabla f(x^k)\|\|d^k\|
	\le L_{\nabla\!f}t_k^2\|d^k\|^2,
	\]	
	which implies that $t_k\ge\frac{1-2\sigma}{2L_{\nabla\!f}}\mu_k$. By the arbitrariness of  $k\in\mathcal{K}$, $\alpha_k\ge\min\{1,\frac{1-2\sigma}{2L_{\nabla\!f}}\beta\mu_k\}$ for all $k\in\mathbb{N}$. While from \eqref{lsearch} and part (i), $\lim_{k\to\infty}\alpha_k\mu_k\|d^k\|^2=0$.
	Thus, $\lim_{k\to\infty}\min\big\{\mu_k,(1\!-\!2\sigma)\beta (2L_{\nabla\!f})^{-1}\mu_k^2\big\}\|d^k\|^2=0$.
	Recall that $\|d^k\|\ge\frac{1-\eta}{2+\|G_k\|}r(x^k)$ by \eqref{rk-ineq} and $\mu_k=a_2[r(x^k)]^{\varrho}$. The boundedness of $\{\|G_k\|\}$ then implies that $\lim_{k\to\infty}r(x^k)=0$. 
	Together with \eqref{dk-bound}, it follows that $\lim_{k\to\infty}\|d^k\|=0$. 
	
	\noindent
	{\bf(iv)} By part (ii), the set $\omega(x^0)$ is nonempty and compact. Pick any $\overline{x}\in\omega(x^0)$. There is an index set $\mathcal{K}\subset\mathbb{N}$ such that $\lim_{\mathcal{K}\ni k\to\infty}x^{k}=\overline{x}$.
	From part (iii) and the continuity of $r$, we have
	$\overline{x}\in\mathcal{S}^*$, and then $\omega(x^0)\subset\mathcal{S}^*$.   
	Note that $\{x^k\}_{k\in K}\subset{\rm dom}\,g$ and $F$ is continuous relative to ${\rm dom}\,g$ by Assumption \ref{ass-1} (i)-(ii). Then, $F(\overline{x})=\overline{F}$, which shows that 
	$F$ is constant on the set $\omega(x^0)$.
\end{proof}
\section{Convergence analysis of Algorithm \ref{IRPNM}}\label{sec4}
This section focuses on the asymptotic convergence behaviour of Algorithm \ref{IRPNM}. To this end, we assume that the sequence $\{x^k\}_{k\in\mathbb{N}}$ is generated by Algorithm \ref{IRPNM} with $\epsilon_0=0$.

\subsection{Convergence analysis for $\varrho=0$}\label{sec4.1}
First, by using the first condition in \eqref{inexact-condb}, we bound ${\rm dist}(0,\partial F(y^k))$ in terms of $\|d^k\|$. 
\begin{lemma}\label{relgap}
	Under Assumption \ref{ass1}, for each $k\in\mathbb{N}$, there exists
	$w^k\in\partial F(y^k)$ with $\|w^k\|\le\gamma_0\|d^k\|$ for $\gamma_0\!=L_{\nabla\!f}\!+\widetilde{\gamma}+{\eta(2\!+\!\widetilde{\gamma})}/{(1\!-\!\eta)}$, where $\widetilde{\gamma}=L_{\nabla\!f}+a_1c_{\psi}\|A\|^2+a_2$ with $c_{\psi}\!:=\max_{x\in\mathcal{L}_{F}(x^0)}[-\lambda_{\rm min}(\nabla^2\psi(Ax\!-\!b))]_{+}$, and $L_{\nabla\!f}$ is the constant appearing in \eqref{grad-Lip}.
\end{lemma}
\begin{proof}
	Fix any $k\in\mathbb{N}$. Since ${\rm dist}(0,\partial\Theta_k(y^k))\le\eta r(x^k)$ by \eqref{inexact-condb}, there exists $\xi^k\in\partial\Theta_k(y^k)$ with $\|\xi^k\|\le\eta r(x^k)$. Let $w^k\!:=\xi^k\!+\!\nabla\!f(y^k)\!-\!\nabla\!f(x^k)-G_k(y^k\!-\!x^k)$. From $\xi^k\in\partial\Theta_k(y^k)$, we have $w^k\in\partial F(y^k)$. By the expression of $w^k$ and \eqref{grad-Lip}, it holds that  
	\begin{align}\label{wk-ineq}
		\|w^k\|&\le\|G_k(x^k\!-\!y^k)\|+\|\xi^k+\nabla\!f(y^k)\!-\!\nabla\!f(x^k)\|\nonumber\\
		&\le \big[\|G_k\|+L_{\nabla\!f}\big]\|d^k\|+\eta r(x^k).
	\end{align}
	From the expression of $G_k$ for $\varrho=0$ and \eqref{grad-Lip}, it follows that $\|G_k\|\le \widetilde{\gamma}$. Together with \eqref{rk-ineq}, $r(x^k)\le (1-\eta)^{-1}(2+\widetilde{\gamma})\|d^k\|$. The desired result then follows by \eqref{wk-ineq}. 
\end{proof}

With Lemma \ref{relgap}, the following theorem establishes the convergence of the sequence $\{x^k\}_{k\in\mathbb{N}}$ under the KL assumption on $F$. Due to the selection scheme in \eqref{new-iter}, the analysis technique is a little different from the common one adopted in \cite{Attouch10,Attouch13}.  
\begin{theorem}\label{gconverge-KL}
	If $F$ is a KL function, then under Assumption \ref{ass1},  
	$\sum_{k=0}^{\infty}\|x^{k+1}\!-\!x^k\|<\infty$ and the sequence $\{x^k\}_{k\in\mathbb{N}}$ converges to a stationary point of \eqref{prob}.
\end{theorem}
\begin{proof}
	If there exists $\overline{k}_1\in\mathbb{N}$ such that
	$F(x^{\overline{k}_1})=F(x^{{\overline{k}_1}+1})$, we have $d^{\overline{k}_1}=0$ by step \ref{ls-step}, and Algorithm \ref{IRPNM} stops within a finite number of steps. In this case, $r(x^{\overline{k}_1})=0$ follows from \eqref{rk-ineq}, i.e., $x^{\overline{k}_1}$ is a stationary point of \eqref{prob}. Hence, it suffices to consider that $F(x^k)>F(x^{k+1})$ for all $k\in\mathbb{N}$. By invoking equation \eqref{new-iter}, for each $k\in\mathbb{N}$,
	\begin{equation}\label{relationPhi}
		F(y^k)-\overline{F}\ge F(x^{k+1})-\overline{F}>0,
	\end{equation}
	where $\overline{F}$ is the same as in Proposition \ref{prop-xk} (iv).
	By Proposition \ref{prop-xk} (iv), the set $\omega(x^0)$ is nonempty and compact, and $F\equiv \overline{F}$ on the set $\omega(x^0)$.
	Since $F$ is assumed to be a KL function, 
	by invoking \cite[Lemma 6]{Bolte14}, there exist $\varepsilon>0,\varpi>0$
	and $\varphi\in\Upsilon_{\!\varpi}$ such that for all
	$y\in[\overline{F}\!<F<\overline{F}+\varpi]\cap\mathfrak{B}(\omega(x^0),\varepsilon)$ 
	with $\mathfrak{B}(\omega(x^0),\varepsilon)
	\!:=\!\{y\in\mathbb{R}^n\ |\ {\rm dist}(y,\omega(x^0))\le\varepsilon\}$,
	\[
	\varphi'(F(y)-\overline{F}){\rm dist}(0,\partial F(y)) \geq 1.
	\]
	By Proposition \ref{prop-xk} (iii) and $d^k=y^k-x^k$, $\lim_{k\to\infty}\|y^k-x^k\|=0$. Together with $\lim_{k\to\infty}{\rm dist}(x^k,\omega(x^0))=0$, we have $\lim_{k\rightarrow\infty} {\rm dist}(y^k,\omega(x^0))=0$. Obviously, $\{y^k\}_{k\in\mathbb{N}}$ is bounded, which along with $\{y^k\}_{k\in\mathbb{N}}\subset{\rm dom}\,g$ and Assumption \ref{ass-1} (i)-(ii) implies that $\{F(y^k)\}_{k\in\mathbb{N}}$ is bounded. We claim that $\lim_{k\to\infty}F(y^k)=\overline{F}$. If not, by \eqref{relationPhi} there must exist an index set $\mathcal{K}\subset\mathbb{N}$ such that $\lim_{\mathcal{K}\ni k\to\infty}F(y^k)>\overline{F}$. Since $\{y^k\}_{k\in \mathcal{K}}$ is bounded, there exists an index set $\mathcal{K}_1\subset\mathcal{K}$ such that $\lim_{\mathcal{K}_1\ni k\to\infty}y^k=y^*$, which along with $\lim_{k\to\infty}\|d^k\|=0$ yields that $y^*\in\omega(x^0)$. Thus, from the continuity of $F$ relative to ${\rm dom}\,g$, $\lim_{\mathcal{K}\ni k\to\infty}F(y^k)=\lim_{\mathcal{K}_1\ni k\to\infty}F(y^k)=F(y^*)=\overline{F}$, a contradiction to $\lim_{\mathcal{K}\ni k\to\infty}F(y^k)>\overline{F}$. Thus, the claimed limit $\lim_{k\to\infty}F(y^k)=\overline{F}$ holds. Then, there exists $\overline{k}\in\mathbb{N}$ such that for all $k\ge\overline{k}$,
	$y^k\in\mathfrak{B}(\omega(x^0),\varepsilon)\cap[\overline{F}<F<\overline{F}+\varpi]$, and hence 
	$\varphi'\big(F(y^k)\!-\!\overline{F}\big){\rm dist}(0,\partial F(y^k))\geq 1$.
	By Lemma \ref{relgap}, for each $k\in\mathbb{N}$, there exists $w^k\in\partial F(y^{k})$ with $\|w^k\|\le\gamma_0\|d^k\|$. Consequently,
	for each $k\ge\widehat{k}:=\overline{k}+1$,
	\begin{equation}\label{ineq_KL}
		\varphi'\big(F(y^{k-1})-\overline{F}\big)\|w^{k-1}\|\ge 1.
	\end{equation}
	In addition, from the proof of Proposition \ref{prop-xk} (iii), 
	$\alpha_k\mu_k\ge a_2\min(1,\frac{(1-2\sigma)\beta a_2}{2L_{\nabla\!f}})\!:=\underline{\alpha}$ for each $k\in\mathbb{N}$, which along with \eqref{lsearch}-\eqref{new-iter} implies that for each $k\in\mathbb{N}$,
	\begin{equation}\label{descend-obj}
		F(x^k)-F(x^{k+1})\ge\sigma\underline{\alpha}\|d^k\|^2
		=\sigma\underline{\alpha}\|y^k-x^k\|^2.
	\end{equation}
	Fix any $k\ge\widehat{k}$.
	Since $\varphi'$ is nonincreasing on $(0,\varpi)$
	by the concavity of $\varphi$, combining \eqref{relationPhi} with \eqref{ineq_KL}
	and using $\|w^{k-1}\|\le\gamma_0\|d^{k-1}\|$ yields that
	\begin{equation}\label{equa1-later}
		\varphi'(F(x^k)\!-\!\overline{F})\ge\varphi'(F(y^{k-1})\!-\!\overline{F})
		\ge\frac{1}{\|w^{k-1}\|}\ge\frac{1}{\gamma_0\|y^{k-1}\!-\!x^{k-1}\|}.
	\end{equation}
	Together with the concavity of $\varphi$ and inequality \eqref{descend-obj}, it follows that
	\begin{align*}
		\Delta_{k,k+1}&:=\varphi(F(x^k)\!-\!\overline{F})-\varphi(F(x^{k+1})\!-\!\overline{F})
		\ge\varphi'(F(x^k)-\overline{F})(F(x^{k})\!-\!F(x^{k+1}))\\
		&\ge\frac{F(x^{k})\!-\!F(x^{k+1})}{\gamma_0\|y^{k-1}-x^{k-1}\|}
		\ge\frac{\sigma\underline{\alpha}\|y^k-x^k\|^2}{\gamma_0\|y^{k-1}-x^{k-1}\|}.
	\end{align*}
	Then, $\|y^k\!-\!x^k\|
	\le\sqrt{\frac{\gamma_0}{\sigma\underline{\alpha}}\Delta_{k,k+1}\|y^{k-1}\!-\!x^{k-1}\|}$.
	From $2\sqrt{ts}\le t+s$ for $t\ge 0,s\ge 0$, 
	\begin{equation*}
		2\|y^k\!-\!x^k\|
		\le \gamma_0(\sigma\underline{\alpha})^{-1}\Delta_{k,k+1}+\|y^{k-1}\!-\!x^{k-1}\|.
	\end{equation*}
	By summing this inequality from $k$ to $l>k$, it is immediate to obtain that
	\begin{align*}
		2{\textstyle\sum_{i=k}^l}\|y^i-x^i\|
		&\le{\textstyle\sum_{i=k}^l}\|y^{i-1}-x^{i-1}\|
		+\gamma_0(\sigma\underline{\alpha})^{-1}{\textstyle\sum_{i=k}^l}\Delta_{i,i+1}\\
		&\le{\textstyle\sum_{i=k}^l}\|y^i-x^i\|+\|y^{k-1}-x^{k-1}\|
		+\gamma_0(\sigma\underline{\alpha})^{-1}\varphi\big(F(x^{k})-\overline{F}\big),
	\end{align*}
	where the second inequality is using the nonnegativity of $\varphi(F(x^{l+1})-\overline{F})$.
	Thus,
	\begin{equation}\label{temp-ratek}
		{\textstyle\sum_{i=k}^l}\|y^i-x^i\|\le\|y^{k-1}-x^{k-1}\|
		+\gamma_0(\sigma\underline{\alpha})^{-1}\varphi(F(x^{k})-\overline{F}).
	\end{equation}
	Passing the limit $l\to\infty$ to this inequality yields that
	$\sum_{i=k}^{\infty}\|y^i-x^i\|<\infty$.
	Note that $\|x^{i+1}-x^i\|\le\|y^i-x^i\|$ for each $i\in\mathbb{N}$.
	Then, $\sum_{i=0}^{\infty}\|x^{i+1}-x^i\|<\infty$. 
\end{proof}

Next we deduce the linear and sublinear convergence rates of the sequence $\{x^k\}_{k\in\mathbb{N}}$ under the KL property of the function $F$ with exponent $1/(2q)$ for $q\in(1/2,1]$.  
\begin{theorem}\label{R-linear}
	If $F$ is a KL function of exponent ${1}/{(2q)}$ with $q\in({1}/{2},1]$, then under Assumption \ref{ass1} the sequence $\{x^k\}_{k\in\mathbb{N}}$ converges to a point $\overline{x}\in\mathcal{S}^*$ and there exist constants $\gamma\in(0,1)$ and $c_1>0$ such that for all sufficiently large $k$,
	\[
	\|x^k-\overline{x}\|\le\left\{\begin{array}{cl}
		c_1\gamma^{k} &{\rm for}\ q=1,\\
		c_1 k^{\frac{2q-1}{2(q-1)}}&{\rm for}\ q\in({1}/{2},1).
	\end{array}\right.
	\]
\end{theorem}
\begin{proof}
	For each $k\in\mathbb{N}$, write $\Delta_k\!:=\sum_{i=k}^{\infty}\|y^i-x^i\|$. Fix any $k\ge\widehat{k}$ where $\widehat{k}$ is the same as in the proof of Theorem \ref{gconverge-KL}. From inequality \eqref{temp-ratek}, it follows that 
	\[
	\Delta_k\le \|y^{k-1}\!-\!x^{k-1}\|
	+\gamma_0(\sigma\underline{\alpha})^{-1}\varphi(F(x^{k})-\overline{F})
	\]
	where $\varphi(t)=ct^{\frac{2q-1}{2q}}\ (t>0)$ for some $c>0$. From the expression of $\varphi$ and \eqref{equa1-later},	$(F(x^{k})\!-\!\overline{F})^{\frac{1}{2q}}\le c(1\!-\!0.5/q)\gamma_0\|y^{k-1}\!-\!x^{k-1}\|$. Together with the last inequality, 
	\begin{align*}
		\Delta_k
		&\le \|y^{k-1}\!-\!x^{k-1}\|+\frac{c\gamma_0}{\sigma\underline{\alpha}}  
		\Big[\frac{c(2q\!-\!1)\gamma_0}{2q}\Big]^{2q-1}\|y^{k-1}\!-\!x^{k-1}\|^{2q-1}\\
		&\le\|y^{k-1}\!-\!x^{k-1}\|^{2q-1}+\frac{(c\gamma_0)^{2q}}{\sigma\underline{\alpha}}
		\Big[\frac{(2q\!-\!1)}{2q}\Big]^{2q-1}\|y^{k-1}\!-\!x^{k-1}\|^{2q-1}\\
		&=\gamma_1(\Delta_{k-1}\!-\!\Delta_k)^{2q-1}\ \ {\rm with}\ \ 
		\gamma_1=\Big[1+\frac{(c\gamma_0)^{2q}}{\sigma\underline{\alpha}}
		\Big(\frac{(2q\!-\!1)}{2q}\Big)^{2q-1}\Big],
	\end{align*}
	where the second inequality is due to $\lim_{k\to\infty}\|d^k\|=0$.
	When $q=1$, $\Delta_k\le\frac{\gamma_1}{1+\gamma_1}\Delta_{k-1}$.
	From this recursion formula, it follows that
	$\Delta_k\le(\frac{\gamma_1}{1+\gamma_1})^{k-\widehat{k}}\Delta_{\widehat{k}}$.
	Note that $\|x^k-\overline{x}\|\le\sum_{i=k}^{\infty}\|x^{i+1}-x^i\|
	\le \sum_{i=k}^{\infty}\|y^i-x^i\|=\Delta_k$.
	The conclusion holds with $\gamma={\gamma_1}/{(1+\!\gamma_1)}$ and
	$c_1=\Delta_{\widehat{k}}(\frac{\gamma_1}{1+\gamma_1})^{-\widehat{k}}$.
	When $q\in(1/2,1)$, from the last inequality, we have
	$\Delta_k^{\frac{1}{2q-1}}\le(\gamma_1)^{\frac{1}{2q-1}}(\Delta_{k-1}\!-\!\Delta_k)$ 
	for all $k\ge\widehat{k}$.
	By using this inequality and following the same analysis technique as
	those in \cite[Page 14]{Attouch09}, we obtain
	$\Delta_{k}\le c_1k^{\frac{2q-1}{2(q-1)}}$ for some $c_1>0$.
	Consequently, $\|x^k-\overline{x}\|\le c_1k^{\frac{2q-1}{2(q-1)}}$ 
	for all $k\ge\widehat{k}$. 
\end{proof}

\begin{remark}\label{Rlinear-obj}
		The linear and sublinear convergence rates of the objective value sequence $\{F(x^k)\}_{k\in\mathbb{N}}$ can also be obtained under the assumption of Theorem \ref{R-linear}. Indeed, let $\Delta_k=F(x^k)-\overline{F}$ for each $k\in\mathbb{N}$. Fix any $k\ge\widehat{k}$ where $\widehat{k}$ is the same as in the proof of Theorem \ref{gconverge-KL}.
		From inequality \eqref{equa1-later} with $\varphi(t)=ct^{(2q-1)/(2q)}$ and \eqref{descend-obj}, it follows that 
		\[
		\Delta_k^{\frac{1}{q}}\le\frac{[c\gamma_0(2q-1)]^2}{4q^2}\|d^{k-1}\|^2\le\gamma_1(\Delta_{k-1}-\Delta_k)\ \ {\rm with}\ \ 
		\gamma_1=\frac{[c\gamma_0(2q-1)]^2}{4q^2\underline{\alpha}\sigma}.
		\]
		When $q=1$, $\Delta_k\le\frac{\gamma_1}{1+\gamma_1}\Delta_{k-1}$ for all $k\ge\widehat{k}$. This implies that $\{F(x^k)\}_{k\in\mathbb{N}}$ converges to $\overline{F}$ with Q-linear rate.
		When $q\in(1/2,1)$, we have $\Delta_k^{1/q}\le \gamma_1(\Delta_{k-1}-\Delta_k)$ for all $k\ge\widehat{k}$. By using this inequality and following the same analysis technique as
		those in \cite[Page 14]{Attouch09}, there exists $c_1>0$ such that $F(x^k)-\overline{F}\le c_1k^{-q/(1-q)}$ for all $k\ge\widehat{k}$.
\end{remark}
\subsection{Convergence analysis for $0<\varrho<1$}\label{sec4.2}
To analyze the global convergence of the sequence
$\{x^k\}_{k\in\mathbb{N}}$ and its superlinear rate in this scenario, we need several technical lemmas. First, by noting that $G_k(x^k\!-\!\overline{x}^k)\!-\!\nabla\!f(x^k)\in \partial g(\overline{x}^k)$ and $R_k(y^k)\!-\!\nabla\!f(x^k)-G_k(y^k\!-\!x^k)\in\partial g(y^k\!-\!R_k(y^k))$ for each $k\in\mathbb{N}$, using the monotonicity of $\partial g$ and $G_k\succeq\mu_kI$ can bound the error between $y^k$ and $\overline{x}^k$.  
\begin{lemma}\label{lemma-yk}
	For each $k\in\!\mathbb{N}$, it holds that
	\(
	\|y^k\!-\overline{x}^k\|\!\le a_2^{-1}\eta (1+\|G_k\|)[r(x^k)]^{1+\tau-\varrho}.
	\)
\end{lemma}

The following lemma bounds the distance from $x^k$ to the exact minimizer 
$\overline{x}^k$ of \eqref{subprobx} by ${\rm dist}(x^k,\mathcal{S}^*)$,  
which extends the result of \cite[Lemma 4]{Yue19} to the nonconvex setting.
\begin{lemma}\label{xbark}
	Consider any $\overline{x}\in\omega(x^0)$. Suppose that Assumption \ref{ass1} holds, and that $\nabla^2\psi$ is strictly continuous at $A\overline{x}-b$ relative to $A({\rm dom}\,g)-b$. Then, there exist $\varepsilon_0>0$ and $L_{\psi}>0$ such that for all $x^k\in\mathbb{B}(\overline{x},{\varepsilon_0}/{2})$, $\|x^k\!-\!\overline{x}^k\|\le\big((0.5L_{\psi}/\mu_k)\|A\|^3{\rm dist}(x^k,\mathcal{S}^*)+(\Lambda_k/\mu_k)\|A\|^2+2\big){\rm dist}(x^k,\mathcal{S}^*)$, where $\Lambda_k\!:=a_1[-\lambda_{\rm min}(\nabla^2\psi(Ax^k\!-\!b))]_{+}$. 
\end{lemma}
\begin{proof}
	Since $\nabla^2\psi$ is strictly continuous at $A\overline{x}-b$ relative to
	$A({\rm dom}\,g)-b$, there exist $\delta_0>0$ and $L_{\psi}>0$ such that
	for any $z,z'\in\mathbb{B}(A\overline{x}-b,\delta_0)\cap[A({\rm dom}\,g)-b]$, 
	\begin{equation}\label{Hessian-Lip1}
		\|\nabla^2\psi(z)\!-\!\nabla^2\psi(z')\|\le L_{\psi}\|z-z'\|.
	\end{equation}  
	Take $\varepsilon_0={\delta_0}/{\|A\|}$. For any  $x,x'\in\mathbb{B}(\overline{x},\varepsilon_0)\cap {\rm dom}\,g$,
	we have $Ax\!-\!b,Ax'\!-\!b\in\mathbb{B}(A\overline{x}\!-\!b,\delta_0)\cap[A({\rm dom}\,g)\!-\!b]$, which together with $\nabla^2\!f(\cdot)=A^{\top}\nabla^2\psi(A\cdot\!-b)A$ implies that 
	\begin{equation}\label{Hessian-Lip}
		\|\nabla^2\!f(x)\!-\!\nabla^2\!f(x')\|\le L_{\psi}\|A\|^3\|x-x'\|.
	\end{equation}  
	Fix any $x^k\in\mathbb{B}(\overline{x},{\varepsilon_0}/{2})$.
	Pick any $x^{k,*}\!\in\Pi_{\mathcal{S}^*}(x^k)$. Since $\overline{x}\in\mathcal{S}^*$ 
	by Proposition \ref{prop-xk} (iv), we have 
	$\|x^{k,*}-\overline{x}\|\le 2\|x^k-\overline{x}\|\le\varepsilon_0$, so  
	$\|(1\!-\!t)x^k+\!t x^{k,*}\!-\overline{x}\|\le\varepsilon_0$ for all $t\in[0,1]$. Clearly, $(1-\!t)x^k\!+\!t x^{k,*}\in{\rm dom}\,g$ for all $t\in[0,1]$ by the convexity of ${\rm dom}\,g$. Thus, $(1-\!t)x^k\!+\!t x^{k,*}\in\mathbb{B}(\overline{x},\varepsilon_0)\cap{\rm dom}\,g$ for all $t\in[0,1]$. Note that $0\in\nabla\!f(x^{k,*})+\partial g(x^{k,*})$ and
	$0\in\nabla\!f(x^k)+G_k(\overline{x}^k\!-\!x^k)+\partial g(\overline{x}^k)$.
	From the monotonicity of $\partial g$, it follows that
	\[
	0\le\langle-\nabla\!f(x^{k,*})+\nabla\!f(x^k)+G_k(x^{k,*}\!-\!x^k),
	x^{k,*}\!-\!\overline{x}^k\rangle
	+\langle G_k(\overline{x}^k\!-\!x^{k,*}),x^{k,*}\!-\!\overline{x}^k\rangle.
	\]
	Together with $G_k\succeq \mu_kI$ and the triangle inequality, we obtain that
	\begin{align*}
		\|\overline{x}^k-x^{k,*}\|
		&\le \mu_k^{-1}\|\nabla\!f(x^k)-\nabla\!f(x^{k,*})+G_k(x^{k,*}\!-\!x^k)\|\nonumber\\
		&=\frac{1}{\mu_k}\Big\|\int_{0}^1\big[G_k\!-\!\nabla^2\!f(x^{k}\!+\!t(x^{k,*}\!-\!x^k))\big](x^{k,*}\!-\!x^k)dt\Big\|\nonumber\\
		&\le\mu_k^{-1}\big[(L_{\psi}\|A\|^3/2)\|x^{k,*}\!-\!x^k\|^2
		+\Lambda_k\|A\|^2\|x^{k,*}\!-\!x^k\|\big]+\|x^{k,*}\!-\!x^k\|,
	\end{align*}
	where the last inequality is using \eqref{Hessian-Lip} with $x=x^k$ and $x'=x^{k}\!+\!t(x^{k,*}\!-\!x^k)$. The conclusion then follows by using
	$\|\overline{x}^k\!-\!x^k\|\le\|\overline{x}^k\!-\!x^{k,*}\|+\|x^{k,*}\!-\!x^k\|$. 
\end{proof}  

The following lemma bounds $\Lambda_k$ defined in Lemma \ref{xbark} in terms of ${\rm dist}(x^k,\mathcal{X}^*)$. This result appeared early in \cite[Lemma 5.2]{Ueda10}, and we here provide a concise proof.
\begin{lemma}\label{lemma-Lamk}
	Consider any $\overline{x}\in\mathcal{X}^*$. Suppose that Assumption \ref{ass1} holds, and that $\nabla^2\psi$ is strictly continuous at $A\overline{x}-b$ relative to $A({\rm dom}\,g)-b$. Then,
	\[
	\Lambda_k\!\le\!a_1L_{\psi}\|A\|{\rm dist}(x^k,\mathcal{X}^*)\quad {\rm for\ all}\ 
	x^k\!\in\mathbb{B}(\overline{x},{\varepsilon_0}/{2})
	\]
	where $\varepsilon_0,\,L_{\psi}$ and $\Lambda_k$ are the same as those appearing in Lemma \ref{xbark}.
\end{lemma}
\begin{proof}
	Fix any $x^k\in\mathbb{B}(\overline{x},{\varepsilon_0}/{2})$. By the expression of $\Lambda_k$, it suffices to consider that $\lambda_{\rm min}(\nabla^2\psi(Ax^k\!-\!b))<0$. Pick any $x^{k,*}\in\Pi_{\mathcal{X}^*}(x^k)$. From $\overline{x}\in\mathcal{X}^*$, we have $\|x^{k,*}-\overline{x}\|\le 2\|x^k-\overline{x}\|\le\varepsilon_0$, and consequently $x^{k,*}\in\mathbb{B}(\overline{x},\varepsilon_0)\cap{\rm dom}\,g$. From $x^{k,*}\in\mathcal{X}^*$, we have $\nabla^2\psi(Ax^{k,*}\!-\!b)\succeq 0$. When $\lambda_{\rm min}(\nabla^2\psi(Ax^{k,*}\!-\!b))=0$, it holds that
	\begin{align*}
		\Lambda_k&=-a_1\lambda_{\rm min}(\nabla^2\psi(Ax^k\!-\!b))
		=a_1[\lambda_{\rm min}(\nabla^2\psi(Ax^{k,*}\!-\!b))-\lambda_{\rm min}(\nabla^2\psi(Ax^k\!-\!b))]\\
		&\le a_1\|\nabla^2\psi(Ax^{k,*}\!-\!b)\!-\!\nabla^2\psi(Ax^k\!-\!b)\|
		\le a_1L_{\psi}\|A\|\|x^k\!-\!x^{k,*}\|,
	\end{align*}
	where the first inequality is using the Lipschitz continuity of the function $\mathbb{S}^n\ni Z\mapsto\lambda_{\rm min}(Z)$ with modulus $1$, and the second one 
	is using \eqref{Hessian-Lip1} with $Ax^{k,*}-b,Ax^k-b\in\mathbb{B}(A\overline{x}-b,\delta_0)\cap[A({\rm dom}\,g)-b]$. So we only need to consider that $\lambda_{\rm min}(\nabla^2\psi(Ax^{k,*}\!-\!b))>0$. For this purpose, let $\phi_{k}(t):=\lambda_{\rm min}[\nabla^2\psi(Ax^k\!-\!b+\!tA(x^{k,*}\!-\!x^k))]$ for $t\ge 0$. Clearly, $\phi_{k}$ is continuous on any open interval containing $[0,1]$. Note that $\phi_k(0)<0$ and $\phi_k(1)>0$. There exists $\overline{t}_k\in(0,1)$ such that $\phi_k(\overline{t}_k)=0$. Consequently, 
	\begin{align*}
		\Lambda_k&=a_1\big[\lambda_{\rm min}(\nabla^2\psi(Ax^k\!-\!b+\overline{t}_kA(x^{k,*}\!-\!x^k)))
		-\lambda_{\rm min}(\nabla^2\psi(Ax^k\!-\!b))\big]\\
		&\le a_1\|\nabla^2\psi(Ax^k\!-\!b+\!\overline{t}_kA(x^{k,*}\!-\!x^k))\!-\!\nabla^2\psi(Ax^k\!-\!b)\|
		\le a_1L_{\psi}\|A\|\|x^{k,*}\!-\!x^k\|.
	\end{align*}
	This shows that the desired result holds. The proof is completed. 
\end{proof}
\begin{remark}\label{remark-Lamk}
	{\bf(a)} When $\mathcal{X}^*$ in Lemma \ref{lemma-Lamk} is replaced by $\mathcal{S}^*$, the conclusion may not hold. For example, consider problem \eqref{prob} with $g\equiv 0$ and $f$ given by Remark \ref{remark-qsubregular} (a), and $\overline{x}=(0,0)^{\top}\!\in\mathcal{S}^*$. For each $k>1$, let $x^k=(0,\frac{1}{k})^{\top}$. We have $\Lambda_k=a_1(1\!-\!\frac{3}{k^2})$ but ${\rm dist}(x^k,\mathcal{S}^*)\!=\frac{1}{k}$. Clearly, for all $k$ large enough, $\Lambda_k\le a_1L{\rm dist}(x^k,\mathcal{S}^*)$ does not hold. 
	
	\medskip	
	\noindent
	{\bf(b)} The result of Lemma \ref{lemma-Lamk} may not hold if ${\rm dist}(x^k,\mathcal{X}^*)$ is replaced by ${\rm dist}(x^k,\mathcal{S}^*)$. Indeed, consider $f(x)=-(x-2)^{4}$ and $g(x)=\left\{\begin{array}{cl}
		\infty & {\rm if}\ x<0,\\
		\!(x-2)^{4} & {\rm if}\ x\ge 0
	\end{array}\right.$ for $x\in\mathbb{R}$. 
	We have  $\mathcal{S}^*\!=\!\mathbb{R}_{+}$ and then $\mathcal{X}^*\!=\!\{2\}$. Let $\overline{x}=2$ and $x^k\!=2\!-\frac{1}{k}$. Then, $\Lambda_k\!=\frac{12a_1}{k^2}$, but ${\rm dist}(x^k,\mathcal{S}^*)\!=0$. Clearly, $\Lambda_k\le a_1L{\rm dist}(x^k,\mathcal{S}^*)$ does not hold for each $k\in\mathbb{N}$.  
\end{remark}  

Next we show that the unit step-size must occur when the iterates are close enough to a cluster point and the following locally H\"{o}lderian error bound on $\mathcal{X}^*$ holds. 
	\begin{assumption}\label{ass4}
		The locally H\"{o}lderian error bound of order $q>0$ at any $\overline{x}\in\omega(x^0)$ on $\mathcal{X}^*$ holds, i.e., for any $\overline{x}\in\omega(x^0)$, there exist $\varepsilon>0$ and $\kappa>0$ such that for all $x\in\mathbb{B}(\overline{x},\varepsilon)\cap{\rm dom}\,g$, ${\rm dist}(x,\mathcal{X}^*)\le\kappa[r(x)]^q$.
	\end{assumption}
	\begin{lemma}\label{step-size}
		Fix any $\overline{x}\in\omega(x^0)$. Suppose that Assumption \ref{ass1} holds, that $\nabla^2\psi$ is strictly continuous at $A\overline{x}\!-\!b$ relative to $A({\rm dom}\,g)\!-\!b$, and that  Assumption \ref{ass4} holds with $q>\varrho$. Then, there exists $\overline{k}\in\mathbb{N}$ such that $\alpha_k=1$ for all $x^k\in\mathbb{B}(\overline{x},\varepsilon_1)$ with $k\ge\overline{k}$ and 	$\varepsilon_1=\min(\varepsilon,{\varepsilon_0}/{2})$, where $\varepsilon$ and $\varepsilon_0$ are the same as in Assumption \ref{ass4} and Lemma \ref{xbark}.
\end{lemma}
\begin{proof}
	Since $\overline{x}\in\omega(x^0)$, there is $\mathcal{K}\subset\mathbb{N}$ such that $\lim_{\mathcal{K}\ni k\to\infty}x^k=\overline{x}$, which by Assumption \ref{ass4} and  $\{x^k\}_{k\in\mathbb{N}}\!\subset{\rm dom}\,g$ means that ${\rm dist}(x^k,\mathcal{X}^*)\le\kappa[r(x^k)]^q$ for all $k\in\mathcal{K}$ large enough, so $\overline{x}\in\mathcal{X}^*$ follows by passing $\mathcal{K}\ni k\to\infty$ and using Proposition \ref{prop-xk} (iii). Recall that $d^k=y^k-x^k$ for each $k\in\mathbb{N}$. By invoking Lemmas \ref{lemma-yk}-\ref{lemma-Lamk}, for all $x^k\in\mathbb{B}(\overline{x},\varepsilon_1)$,
	\begin{align*}
		\|d^k\|&\le\|y^k-\overline{x}^k\|+\|x^k-\overline{x}^k\|
		\le\frac{\eta(1\!+\!\|G_k\|)}{a_2}[r(x^k)]^{1+\tau-\varrho} \\
		&\quad\!+\!\big[0.5L_{\psi}\|A\|^3\mu_k^{-1}{\rm dist}(x^k,\mathcal{S}^*)
		\!+\!a_1L_{\psi}\|A\|^2\mu_k^{-1}{\rm dist}(x^k,\mathcal{X}^*)\!+\!2\big]{\rm dist}(x^k,\mathcal{S}^*).
	\end{align*} 
	By Proposition \ref{prop-xk} (iv), for each $k\in\mathbb{N}$, ${\rm dist}(x^k,\mathcal{S}^*)\le{\rm dist}(x^k,\omega(x^0))$, which along with $\lim_{k\to\infty}{\rm dist}(x^k,\omega(x^0))=0$ implies that for each $k\in\mathbb{N}$, $\Pi_{\mathcal{S}^*}(x^k)\subset \mathcal{L}_F(x^0)+\widehat{\tau}\mathbb{B}$ (if necessary by increasing $\widehat{\tau}$.) Thus, for each $k\in\mathbb{N}$, with any $x^{k,*}\in\Pi_{\mathcal{S}^*}(x^k)$, we have
	\begin{equation}\label{rk-ineq1}
		r(x^k)\!=\|x^k\!-\!\mathcal{P}_g(x^k\!-\!\nabla\!f(x^k))\!-\!x^{k,*}
		+\mathcal{P}_g(x^{k,*}\!-\!\nabla\!f(x^{k,*}))\|\!\le(2+L_{\nabla\!f}){\rm dist}(x^k,\mathcal{S}^*).
	\end{equation}
	From $\mathcal{X}^*\subset\mathcal{S}^*$ and Assumption \ref{ass4}, for all $x^k\in\mathbb{B}(\overline{x},\varepsilon_1)$, it holds that 
	\begin{equation}\label{muk-inverse}
		\frac{1}{\mu_k}=\frac{1}{a_2}r(x^k)^{-\varrho}\le\frac{\kappa^{\frac{\varrho}{q}}}{a_2}{\rm dist}(x^k,\mathcal{X}^*)^{-\frac{\varrho}{q}}\le\frac{\kappa^{\frac{\varrho}{q}}}{a_2}{\rm dist}(x^k,\mathcal{S}^*)^{-\frac{\varrho}{q}}.
	\end{equation}
	In addition, from \eqref{rk-ineq1} and Assumption \ref{ass4}, it follows that for all $x^k\in\mathbb{B}(\overline{x},\varepsilon_1)$,
	\[
	\frac{{\rm dist}(x^k,\mathcal{X}^*)}{\mu_k}
	\le\frac{\kappa}{a_2}r(x^k)^{q-\varrho}
	\le\frac{\kappa(2\!+L_{\nabla\!f})^{q-\varrho}}{a_2}{\rm dist}(x^k,\mathcal{S}^*)^{q-\varrho}.
	\]
	From the last four inequalities, we deduce that 
	for all $x^k\in\mathbb{B}(\overline{x},\varepsilon_1)$,
	\begin{align}\label{dk-ineq0}
		\|d^k\| &\le\bigg[\frac{\eta(1\!+\!\|G_k\|)(2\!+\!L_{\nabla\!f})^{1+\tau-\varrho}}{a_2}
		[{\rm dist}(x^k,\mathcal{S}^*)]^{\tau-\varrho}
		+\frac{L_{\psi}\|A\|^3\kappa^{\frac{\varrho}{q}}}{2a_2}[{\rm dist}(x^k,\mathcal{S}^*)]^{1-\frac{\varrho}{q}} \nonumber\\
		&\qquad +\frac{\kappa a_1L_{\psi}\|A\|^2(2\!+\!L_{\nabla\!f})^{q-\varrho}}{a_2}
		[{\rm dist}(x^k,\mathcal{S}^*)]^{q-\varrho}+2\bigg]{\rm dist}(x^k,\mathcal{S}^*).
	\end{align}
	From \eqref{muk-inverse}-\eqref{dk-ineq0}, the boundedness of $\{G_k\}$, and 
	$\lim_{k\to\infty}{\rm dist}(x^k,\mathcal{S}^*)=0$ by Proposition \ref{prop-xk} (iv), we conclude that
	$[(1-2\sigma)\mu_k]^{-1}\|d^k\|\le O({\rm dist}(x^k,S^*)^{1-\frac{\varrho}{q}})$.
	Along with $q>\varrho$,  there exists $\overline{k}\in\mathbb{N}$ 
	such that for all $k\ge\overline{k}$, when $x^k\in\mathbb{B}(\overline{x},\varepsilon_1)$, 
	$\frac{\|d^k\|}{(1-2\sigma)\mu_k}\le\frac{3}{L_{\psi}\|A\|^3}$ and
	$x^k+t d^k\in\mathbb{B}(\overline{x},\varepsilon_0)$ for all $t\in[0,1]$. 
	Fix any integer $m\ge 0$. By using \eqref{Hessian-Lip}, 
	\begin{align*}
		&f(x^k\!+\!\beta^md^k)\!-\!f(x^k)
		=\int_{0}^{1}\beta^m\langle[\nabla\!f(x^k\!+\!t\beta^md^k)-\nabla\!f(x^k)],d^k\rangle dt	+\beta^m\langle\nabla\!f(x^k),d^k\rangle\\
		&=\int_{0}^{1}\int_{0}^{1}t\beta^{2m}\langle d^k,[\nabla^2\!f(x^k\!+\!st\beta^md^k)
		\!-\!\nabla^2\!f(x^k)]d^k\rangle dsdt\\
		&\quad+\int_{0}^{1}t\beta^{2m}\langle d^k,\nabla^2\!f(x^k)d^k\rangle dt
		+\beta^m\langle\nabla\!f(x^k),d^k\rangle\\
		&\le\frac{L_{\psi}\|A\|^3}{6}\beta^{3m}\|d^k\|^3
		+\frac{1}{2}\beta^{2m}\langle d^k,\nabla^2\!f(x^k)d^k\rangle
		+\beta^m\langle\nabla\!f(x^k),d^k\rangle\\
		&\le\frac{L_{\psi}\|A\|^3}{6}\beta^{3m}\|d^k\|^3
		+\frac{1}{2}\beta^{m}\langle d^k,[\nabla^2\!f(x^k)\!+\!\Lambda_kA^{\top}A]d^k\rangle
		+\beta^m\langle\nabla\!f(x^k),d^k\rangle
	\end{align*}
	where the last inequality is due to $\Lambda_k\ge 0$ and $\nabla^2\!f(x^k)\!+\!\Lambda_kA^{\top}A\succeq 0$.
	In addition, from the convexity of the function $g$ and $\beta^m\in(0,1]$, it follows that
	\begin{align*}
		g(x^k\!+\!\beta^m d^k)-g(x^k)
		&\le\beta^m[g(y^k)-g(x^k)]=\beta^m[\ell_k(y^k)-\ell_k(x^k)-\langle\nabla\!f(x^k),d^k\rangle]\\
		&\le-\beta^m[\langle\nabla\!f(x^k),d^k\rangle+\frac{1}{2}\langle d^k,G_kd^k\rangle],
	\end{align*}
	where the last inequality is by \eqref{ineq-lk}. Adding the last two inequalities together leads to
	\begin{align}\label{Fineq}
		F(x^k\!+\!\beta^md^k)-F(x^k)+\sigma\mu_k\beta^{m}\|d^k\|^2
		&\le\frac{L_{\psi}\|A\|^3}{6}\beta^{3m}\|d^k\|^3-(1/2-\sigma)\mu_k\beta^m\|d^k\|^2\nonumber\\
		&\le-\frac{1}{2}\beta^{m}\|d^k\|^3\Big(\frac{(1\!-\!2\sigma)\mu_k}{\|d^k\|}-\frac{L_{\psi}\|A\|^3}{3}\Big).
	\end{align}
	Recall that for all $k\ge\overline{k}$, if $x^k\in\mathbb{B}(\overline{x},\varepsilon_1)$, 
	$\frac{(1-2\sigma)\mu_k}{\|d^k\|}\ge\frac{L_{\psi}\|A\|^3}{3}$. This means that \eqref{Fineq} holds with $m=0$ if $k\ge\overline{k}$ and $x^k\in\mathbb{B}(\overline{x},\varepsilon_1)$. 
	The desired result then follows. 
\end{proof}

From the proof of Lemma \ref{step-size}, Assumption \ref{ass4} implies that $\overline{x}\in\mathcal{X}^*$. It is worth pointing out that the result of Lemma \ref{step-size} also holds without the error bound condition if the parameter $\varrho$ of Algorithm \ref{IRPNM} is restricted in $(0,{1}/{2})$; see the {\color{blue}arxiv} version of this paper.

Now we are ready to establish the global convergence of $\{x^k\}_{k\in\mathbb{N}}$ and its superlinear rate in the following theorem, whose proof is inspired by that of \cite[Theorem 5.1]{Mordu20}.
\begin{theorem}\label{lconverge}
	Consider any $\overline{x}\in\omega(x^0)$. Suppose that Assumption \ref{ass1} holds, that $\nabla^2\psi$ is strictly continuous at $A\overline{x}\!-\!b$ relative to $A({\rm dom}\,g)\!-\!b$, and that Assumption \ref{ass4} holds with $q\in(\max\{\varrho, 1/(1\!+\!\varrho)\},1]$. Then, the sequence $\{x^k\}_{k\in\mathbb{N}}$ converges to $\overline{x}\in\mathcal{X}^*$ with the $Q$-superlinear rate of order $q(1\!+\!\varrho)$.
\end{theorem}
\begin{proof}
	Let $\overline{k}\in\mathbb{N},\varepsilon_0$ and $\varepsilon_1$ be the same as in Lemma \ref{step-size}. 
	Since $\lim_{k\to\infty}r(x^k)=0$ by Proposition \ref{prop-xk} (iii), $r(x^k)\le 1$ for all $k\ge\overline{k}$ (if necessary by increasing $\overline{k}$). Since $\lim_{k\to\infty}{\rm dist}(x^k,\mathcal{S}^*)=0$ 
	by Proposition \ref{prop-xk} (iv), from \eqref{dk-ineq0} and $\tau\!\ge\varrho$, there exists $c_1\!>0$ such that for all $k\ge\overline{k}$ (if necessary by increasing $\overline{k}$),
	\begin{equation}\label{dk-ineq1}
		\|d^k\|\le c_1{\rm dist}(x^k,\mathcal{S}^*).
	\end{equation}
	Also, from Lemma \ref{step-size}, $\alpha_k=1$ for all $x^k\in\mathbb{B}(\overline{x},\varepsilon_1)$ with $k\ge\overline{k}$.
	We next argue that for all $k\ge\overline{k}$, whenever $x^k\in\!\mathbb{B}(\overline{x},\varepsilon_1)$ and $x^{k+1}\!=x^k+d^k=y^k\in\mathbb{B}(\overline{x},\varepsilon_1)$, 
	\begin{equation}\label{recursion}
		{\rm dist}(x^{k+1},\mathcal{S}^*)=o({\rm dist}(x^{k},\mathcal{S}^*)).
	\end{equation}  
	Indeed, from ${\rm dist}(x^{k+1},\mathcal{X}^*)\!\le\!\kappa[r(x^{k+1})]^q$ and $r_k(y^k)\!\le\!\eta[r(x^k)]^{1+\tau}$ by \eqref{inexact-conda} and $r(x^k)\le 1$, it follows that
	\begin{align}\label{rk-ineq1aa}
		{\rm dist}(x^{k+1},\mathcal{S}^*)&\le{\rm dist}(x^{k+1},\mathcal{X}^*)
		\le\kappa\big[r(x^{k+1})-r_k(x^{k+1})+r_k(y^{k})\big]^q\nonumber\\
		&\le\kappa\big[r(x^{k+1})-r_k(x^{k+1})+\eta[r(x^{k})]^{1+\tau}\big]^q\nonumber\\
		&\le\kappa\big[|r(x^{k+1})-r_k(x^{k+1})|
		+\eta(2\!+\!L_{\nabla\!f})^{1+\tau}[{\rm dist}(x^k,\mathcal{S}^*)]^{1+\tau}\big]^q
	\end{align}
	where the last inequality is using \eqref{rk-ineq1}. Note that $(1-t)x^k+tx^{k+1}\in\mathbb{B}(\overline{x},\varepsilon_1)$ for all $t\in[0,1]$. Using  the expressions of $r$ and $r_k$ and inequality \eqref{dk-ineq1} yields that
	\begin{align*}
		&|r(x^{k+1})-r_k(x^{k+1})|\le\|\nabla\!f(x^{k+1})-\nabla\!f(x^k)-G_k(x^{k+1}-x^k)\|\\
		&\le\Big\|\int_{0}^{1}\!\big[\nabla^2\!f(x^k\!+\!t(x^{k+1}\!-\!x^k))\!-\!\nabla^2\!f(x^k)\big]
		(x^{k+1}\!-\!x^k)dt\Big\|+(\Lambda_k\|A\|^2\!+\!\mu_k)\|x^{k+1}\!-\!x^k\|\\
		&\le 0.5L_{\psi}\|A\|^3\|x^{k+1}-x^k\|^2+(\Lambda_k\|A\|^2\!+\!\mu_k)\|x^{k+1}-x^k\|\\
		&=0.5L_{\psi}\|A\|^3\|d^k\|^2+(\Lambda_k\|A\|^2\!+\!\mu_k)\|d^k\| \\
		&\le\big[0.5L_{\psi}\|A\|^3c_1^2{\rm dist}(x^k,\mathcal{S}^*)+c_1(\Lambda_k\|A\|^2\!+\!\mu_k)\big]{\rm dist}(x^k,\mathcal{S}^*).
	\end{align*}
	By combining this inequality with \eqref{rk-ineq1aa} and using Lemma \ref{lemma-Lamk}, Assumption \ref{ass4} and \eqref{rk-ineq1}, 
	\begin{align}\label{dist-recursion}
		{\rm dist}(x^{k+1},\mathcal{S}^*)
		&\le\kappa\Big[0.5L_{\psi}\|A\|^3c_1^2({\rm dist}(x^k,\mathcal{S}^*))^2
		+c_1a_1L\kappa(2\!+\!L_{\nabla\!f})^q[{\rm dist}(x^k,\mathcal{S}^*)]^{1+q}\nonumber\\
		&\quad+c_1a_2(2\!+\!L_{\nabla\!f})^{\varrho}[{\rm dist}(x^k,\mathcal{S}^*)]^{1+\varrho}
		+\eta(2\!+\!L_{\nabla\!f})^{1+\tau}[{\rm dist}(x^k,\mathcal{S}^*)]^{1+\tau}\Big]^{q}.
	\end{align}
	Since $q+q\varrho>1$ and $\lim_{k\to\infty}{\rm dist}(x^k,\mathcal{S}^*)=0$, using this inequality yields the stated relation in \eqref{recursion} (if necessary by increasing $\overline{k}$). Then, 
	for any $\widetilde{\sigma}\in(0,1)$, there exists $0<\varepsilon_2<\varepsilon_1$ such that for all $k\ge\overline{k}$, if $x^k\!\in\mathbb{B}(\overline{x},\varepsilon_{2})$ and
	$x^{k+1}\!=x^k+d^k=y^k\in\mathbb{B}(\overline{x},\varepsilon_{2})$,
	\begin{equation}\label{recursion1}
		{\rm dist}(y^k,\mathcal{S}^*)\le\widetilde{\sigma}{\rm dist}(x^k,\mathcal{S}^*).
	\end{equation}
	
	Let $\overline{\varepsilon}\!:=\!\min\big\{\frac{\varepsilon_2}{2},\frac{\varepsilon_2}{2c_1},
	\frac{(1-\widetilde{\sigma})\varepsilon_2}{2c_1}\big\}$. We argue by induction that if some iterate $x^{k_0}\in\mathbb{B}(\overline{x},\overline{\varepsilon})$ 
	with $k_0\ge\overline{k}$, then $\alpha_k=1$ and $x^{k+1}=y^k\in\mathbb{B}(\overline{x},\varepsilon_2)$
	for all $k\ge k_0$. Indeed, as $\overline{x}\in\omega(x^0)$, we can find $k_0>\overline{k}$ 
	such that $x^{k_0}\in\mathbb{B}(\overline{x},\overline{\varepsilon})$. Using \eqref{dk-ineq1} yields that
	\[
	\|y^{k_0}\!-\!\overline{x}\|
	\le\|x^{k_0}\!-\!\overline{x}\|+\|y^{k_0}\!-\!x^{k_0}\|=\|x^{k_0}\!-\!\overline{x}\|+\|d^{k_0}\|
	\le\|x^{k_0}\!-\!\overline{x}\|+c_1\|x^{k_0}\!-\!\overline{x}\|\le\varepsilon_2.
	\]
	Since $\overline{\varepsilon}<\varepsilon_1$, we have $\alpha_{k_0}=1$, and $x^{{k_0}+1}=y^{k_0}\in\mathbb{B}(\overline{x},\varepsilon_2)$.
	Fix any $k>k_0$. Assume that for all $k_0\le l\le k\!-\!1, \alpha_l=1$ and $x^{l+1}=y^{l}\in\mathbb{B}(\overline{x},\varepsilon_2)$. By \eqref{dk-ineq1} and \eqref{recursion1}, 
	\begin{align*}
		\|y^{k}\!-\!x^{k_0}\|\!\le\!\sum_{l=k_0}^{k}\|d^l\|
		\!\le c_1\!\sum_{l=k_0}^{k}{\rm dist}(x^l,\mathcal{S}^*)
		\!\le c_1\!\sum_{l=k_0}^{k}\widetilde{\sigma}^{l-k_0}{\rm dist}(x^{k_0},\mathcal{S}^*)
		\le\!\frac{c_1}{1\!-\!\widetilde{\sigma}}\|x^{k_0}\!-\!\overline{x}\|.
	\end{align*}
	Then $\|y^k-\overline{x}\|\le\|y^k-x^{k_0}\|+\|x^{k_0}-\overline{x}\|\le\varepsilon_2$. Note that $\alpha_k=1$ since $x^{k}\in\mathbb{B}(\overline{x},\varepsilon_2)$. Hence, 
	$x^{k+1}\!=x^k+d^k=y^k\!\in\mathbb{B}(\overline{x},\varepsilon_2)$. By induction the stated result holds.
	
	Since $\lim_{k\to\infty}{\rm dist}(x^k,\mathcal{S}^*)=0$, for any $\epsilon>0$ there exists
	$\mathbb{N}\ni\overline{k}_0\ge k_0$ such that for all $k\ge\overline{k}_0$, ${\rm dist}(x^{k},\mathcal{S}^*)\le\epsilon$.	Fix any $k_1\ge k_2\ge\overline{k}_0$. By invoking \eqref{dk-ineq1} and \eqref{recursion1},
	\begin{align}\label{xk-ineq2}
		\|x^{k_1}\!-\!x^{k_2}\|&\le\sum_{j=k_2}^{k_1-1}\|x^{j+1}-x^j\|
		= \sum_{j=k_2}^{k_1-1}\|d^j\|
		\le c_1\sum_{j=k_2}^{k_1-1}{\rm dist}(x^j,\mathcal{S}^*)\nonumber\\
		&\le c_1\sum_{j=k_2}^{k_1-1}\widetilde{\sigma}^{j-k_2}{\rm dist}(x^{k_2},\mathcal{S}^*)
		=\frac{c_1}{1-\widetilde{\sigma}}{\rm dist}(x^{k_2},\mathcal{S}^*)
		\le\frac{c_1\epsilon}{1-\widetilde{\sigma}},
	\end{align}
	where the first equality is using $\alpha_k=1$ for all $k\ge k_0$. This shows that $\{x^k\}_{k\in\mathbb{N}}$ is a Cauchy sequence and thus converges to $\overline{x}\in\mathcal{X}^*$.
	By passing the limit $k_1\to\infty$ to \eqref{xk-ineq2} and using \eqref{dist-recursion}, 
	we conclude that for any $k>\overline{k}_0$,
	\[
	\|x^{k+1}-\overline{x}\|
	\le\frac{c_1}{1-\widetilde{\sigma}}{\rm dist}(x^{k+1},\mathcal{S}^*)
	\le O([{\rm dist}(x^{k},\mathcal{S}^*)]^{q(1+\varrho)})
	\le O(\|x^k-\overline{x}\|^{q(1+\varrho)}).
	\]
	That is, $\{x^k\}_{k\in\mathbb{N}}$ converges to $\overline{x}$ with
	the $Q$-superlinear rate of order $q(1\!+\!\varrho)$. 
\end{proof}

The local error bound at $\overline{x}\in\omega(x^0)$ on $\mathcal{X}^*$ in Assumption \ref{ass4} is stronger than the $q$-subregularity of $\partial F$ at $\overline{x}$ for $0$ by Lemma \ref{sregular-relation}, but it does not require the isolatedness of $\overline{x}$. When $\overline{x}\in\mathcal{X}^*$, by Lemma \ref{alemma-ebound}, it is implied by the KL property of $F$ at $\overline{x}$ with exponent $1/2$ along with the quadratic growth of $F$ at $\overline{x}$ on $\mathcal{X}^*$, which is further implied by the local strong convexity of $F$ in a neighborhood of $\overline{x}$ by Remark \ref{remark-qsubregular} (c). 

The locally H\"{o}lderian error bound at any $\overline{x}\in\omega(x^0)$ on $\mathcal{X}^*$ in Assumption \ref{ass4} implicitly requires that $\mathcal{X}^*\!\ne\emptyset$. When $\mathcal{X}^*=\emptyset$, we can achieve the global convergence of $\{x^k\}$ and its superlinear rate under a local error bound on $\mathcal{S}^*$ as follows. Its proof is similar to that of Theorem \ref{lconverge}, see Appendix \ref{asec2}.
\begin{theorem}\label{converge}
	Fix any $\overline{x}\in\omega(x^0)$ and any $q\!>1\!+\!\varrho$. Suppose that Assumption \ref{ass1} holds, that $\nabla^2\psi$ is strictly continuous at $A\overline{x}\!-\!b$ relative to $A({\rm dom}\,g)\!-\!b$, and that there exist $\varepsilon>0$ and $\kappa>0$ such that ${\rm dist}(x,\mathcal{S}^*)\le \kappa[r(x)]^q$ for all $x\in\mathbb{B}(\overline{x},\varepsilon)$. Then, the sequence $\{x^k\}_{k\in\mathbb{N}}$ converges to $\overline{x}$ with the $Q$-superlinear rate of order ${(q\!-\!\varrho)^2}/{q}$ for $q>{(2\varrho\!+1+\!\sqrt{4\varrho\!+1})}/{2}$.
\end{theorem}
\begin{remark}\label{remark-converge} 
	When $H_k$ and $\eta_k$ in step (S.1) of \cite[Algorithm 3.1]{Kanzow21} respectively take $G_k$ and $\eta\min\big\{1,[r(x^k)]^{\tau}\big\}$, under the local error bound assumption of Theorem \ref{lconverge} or \ref{converge}, the unit step-size always occurs and inequality \eqref{dk-rk} holds with $\widehat{\kappa}=c_1\kappa$ and $\widehat{q}=q$. By Remark \ref{remark-IRPNM} (d), we can achieve the global convergence of the iterate sequence generated by GIPN and its superlinear rate under the same assumption as that of Theorem \ref{lconverge} or \ref{converge}, which greatly improve the convergence results in \cite{Kanzow21} where the uniformly bounded positive definiteness of $G_k$ and the local strong convexity of $F$ are required.
\end{remark}
	To close this section, we emphasize that the convergence results obtained in this section are also applicable to Algorithm \ref{IRPNM} with $G_k$ constructed by
	\begin{equation}\label{def-newG}
		G_k=\nabla^2f(x^k)+a_1A^{\top}{\rm Diag}(\max\{-\lambda^k,0\})A+\mu_kI,
	\end{equation}
	where $\lambda^k\in\mathbb{R}^m$ is the eigenvalue vector of  $\nabla^2\psi(Ax^k\!-\!b)$, and ${\rm Diag}(\max\{-\lambda^k,0\})$ is the diagonal matrix with the components of $\max\{0,-\lambda^k\}$ as the diagonal entries. Indeed, one can check that $G_k\succeq\mu_kI$ and ${\rm Diag}(\max\{-\lambda^k,0\})\preceq[-\lambda_{\rm min}(\nabla^2\psi(Ax^k\!-\!b))]_{+}I$. For $\varrho=0$, the uniformly bounded positive definiteness of $G_k$ guarantees that the sufficient descent property of objective values and the relative error condition in Lemma \ref{relgap} continue to hold. Thus, the global convergence follows from Theorem \ref{gconverge-KL}. For $\varrho\in(0,1)$, the order relation ${\rm Diag}(\max\{-\lambda^k,0\})\preceq[-\lambda_{\rm min}(\nabla^2\psi(Ax^k\!-\!b))]_{+}I$ along with Lemma \ref{lemma-Lamk} implies that $\|{\rm Diag}(\max\{-\lambda^k,0\})\|$ can be bounded by ${\rm dist}(x^k,\mathcal{X}^*)$. Thus, the global convergence and superlinear convergence rate results in Theorems \ref{lconverge}-\ref{converge} continue to hold.
\section{Numerical experiments}\label{sec5}
Before testing the performance of Algorithm \ref{IRPNM}, we first focus on its implementation (see \url{https://github.com/SCUT-OptGroup/IRPNM} for the corresponding code).
\subsection{Implementation of Algorithm \ref{IRPNM}}\label{sec5.1}
The core in the implementation of Algorithm \ref{IRPNM} is to find an approximate minimizer $y^k$ of subproblem \eqref{subprobx} satisfying \eqref{inexact-conda} or \eqref{inexact-condb}. By the expression of $G_k$ in \eqref{def-G}, we have $G_k=\!A_k^{\top}A_k+\mu_kI$ with $A_k\!=\big(D_k+a_1(-\lambda_{\rm min}(D_k))_{+}I\big)^{1/2}A$ 
for $D_k=\nabla^2\psi(Ax^k\!-\!b)$. When the function $\psi$ is separable, it is very cheap to achieve such a reformulation of $G_k$ as $D_k$ is a diagonal matrix. Let $b_k\!:=G_kx^k\!-\!\nabla\!f(x^k)$ and $g_k (\cdot):=g(\cdot)+({\mu_k}/{2})\|\cdot\|^2$. Then, subproblem \eqref{subprobx} can be equivalently written as 
\begin{equation}\label{EprobQ}
	\min_{y\in\mathbb{R}^n,z\in\mathbb{R}^m}
	\Big\{\frac{1}{2}\|z\|^2-b_k^{\top}y+g_{k}(y)
	\ \ {\rm s.t.}\ \ A_ky-z=0\Big\},
\end{equation}
which is a Lasso problem when $g$ takes the $\ell_1$-norm function.
Inspired by the encouraging results reported in \cite{LiSunToh18}, we develop 
an efficient SNALM for solving subproblem \eqref{EprobQ}. Note that a SNALM was also developed in \cite[Section 4.2]{Liu22} to solve the subproblems of an inexact SQA method for composite DC problems.
\subsubsection{SNALM for solving subproblems}\label{sec5.1.1}

Let $g_k^*$ be the conjugate function of $g_k$. The dual problem of \eqref{EprobQ} takes the form of 
\begin{equation}\label{DprobQ}
	\min_{\xi\in\mathbb{R}^m,\zeta\in\mathbb{R}^n}\Big\{\frac{1}{2}\|\xi\|^2+g_k^*(\zeta)
	\ \ {\rm s.t.}\ A_k^{\top}\xi+\zeta-b_k=0\Big\}.
\end{equation}
The basic iterate steps of the augmented Lagrangian method for \eqref{DprobQ} are as follows:
\begin{subnumcases}{}
	\label{DALM-subprob}
	(\xi^{j+1},\zeta^{j+1})
	:=\mathop{\arg\min}_{\xi\in\mathbb{R}^{m},\zeta\in\mathbb{R}^{n}}L_{\sigma_{\!j}}(\xi,\zeta;x^j),\\
	\label{DALM-multiplier} 
	x^{j+1}:=x^j+\sigma_{\!j}(A_k^{\top}\xi^{j+1}+\zeta^{j+1}-b_k),\\
	\sigma_{j+1}\uparrow\sigma_\infty\leq\infty\qquad\qquad   
\end{subnumcases}
where $L_{\sigma}(\cdot,\cdot;x)$ is the augmented Lagrangian function of \eqref{DprobQ} associated to the penalty factor $\sigma>0$ and the Lagrange multiplier $x$.
For the ALM, the primary computation cost is to solve \eqref{DALM-subprob}.
After calculating, $\zeta^{j+1}
=\mathcal{P}_{\!\sigma_{\!j}^{-1}g_k^*}(b_k\!-\!A_k^{\top}\xi^{j+1}-\sigma_{\!j}^{-1}x^j)$ with
\begin{equation}\label{Phij}
	\xi^{j+1}=\mathop{\arg\min}_{\xi\in\mathbb{R}^m}
	\Phi_j(\xi):=(1/2)\|\xi\|^2+e_{\sigma_{\!j}^{-1}g_k^*}\big(b_k-A_k^{\top}\xi-\sigma_{\!j}^{-1}x^j\big).
\end{equation}
By the strong convexity of $\Phi_j$, $\xi^{j+1}$ is a solution of \eqref{Phij} if and only if it is the unique root of system $0=\nabla\Phi_j(\xi)$. From $\mathcal{P}_{\sigma_j^{-1}g_k^*}\!=(I+\sigma_j^{-1}\partial g_k^*)^{-1}$ and \cite{Ioffe09}, when $g$ is definable in an o-minimal structure over the real field, the mapping $\mathcal{P}_{\sigma_j^{-1}g_k^*}$ is definable in this o-minimal structure. Thus, by the expression of $\nabla\Phi_j$ and \cite[Theorem 1]{Bolte09}, we conclude that $\nabla\Phi_j$ is semismooth if $g$ is definable in an o-minimal structure over the real field. Moreover, when $g$ is a piecewise linear-quadratic convex function, by \cite[Theorem 11.14 \& Proposition 12.30]{RW98} the mapping $\mathcal{P}_{\sigma_j^{-1}g_k^*}$ is piecewise linear, which by \cite[Proposition 7.4.7 \& 7.4.4]{Facchinei03} implies that $\nabla\Phi_j$ is strongly semismooth.
The Clarke Jacobian \cite{Clarke83} of $\nabla\Phi_j$ is always nonsingular due to its strong monotonicity, and one can achieve the exact characterization of its Clarke Jacobian for some specific $g$. Hence, we apply the semismooth Newton method to seeking a root of system $\nabla\Phi_j(\xi)=0$. For more details on semismooth Newton methods, see \cite{Qi93} or \cite[Chapter 7]{Facchinei03}.  

By combining the optimality conditions of \eqref{DALM-subprob}
with equation \eqref{DALM-multiplier}, we deduce that 
$\zeta^{j+1}\in\partial g_{k}(-x^{j+1})$, which implies that  $\omega^{j+1}\!:=G_ky^k\!-\!b_k+\zeta^{j+1}\in\partial\Theta_k(y^k)$ if taking 
$y^k=-x^{j+1}$. Inspired by this and criterion \eqref{inexact-conda} or \eqref{inexact-condb}, 
we terminate the ALM at iterate $(\xi^{j+1},\zeta^{j+1},x^{j+1})$ 
when $\Theta_k(-x^{j+1})\le\Theta_k(x^k)$ 
and $r_k(-x^{j+1})\le\epsilon_{\rm ALM}\!:=\!\eta\min\{r(x^k),[r(x^k)]^{1+\tau}\}$ 
for $\varrho\in(0,1)$, or when $\Theta_k(-x^{j+1})\le\Theta_k(x^k)$ and $\|\omega^{j+1}\|\le\epsilon_{\rm ALM}:=\eta r(x^k)$ for $\varrho=0$. In the implementation of the SNALM, we adjust the penalty factor $\sigma_j$ in terms of the primal and dual infeasibility. 

\subsubsection{Choice of parameters for Algorithm \ref{IRPNM}}\label{sec5.1.2}

For the subsequent tests, we choose $a_1=1$, which means that our $G_k$ may not be uniformly positive definite if $\varrho>0$. The other parameters of Algorithm \ref{IRPNM} are chosen as 
$a_2\!=\min\{\sigma,\frac{10^{-2}}{\max\{1,r(x^0)\}}\}$ with
$\sigma=10^{-4}$, $\beta=0.1,\eta=0.9,\tau=\varrho$. From Figure \ref{fig0a} below, a larger $\varrho$ corresponds to a better convergence rate, but preliminiary tests indicate that Algorithm \ref{IRPNM} with a larger $\varrho$ does not necessarily require less running time and Algorithm \ref{IRPNM} with $\varrho\in[0.4,0.5]$ usually works well in terms of running time. Inspired by this, we choose Algorithm \ref{IRPNM} with $\varrho=0.45$ for the subsequent testing. 

Figure \ref{fig0} below plots the convergence rate curves of the iterate sequences yielded by Algorithm \ref{IRPNM} with different $\varrho$ for the example in Section \ref{sec5.2} with $d=80${\bf dB} in Figure \ref{fig0a} and the example in Section \ref{sec5.4} with $\lambda=10^{-2}$ in Figure \ref{fig0b}. We see that every curve in Figure \ref{fig0a} shows a superlinear convergence rate, but no curve in Figure \ref{fig0b} does this. After checking, we find that the sequences $\{x^k\}$ in Figure \ref{fig0a} all converge to a second-order stationary point, while those in Figure \ref{fig0b} converge to a non-strong one. This coincides with the theoretical results obtained in Section \ref{sec4}.  
\begin{figure}[h]
	\centering
	\subfigure[\label{fig0a}]{\includegraphics[scale=0.45]{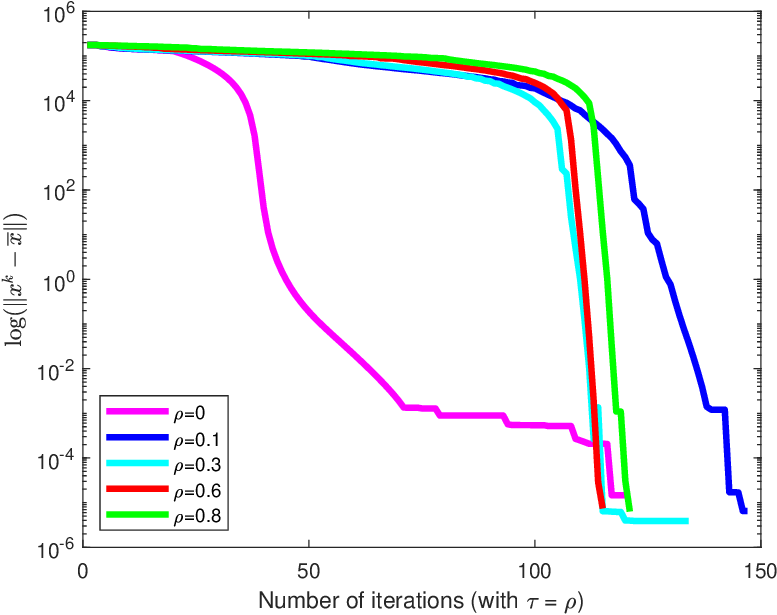}}
	\subfigure[\label{fig0b}]{\includegraphics[scale=0.45]{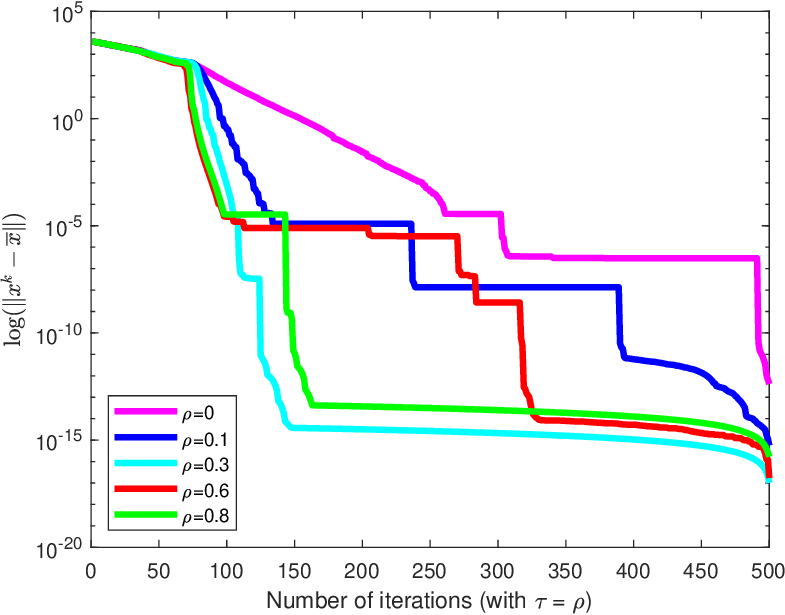}}
	\setlength{\abovecaptionskip}{10pt}
	\caption{Influence of $\varrho$ on the convergence behavior of the iterate sequences of Algorithm \ref{IRPNM}}
	\label{fig0}
\end{figure}

In the rest of this section, we compare the performance of Algorithm \ref{IRPNM} armed with the SNALM (called IRPNM) with that of GIPN and ZeroFPR on $\ell_1$-regularized Student's $t$-regressions, group penalized Student's $t$-regressions, and nonconvex image restoration. Among others, GIPN was proposed in \cite{Kanzow21} by solving subproblem \eqref{subprobx} with FISTA \cite{Beck09} (named GIPN-F) or the KKT system $R_k(x)=0$ with the semismooth Newton method \cite{Milzarek14} (named GIPN-S), and ZeroFPR was proposed in \cite{Themelis18} by seeking the root of the KKT system for minimizing the FBE of $F$ along a quasi-Newton direction. For GIPN, we choose $a_1=1.001$ to ensure that the approximate matrix $G_k$ is uniformly bounded positive definite as required by its convergence results. For GIPN and ZeroFPR, we adopt the default setting except that their maximum number of iterations are set as $2000$ and $20000$, respectively. For IRPNM, the maximum numbers of iterations of Algorithm \ref{IRPNM} and SNALM are set as $1000$ and $100$, respectively. For fairness, three solvers use the same starting point $x^{{\rm init}}$ and the same stopping condition $r(x^{k})\le \epsilon_0$. All tests are performed on a desktop running 64-bit Windows System with an Intel(R) Core(TM) i9-10850K CPU 3.60GHz and 32.0 GB RAM on Matlab 2020b.

\subsection{$\ell_{1}$-regularized Student's $t$-regression}\label{sec5.2}
This class of problems takes the form of \eqref{prob} with 
$\psi(u)\!:=\sum_{i=1}^m\log(1+{u_i^2}/{\nu})\ (\nu>0)$ for $u\in\mathbb{R}^m$
and $g(x)\!:=\lambda\|x\|_1$ for $x\in\mathbb{R}^n$, where $\lambda>0$ is the regularized parameter. Such a function was introduced in \cite{Aravkin12} to deal with the data contaminated by heavy-tailed Student-$t$ errors. The test examples are randomly generated in the same way as in \cite{Becker11,Milzarek14}. Specifically, we generate a true sparse signal $x^{\rm true}$ of length $n=512^2$ with $s=\lfloor\frac{n}{40}\rfloor$ nonzero entries, 
whose indices are chosen randomly; and then compute each nonzero component via $x^{\rm true}_i=\eta_1(i)10^{d\eta_2(i)/20}$, where $\eta_1(i)\in\{\pm1\}$ is a random sign and $\eta_2(i)$ is uniformly distributed in $[0,1]$. The signal has dynamic range of $d${\bf dB} 
with $d\in\{20,40,60,80\}$. The matrix $A\in\mathbb{R}^{m\times n}$ 
takes $m={n}/{8}$ random cosine measurements, i.e., $Ax = (\verb"dct"(x))_J$, 
where $\verb"dct"$ is the discrete cosine transform and 
$J\subset\{1,\ldots,n\}$ with $|J|=m$ is an index set chosen randomly.
The vector $b$ is obtained by adding Student's $t$-noise with
degree of freedom $4$ and rescaled by $0.1$ to $Ax^{\rm true}$. 

For each $\lambda=c_{\lambda}\|\nabla\!f(0)\|_{\infty}$ with $c_{\lambda}=0.1$ and $0.01$, we run the three solvers with $\nu=0.25,x^{{\rm init}}\!=A^{\top}b$ and $\epsilon_0=10^{-5}$ for $10$ independent trials, i.e., ten groups of data $(A,b)$ generated randomly. Table \ref{table0} lists the average objective values, KKT residuals and running times over $10$ independent trials. We see that IRPNM yields the same objective values as ZeroFPR does, but requires much less running time, while GIPN-S can not achieve the desired accuracy within $2000$ iterations for most of test instances. The three solvers all require less running time for $c_{\lambda}=0.1$ than $c_{\lambda}=0.01$ due to more sparsity of stationary points. After checking, we find that IRPNM yields a second-order stationary point for all test instances except those for $c_{\lambda}=0.1,d=20$, and  the smallest eigenvalue of Hessian at these second-order stationary points is numerically close to $0$, so it is highly possible for the stationary point to be nonisolated.
\begin{table}[h]
	\centering
	\setlength{\tabcolsep}{1.2pt}
	\setlength{\belowcaptionskip}{2pt}
	\caption{Numerical comparisons on $\ell_1$-regularized Student's $t$-regressions}
	\label{table0}
	\begin{tabular}{c|c|ccc|ccc|ccc}
		\hline
		& &\multicolumn{3}{c|}{IRPNM} & \multicolumn{3}{c|}{ZeroFPR} & \multicolumn{3}{c}{GIPN-S} \\
		\hline
		$c_{\lambda}$ & $d$ & Fval & r(x) & time(s) & Fval & r(x) & time(s) & Fval & r(x) & time(s) \\
		\hline
		\multirow{4}{*}{0.1} & {20}  
		& 9532.54 & 8.80e-6 & 9.1 
		& 9532.54 & 9.15e-6 & 9.2 
		& 9532.54 & 7.80e-6 & 12.2  \\ 
		\multirow{4}{*}{} & {40}  
		& 23812.87 & 8.40e-6 & 24.0 
		& 23812.88 & 9.26e-6 & 27.3 
		& 23812.88 & 5.18e-6 & 142.4  \\ 
		\multirow{4}{*}{} & {60}  
		& 54228.01 & 8.09e-6 & 53.3 
		& 54228.01 & 9.48e-6 & 95.8 
		& 54228.01 & 2.48e-3 & 1413.8  \\ 
		\multirow{4}{*}{} & {80} 
		& 134779.26 & 8.35e-6 & 229.9 
		& 134779.26 & 9.51e-6 & 327.4 
		& 222306.40 & 4.22 & 1116.3   \\ \hline
		\multirow{4}{*}{0.01} & {20}   
		& 1020.43 & 8.30e-6 & 24.6 
		& 1020.43 & 9.56e-6 & 63.5 
		& 1020.43 & 3.74e-3 & 1459.5 \\ 
		\multirow{4}{*}{} & {40}
		& 2395.07 & 7.24e-6 & 79.1 
		& 2395.07 & 9.63e-6 & 201.6 
		& 2395.10 & 1.17e-2 & 1400.6 \\ 
		\multirow{4}{*}{} & {60}
		& 5424.40 & 8.21e-6 & 172.0 
		& 5424.40 & 9.62e-6 & 388.0 
		& 5424.41 & 9.74e-4 & 1206.3 \\ 
		\multirow{4}{*}{} & {80}
		& 13478.10 & 6.92e-6 & 768.7 
		& 13478.10 & 9.73e-6 & 920.4 
		& 37588.03 & 1.41 & 1019.4\\ \hline
	\end{tabular}
\end{table}
\subsection{Group penalized Student's $t$-regression}\label{sec5.3}
This class of problems takes the form of \eqref{prob} with $\psi$ being the same as in Section \ref{sec5.2} and $g(x)\!:=\lambda\sum_{i=1}^l\|x_{J_i}\|$ for $x\in\mathbb{R}^n$, where the index sets $J_1,\ldots,J_{l}$ satisfy 
$J_i\cap J_j=\emptyset$ for any $i\ne j$ and $\bigcup_{i=1}^lJ_i=\{1,\ldots,n\}$. We generate a true group sparse signal $x^{\rm true}\in\mathbb{R}^n$ of length $n=512^2$ with $s$ nonzero groups, whose indices are chosen randomly, and compute each nonzero entry of $x^{\rm true}$ by the same formula as in Section \ref{sec5.2}. The matrix $A\in\mathbb{R}^{m\times n}$ is generated in the same way as in Section \ref{sec5.2}, and $b\in\mathbb{R}^m$ is obtained by adding Student's $t$-noise with degree of freedom $5$ and rescaled by $0.1$ to $Ax^{\rm true}$. 

We run the three solvers with $\lambda\!=\!0.1\|\nabla\!f(0)\|,\nu\!=0.2,x^{{\rm init}}\!=A^{\top}b$ and $\epsilon_0=10^{-5}$ for $10$ independent trials, i.e., ten groups of data $(A,b)$ generated randomly. Table \ref{table1} lists the average objective values, KKT residuals and running times over $10$ independent trials for corresponding $d=60, 80$ {\bf dB} and nonzero group $s=16, 64, 128$. We see that IRPNM yields the same objective values as ZeroFPR does, but requires much less running time than the latter does; while GIPN-F can not achieve the desired accuracy within $2000$ iterations for all test instances. Also, after checking, IRPNM yields a second-order stationary point for each test instance. 

\begin{table}[h]
	\centering
	\setlength{\tabcolsep}{2pt}
	\setlength{\belowcaptionskip}{2pt}
	\caption{Numerical comparisons on group penalized Student's $t$-regressions}
	\label{table1}
	\begin{tabular}{c|c|ccc|ccc|ccc}
		\hline
		& &\multicolumn{3}{c|}{IRPNM} & \multicolumn{3}{c|}{ZeroFPR} & \multicolumn{3}{c}{GIPN-F} \\
		\hline
		$d$ & $s$ & Fval & r(x) & time(s) & Fval & r(x) & time(s) & Fval & r(x) & time(s) \\
		\hline
		\multirow{3}{*}{60} & {16}  
		& 12711.87 & 8.43e-6 & 14.97 
		& 12711.87 & 9.17e-6 & 169.18
		& 12711.89 & 1.32e-3 & 1715.69   \\ 
		\multirow{3}{*}{} & {64}  
		& 17852.99 & 8.10e-6 & 23.21 
		& 17852.99 & 9.10e-6 & 233.36 
		& 17859.13 & 1.33e-2 & 1718.29\\ 
		\multirow{3}{*}{} & {128}  
		& 21670.19 & 9.14e-6 & 29.49 
		& 21670.19 & 9.05e-6 & 220.83
		& 21693.24 & 2.03e-2 & 1727.34 \\ \hline
		\multirow{3}{*}{80} & {16}   
		& 37037.71 & 9.66e-6 & 118.62
		& 37037.71 & 9.35e-6 & 642.72 
		& 38026.26 & 5.70e-2 & 1715.79 \\ 
		\multirow{3}{*}{} & {64}
		& 52741.59 & 9.71e-6 & 222.61
		& 52741.59 & 9.17e-6 & 688.95 
		& 53417.19 & 2.85e-2 & 1718.22 \\ 
		\multirow{3}{*}{} & {128}
		& 63451.74 & 9.34e-6 & 331.96 
		& 63451.74 & 9.35e-6 & 589.64 
		& 64056.72 & 2.09e-2 & 1717.74\\ \hline
	\end{tabular}
\end{table}
\subsection{Restoration of blurred images}\label{sec5.4}

This class of problems has the form of \eqref{prob} with $\psi$ described as in Section \ref{sec5.2} and $g(x)\!=\lambda\|Bx\|_1$ for $x\in\mathbb{R}^n$, where $A\in\mathbb{R}^{n\times n}$ is a matrix to represent a Gaussian blur operator with standard deviation 4 and a filter size of $9$, 
the vector $b\in\mathbb{R}^n$ represents a blurred image,
and $B\in\mathbb{R}^{n\times n}$ is an orthogonal matrix to represent 
a two-dimensional discrete Haar wavelet transform of level $4$. 
A $256\times 256$ image \textsf{cameraman.tif} is chosen as the test image 
$x^{\rm true}\in\mathbb{R}^n$ with $n=256^2$, and the blurred noisy image 
$b$ is obtained by adding Student's $t$-noise with degree of freedom 1 
and rescaled by $10^{-3}$ to $Ax^{\rm true}$.

For each $\lambda\in\{10^{-2},10^{-3},10^{-4}\}$, we run the three solvers with $\nu=1,x^{\rm init}\!=b$ and $\epsilon_0=10^{-4}$ for $10$ independent trials, i.e., ten groups of data $b$ generated randomly. As shown in Figure \ref{fig2}, the three solvers all demonstrate good restorations. Table \ref{table2} reports the average objective values, KKT residuals and running times over $10$ independent trials. We see that IRPNM and ZeroFPR outperform GIPN-F in achieving the desired accuray within less running time, and IRPNM yields the objective values (or the KKT residuals) comparable with those yielded by ZeroFPR though it needs more time than ZeroFPR does. For this class of problems, IRPNM yields a common stationary point for each test instance, and displays a slower convergence rate than ZeroFPR does, which accounts for why it requires more running time than the latter does. 

\begin{figure}[h]
	\centering
	\setlength{\abovecaptionskip}{2pt}
	\subfigure[]{
		\includegraphics[width=0.16\linewidth]{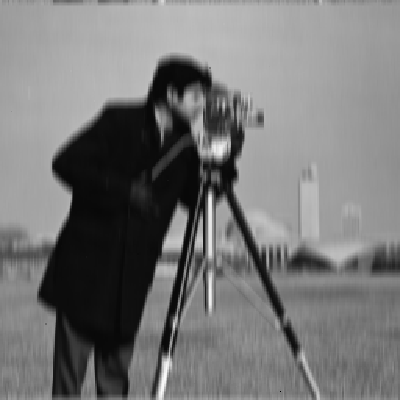}
		\label{fig2a}}
	\subfigure[]{
		\includegraphics[width=0.16\linewidth]{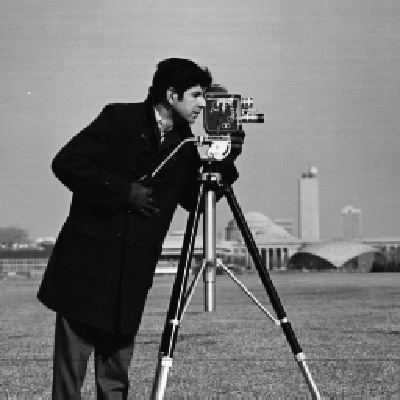}
		\label{fig2b}}
	\subfigure[]{
		\includegraphics[width=0.16\linewidth]{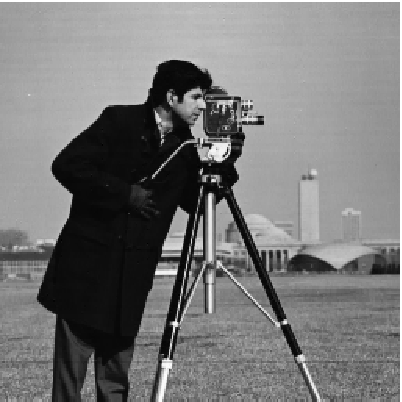}
		\label{fig2c}}
	\subfigure[]{
		\includegraphics[width=0.16\linewidth]{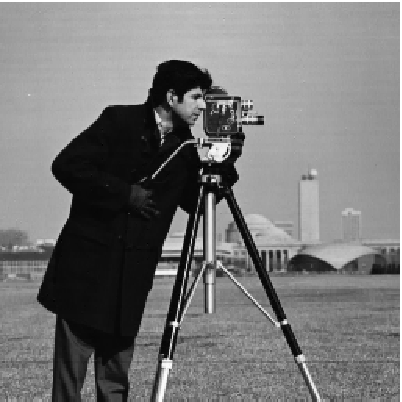}
		\label{fig2d}}
	\subfigure[]{
		\includegraphics[width=0.16\linewidth]{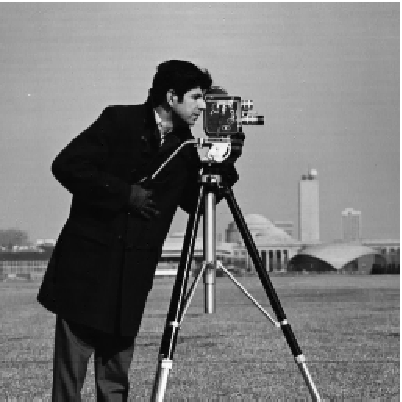}
		\label{fig2e}}
	\caption{Nonconvex image restoration. Recovered images with the three solvers for
		$\epsilon_0=10^{-3}$ and $\lambda=10^{-2}$. (a) Noisy blurred image, (b) original image, (c) IRPNM, 
		(d) ZeroFPR, (e) GIPN-F}
	\label{fig2}
\end{figure}
\begin{table}[h]
	\centering
	\setlength{\tabcolsep}{1.5pt}
	\setlength{\belowcaptionskip}{2pt}
	\caption{Numerical comparisons on nonconvex image restoration}
	\label{table2}
	\begin{tabular}{c|ccc|ccc|ccc}
		\hline
		&\multicolumn{3}{c|}{IRPNM} &\multicolumn{3}{c|}{ZeroFPR} & \multicolumn{3}{c}{GIPN-F}	\\ 
		\hline
		$\lambda$ & Fval & r(x) & time(s)  & Fval & r(x) & time(s)  & Fval & r(x) & time(s) \\ \hline
		1e-2 & 11245.27 & 9.80e-5 & 134.46 
		& 11245.27 & 9.28e-5 & 18.23 
		& 11245.27 & 9.75e-5 & 424.21 \\ \hline
		1e-3 & 1199.45	& 8.35e-5 & 270.62 
		& 1199.33 & 9.59e-5 & 43.71 
		& 1199.45 & 9.95e-5 & 931.03 \\ \hline
		1e-4 & 146.99 & 9.41e-5 & 425.60
		& 146.60 & 9.43e-5 & 122.26 
		& 147.23 & 3.60e-4 & 2713.90 \\ \hline
	\end{tabular}
\end{table}
\section{Conclusions.}\label{sec6}

 We proposed an inexact regularized proximal Newton method for the nonconvex and nonsmooth problem \eqref{prob} by using the approximate matrix $G_k$ in \eqref{def-G} for the Hessian of $f$ at $x^k$. For $\varrho=0$, we verified the global convergence of the iterate sequence and its R-linear convergence rate under suitable KL assumptions on $F$; and for $\varrho\in(0,1)$, established the global convergence of the iterate sequence and its superlinear convergence rate under suitable locally error bound on the (second-order) stationary point set. Numerical experiments confirmed our theoretical findings, and numerical comparisons with the state-of-the-art solvers GIPN and ZeroFPR demonstrated the efficiency of IRNPM armed with SNALM, especially for those problems with second-order stationary points.
\section*{Acknowledgments}
The authors would like to thank Professor Kanzow and Dr. Lechner from 
University of Würzburg for their codes sharing. 

\bibliographystyle{siam}
\bibliography{references.bib}

\begin{thebibliography}{10}

\bibitem{Aravkin12}
{\sc A.~Aravkin, M.~P. Friedlander, F.~J. Herrmann, and T.~van Leeuwen}, {\em
  Robust inversion, dimensionality reduction, and randomized sampling}, Math.
  Program., 134 (2012), pp.~101--125.

\bibitem{Attouch09}
{\sc H.~Attouch and J.~Bolte}, {\em On the convergence of the proximal
  algorithm for nonsmooth functions involving analytic features}, Math.
  Program., 116 (2009), pp.~5--16.

\bibitem{Attouch10}
{\sc H.~Attouch, J.~Bolte, P.~Redont, and A.~Soubeyran}, {\em Proximal
  alternating minimization and projection methods for nonconvex problems: an
  approach based on the {K}urdyka-{{\L}}ojasiewicz inequality}, Math. Oper.
  Res., 35 (2010), pp.~438--457.

\bibitem{Attouch13}
{\sc H.~Attouch, J.~Bolte, and B.~F. Svaiter}, {\em Convergence of descent
  methods for semi-algebraic and tame problems: proximal algorithms,
  forward-backward splitting, and regularized {G}auss-{S}eidel methods}, Math.
  Program., 137 (2013), pp.~91--129.

\bibitem{Beck19}
{\sc A.~Beck and N.~Hallak}, {\em Optimization problems involving group
  sparsity terms}, Math. Program., 178 (2019), pp.~39--67.

\bibitem{Beck09}
{\sc A.~Beck and M.~Teboulle}, {\em A fast iterative shrinkage-thresholding
  algorithm for linear inverse problems}, SIAM J. Imaging Sci., 2 (2009),
  pp.~183--202.

\bibitem{Becker11}
{\sc S.~Becker, J.~Bobin, and E.~J. Cand\`es}, {\em N{ESTA}: a fast and
  accurate first-order method for sparse recovery}, SIAM J. Imaging Sci., 4
  (2011), pp.~1--39.

\bibitem{Bertsekas99}
{\sc D.~P. Bertsekas}, {\em Nonlinear programming}, Athena Scientific, Belmont,
  Massachusetts, 2nd~ed., 1999.

\bibitem{Bolte09}
{\sc J.~Bolte, A.~Daniilidis, and A.~Lewis}, {\em Tame functions are
  semismooth}, Math. Program., 117 (2009), pp.~5--19.

\bibitem{Bolte17}
{\sc J.~Bolte, T.~P. Nguyen, J.~Peypouquet, and B.~W. Suter}, {\em From error
  bounds to the complexity of first-order descent methods for convex
  functions}, Math. Program., 165 (2017), pp.~471--507.

\bibitem{Bolte14}
{\sc J.~Bolte, S.~Sabach, and M.~Teboulle}, {\em Proximal alternating
  linearized minimization for nonconvex and nonsmooth problems}, Math.
  Program., 146 (2014), pp.~459--494.

\bibitem{Bonettini16}
{\sc S.~Bonettini, I.~Loris, F.~Porta, and M.~Prato}, {\em Variable metric
  inexact line-search-based methods for nonsmooth optimization}, SIAM J.
  Optim., 26 (2016), pp.~891--921.

\bibitem{Bonettini17}
{\sc S.~Bonettini, I.~Loris, F.~Porta, M.~Prato, and S.~Rebegoldi}, {\em On the
  convergence of a linesearch based proximal-gradient method for nonconvex
  optimization}, Inverse Problems, 33 (2017).

\bibitem{Bonettini20}
{\sc S.~Bonettini, M.~Prato, and S.~Rebegoldi}, {\em Convergence of inexact
  forward-backward algorithms using the forward-backward envelope}, SIAM J.
  Optim., 30 (2020), pp.~3069--3097.

\bibitem{Byrd16}
{\sc R.~H. Byrd, J.~Nocedal, and F.~Oztoprak}, {\em An inexact successive
  quadratic approximation method for {$\ell_1$} regularized optimization},
  Math. Program., 157 (2016), pp.~375--396.

\bibitem{Clarke83}
{\sc F.~H. Clarke}, {\em Optimization and nonsmooth analysis}, vol.~5 of
  Classics in Applied Mathematics, Society for Industrial and Applied
  Mathematics (SIAM), Philadelphia, 2nd~ed., 1990.

\bibitem{Dong09}
{\sc Y.~Dong}, {\em An extension of {L}uque's growth condition}, Appl. Math.
  Lett., 22 (2009), pp.~1390--1393.

\bibitem{Drusvyatskiy18}
{\sc D.~Drusvyatskiy and A.~S. Lewis}, {\em Error bounds, quadratic growth, and
  linear convergence of proximal methods}, Math. Oper. Res., 43 (2018),
  pp.~919--948.

\bibitem{Facchinei03}
{\sc F.~Facchinei and J.-S. Pang}, {\em Finite-dimensional variational
  inequalities and complementarity problems. {V}ol. {II}}, Springer Series in
  Operations Research, Springer-Verlag, New York, 2003.

\bibitem{Fischer02}
{\sc A.~Fischer}, {\em Local behavior of an iterative framework for generalized
  equations with nonisolated solutions}, Math. Program., 94 (2002),
  pp.~91--124.

\bibitem{Friedman07}
{\sc J.~Friedman, T.~Hastie, H.~H\"{o}fling, and R.~Tibshirani}, {\em Pathwise
  coordinate optimization}, Ann. Appl. Stat., 1 (2007), pp.~302--332.

\bibitem{Fukushima81}
{\sc M.~Fukushima and H.~Mine}, {\em A generalized proximal point algorithm for
  certain non-convex minimization problems}, Internat. J. Systems Sci., 12
  (1981), pp.~989--1000.

\bibitem{Hsieh11}
{\sc C.-J. Hsieh, M.~A. Sustik, I.~S. Dhillon, and P.~Ravikumar}, {\em Sparse
  inverse covariance matrix estimation using quadratic approximation}, in
  Proceedings of the 24th International Conference on Neural Information
  Processing Systems, Red Hook, NY, USA, 2011, Curran Associates Inc.,
  p.~2330–2338.

\bibitem{Ioffe09}
{\sc A.~D. Ioffe}, {\em An invitation to tame optimization}, SIAM J. Optim., 19
  (2008), pp.~1894--1917.

\bibitem{Kanzow21}
{\sc C.~Kanzow and T.~Lechner}, {\em Globalized inexact proximal {N}ewton-type
  methods for nonconvex composite functions}, Comput. Optim. Appl., 78 (2021),
  pp.~377--410.

\bibitem{Lee19}
{\sc C.~Lee and S.~J. Wright}, {\em Inexact successive quadratic approximation
  for regularized optimization}, Comput. Optim. Appl., 72 (2019), pp.~641--674.

\bibitem{Lee14}
{\sc J.~D. Lee, Y.~Sun, and M.~A. Saunders}, {\em Proximal {N}ewton-type
  methods for minimizing composite functions}, SIAM J. Optim., 24 (2014),
  pp.~1420--1443.

\bibitem{Li12}
{\sc G.~Li and B.~S. Mordukhovich}, {\em H\"{o}lder metric subregularity with
  applications to proximal point method}, SIAM J. Optim., 22 (2012),
  pp.~1655--1684.

\bibitem{LiPong18}
{\sc G.~Li and T.~K. Pong}, {\em Calculus of the exponent of
  {K}urdyka-{{\L}}ojasiewicz inequality and its applications to linear
  convergence of first-order methods}, Found. Comput. Math., 18 (2018),
  pp.~1199--1232.

\bibitem{Li16}
{\sc J.~Li, M.~S. Andersen, and L.~Vandenberghe}, {\em Inexact proximal
  {N}ewton methods for self-concordant functions}, Math. Methods Oper. Res., 85
  (2017), pp.~19--41.

\bibitem{LiSunToh18}
{\sc X.~Li, D.~Sun, and K.-C. Toh}, {\em A highly efficient semismooth {N}ewton
  augmented {L}agrangian method for solving lasso problems}, SIAM J. Optim., 28
  (2018), pp.~433--458.

\bibitem{Liu22}
{\sc T.~Liu and A.~Takeda}, {\em An inexact successive quadratic approximation
  method for a class of difference-of-convex optimization problems}, Comput.
  Optim. Appl., 82 (2022), pp.~141--173.

\bibitem{Luo93}
{\sc Z.-Q. Luo and P.~Tseng}, {\em Error bounds and convergence analysis of
  feasible descent methods: a general approach}, Ann. Oper. Res., 46 (1993),
  pp.~157--178.

\bibitem{Milzarek14}
{\sc A.~Milzarek and M.~Ulbrich}, {\em A semismooth {N}ewton method with
  multidimensional filter globalization for {$l_1$}-optimization}, SIAM J.
  Optim., 24 (2014), pp.~298--333.

\bibitem{Mordu15}
{\sc B.~S. Mordukhovich and W.~Ouyang}, {\em Higher-order metric subregularity
  and its applications}, J. Global Optim., 63 (2015), pp.~777--795.

\bibitem{Mordu20}
{\sc B.~S. Mordukhovich, X.~Yuan, S.~Zeng, and J.~Zhang}, {\em A globally
  convergent proximal {N}ewton-type method in nonsmooth convex optimization},
  Math. Program., 198 (2023), pp.~899--936.

\bibitem{Nesterov04}
{\sc Y.~Nesterov}, {\em Introductory lectures on convex optimization: A basic
  course}, vol.~87 of Applied Optimization, Kluwer Academic Publishers, Boston,
  2004.

\bibitem{Oztoprak12}
{\sc F.~Oztoprak, J.~Nocedal, S.~Rennie, and P.~A. Olsen}, {\em Newton-like
  methods for sparse inverse covariance estimation}, in Advances in Neural
  Information Processing Systems, vol.~25, Red Hook, NY, USA, 2012, Curran
  Associates, Inc.

\bibitem{LiuPan19}
{\sc S.~Pan and Y.~Liu}, {\em Subregularity of subdifferential mappings
  relative to the critical set and kl property of exponent 1/2}, 2019.

\bibitem{Patrinos13}
{\sc P.~Patrinos and A.~Bemporad}, {\em Proximal {N}ewton methods for convex
  composite optimization}, in 52nd IEEE Conference on Decision and Control,
  IEEE, 2013, pp.~2358--2363.

\bibitem{Dihn11}
{\sc T.~Pham~Dinh and Y.-S. Niu}, {\em An efficient {DC} programming approach
  for portfolio decision with higher moments}, Comput. Optim. Appl., 50 (2011),
  pp.~525--554.

\bibitem{Qi93}
{\sc L.~Q. Qi and J.~Sun}, {\em A nonsmooth version of {N}ewton's method},
  Math. Program., 58 (1993), pp.~353--367.

\bibitem{Roc70}
{\sc R.~T. Rockafellar}, {\em Convex analysis}, Princeton Mathematical Series,
  No. 28, Princeton University Press, Princeton, 1970.

\bibitem{RW98}
{\sc R.~T. Rockafellar and R.~J.-B. Wets}, {\em Variational analysis},
  Springer-Verlag, Berlin, Heidelberg, 1998.

\bibitem{Stella17}
{\sc L.~Stella, A.~Themelis, and P.~Patrinos}, {\em Forward--backward
  quasi-{N}ewton methods for nonsmooth optimization problems}, Comput. Optim.
  Appl., 67 (2017), pp.~443--487.

\bibitem{Themelis18}
{\sc A.~Themelis, L.~Stella, and P.~Patrinos}, {\em Forward-backward envelope
  for the sum of two nonconvex functions: Further properties and nonmonotone
  linesearch algorithms}, SIAM J. Optim., 28 (2018), pp.~2274--2303.

\bibitem{Tibshirani96}
{\sc R.~Tibshirani}, {\em Regression shrinkage and selection via the lasso}, J.
  R. Stat. Soc. Ser. B. Stat. Methodol., 58 (1996), pp.~267--288.

\bibitem{Dihn15}
{\sc Q.~Tran-Dinh, A.~Kyrillidis, and V.~Cevher}, {\em Composite
  self-concordant minimization}, J. Mach. Learn. Res., 16 (2015), pp.~371--416.

\bibitem{Tseng09}
{\sc P.~Tseng and S.~Yun}, {\em A coordinate gradient descent method for
  nonsmooth separable minimization}, Math. Program., 117 (2009), pp.~387--423.

\bibitem{Ueda10}
{\sc K.~Ueda and N.~Yamashita}, {\em Convergence properties of the regularized
  {N}ewton method for the unconstrained nonconvex optimization}, Appl. Math.
  Optim., 62 (2010), pp.~27--46.

\bibitem{Ngai09}
{\sc H.~van Ngai and M.~Th\'{e}ra}, {\em Error bounds for systems of lower
  semicontinuous functions in {A}splund spaces}, Math. Program., 116 (2009),
  pp.~397--427.

\bibitem{WuPanBi21}
{\sc Y.~Wu, S.~Pan, and S.~Bi}, {\em {K}urdyka--{{\L}}ojasiewicz property of
  zero-norm composite functions}, J. Optim. Theory Appl., 188 (2021),
  pp.~94--112.

\bibitem{YuLiPong21}
{\sc P.~Yu, G.~Li, and T.~K. Pong}, {\em {K}urdyka--{{\L}}ojasiewicz exponent
  via inf-projection}, Found. Comput. Math.,  (2021), pp.~1--47.

\bibitem{Yuan12}
{\sc G.-X. Yuan, C.-H. Ho, and C.-J. Lin}, {\em An improved {GLMNET} for
  $\ell_1$-regularized logistic regression}, J. Mach. Learn. Res., 13 (2012),
  pp.~1999--2030.

\bibitem{Yuan06}
{\sc M.~Yuan and Y.~Lin}, {\em Model selection and estimation in regression
  with grouped variables}, J. R. Stat. Soc. Ser. B. Stat. Methodol., 68 (2006),
  pp.~49--67.

\bibitem{Yue19}
{\sc M.-C. Yue, Z.~Zhou, and A.~M.-C. So}, {\em A family of inexact {SQA}
  methods for non-smooth convex minimization with provable convergence
  guarantees based on the {L}uo--{T}seng error bound property}, Math. Program.,
  174 (2019), pp.~327--358.

\bibitem{Zhou21}
{\sc R.~Zhou and D.~P. Palomar}, {\em Solving high-order portfolios via
  successive convex approximation algorithms}, IEEE Transactions on Signal
  Processing, 69 (2021), pp.~892--904.

\end{thebibliography}

\appendix
\section{Proof of Proposition \ref{KL-subregular}}\label{secA2}

This part includes the proof of Proposition \ref{KL-subregular}, which requires the following two lemmas. 
\begin{lemma}\label{lemma-descent}
	(see \cite[Proposition A.24]{Bertsekas99}) Consider a function $h\!:\mathbb{R}^n\to\overline{\mathbb{R}}$ and a closed convex set $S\subset{\rm dom}\,h$. Suppose that $h$ is continuously differentiable on an open set containing $S$ and $\nabla h$ is Lipschitz continuous on $S$ with modulus $L_{\nabla h}$. Then,  
	\[
	h(y)\le h(x)+\langle\nabla h(x),y-x\rangle+(L_{\nabla h}/2)\|y-x\|^2
	\quad\ \forall x,y\in S.
	\]
\end{lemma}
\begin{lemma}\label{lemma-localEB}
	(see \cite[Corollary 2(ii)]{Ngai09}) Let $h\!:\mathbb{R}^n\to\overline{\mathbb{R}}$ be an lsc function and  $\overline{x}\in S_h\!:=\{x\in\mathbb{R}^n\ |\ h(x)\ge0\}$. If there exist $a,\gamma,\varepsilon>0$ such that $\gamma\|x^*\|[h(x)]^{\gamma-1}\ge a$ for all $x\in\mathbb{B}(\overline{x},\varepsilon)\backslash S_f$ and $x^*\in\widehat{\partial}h(x)$, then ${\rm dist}(x,S_h)\le(1/a)\max(h(x),0)^{\gamma}$ for all $x\in\mathbb{B}(\overline{x},\varepsilon/2)$, where $\widehat{\partial}h(x)$ denotes the regular subdifferential of $h$ at $x$.
\end{lemma}
\begin{proof}
{\bf(i)} From the expression of $F$ and the descent lemma, for any $x,y\in\mathbb{B}(\overline{x},\widetilde{\varepsilon})\cap{\rm dom}\,g$ and any $v\in\partial F(x)$, where $\widetilde{\varepsilon}$ is the same as the one in \eqref{fgrad-ineq1},   
\begin{equation} \label{Uprox}
	\begin{aligned}
		&F(y)-F(x)-\langle v,y-x\rangle
		=f(y)-f(x)+g(y)-g(x)-\langle v,y-x\rangle \\
		&\ge g(y)\!-\!g(x)-\langle v\!-\!\nabla\!f(x),y\!-\!x\rangle
		-\frac{1}{2}L'\|y-x\|^2\ge-\frac{1}{2}L'\|y-x\|^2,
	\end{aligned}
\end{equation}	
where the first inequality is due to  Lemma \ref{lemma-descent} with $S=\mathbb{B}(\overline{x},\widetilde{\varepsilon})$ by \eqref{fgrad-ineq1} and $L'$ is the same as the one in \eqref{fgrad-ineq1}, and the last one is using $v\!-\!\nabla\!f(x)\in\partial g(x)$ and 
the convexity of $g$. Since $\partial F$ is $q$-subregular at $\overline{x}$ for $0$, there exist $\varepsilon>0,\kappa>0$ such that \eqref{subregular-ineq0} holds for all $z\in\mathbb{B}(\overline{x},\varepsilon)$. Fix any $z\in\mathbb{B}(\overline{x},\delta)\cap[F(\overline{x})<F<F(\overline{x})+\delta]$ 
with $\delta=\min\{\epsilon,\varepsilon,\widetilde{\varepsilon}\}/2$, 
where $\epsilon$ is the constant from Assumption \ref{ass0}. Pick any $u\in\Pi_{\mathcal{S}^*}(z)$. Then, $\|u-\overline{x}\|\le\|u-z\|+\|z-\overline{x}\|\le 2\|z-\overline{x}\|\le\widetilde{\varepsilon}$.
By invoking \eqref{Uprox} with $y=u$ and $x=z$, for all $v\in\partial F(z)$,
$F(u)\ge F(z)+\langle v,u-z\rangle-\frac{1}{2}L'\|u\!-\!z\|^2$.
Since $\|u-\overline{x}\|\le2\|z-\overline{x}\|\le\epsilon$, 
by Assumption \ref{ass0} we have $F(u)\le F(\overline{x})$. Consequently,
\begin{align*}
	F(z)-F(\overline{x})&\le F(z)-F(u)
	\le\inf_{ v\in\partial F(z)}\langle v,z-u\rangle+({L'}/{2})\|z-u\|^{2}\\
	&\le{\rm dist}(0,\partial F(z)){\rm dist}(z,\mathcal{S}^*)
	+({L'}/{2})[{\rm dist}(z,\mathcal{S}^*)]^{2},\\
	&\le (\kappa+{\kappa^2L'}/{2})
	\max\big\{[{\rm dist}(0,\partial F(z))]^{1+q},[{\rm dist}(0,\partial F(z))]^{2q}\big\},
\end{align*}
where the last inequality is using \eqref{subregular-ineq0}. Note that $0<F(z)-F(\overline{x})\le\delta<1$. We have 
${\rm dist}(0,\partial F(z))\ge c\,(F(z)-F(\overline{x}))^{\max\{\frac{1}{2q},\frac{1}{1+q}\}}$ 
for a constant $c>0$ (depending on $\kappa$ and $L'$), so the function $F$ 
has the KL property of exponent $\max\{\frac{1}{2q},\frac{1}{1+q}\}$ at $\overline{x}$. 

\noindent
{\bf(ii)} Since $F$ has the KL property of exponent $\frac{1}{2q}$ for $q\in(1/2,1]$ at $\overline{x}$,	there exist $\varepsilon>0,\varpi>0$ and $c>0$ such that for all $z\in\mathbb{B}(\overline{x},\varepsilon)\cap[F(\overline{x})<F<F(\overline{x})+\varpi]$,
\begin{equation}\label{temp-KL12}
	{\rm dist}(0,\partial F(z))
	\ge 2q[c(2q\!-\!1)]^{-1}\big(F(z)\!-\!F(\overline{x})\big)^{\frac{1}{2q}}.
\end{equation}
From Assumption \ref{ass-1} (i)-(ii), the function $F$ is continuous at $\overline{x}$ relative to ${\rm dom}\,g$. Together with the local optimality of $\overline{x}$,
there exists $\widetilde{\varepsilon}>0$ such that
\begin{equation}\label{equa-continuous}
	F(\overline{x})\le F(z)\le F(\overline{x})+\varpi/2
	\quad\ \forall z\in\mathbb{B}(\overline{x},\widetilde{\varepsilon})\cap{\rm dom}\,g.
\end{equation}
Define $S_{\widetilde{F}}:=\big\{z\in\mathbb{R}^n\,|\,\widetilde{F}(z)\le 0\big\}$ with 
$\widetilde{F}(z):=F(z)+\delta_{\mathbb{B}(\overline{x},\widetilde{\varepsilon})}(z)-F(\overline{x})$ 
for $z\in\mathbb{R}^n$. Take $\delta=\min(\varepsilon,\widetilde{\varepsilon})/2$.
We next argue that for any $x\in\mathbb{B}(\overline{x},\delta)\backslash S_{\widetilde{F}}$,
\begin{equation}\label{aim-ineq21}
	{\rm dist}(0,\widehat{\partial}\widetilde{F}(x))
	\ge 2q[c(2q\!-\!1)]^{-1}\big(F(x)\!-\!F(\overline{x})\big)^{\frac{1}{2q}}
	=2q[c(2q\!-\!1)]^{-1}\big[\widetilde{F}(x)\big]^{\frac{1}{2q}}.
\end{equation}
Pick any $x\in\mathbb{B}(\overline{x},\delta)\backslash S_{\widetilde{F}}$.
If $x\notin{\rm dom}\,g$, inequality \eqref{aim-ineq21} automatically holds,
so it suffices to consider that $x\in[\mathbb{B}(\overline{x},\delta)\cap{\rm dom}\,g]\backslash S_{\widetilde{F}}$. Clearly, $F(x)-F(\overline{x})=\widetilde{F}(x)>0$. Along with \eqref{equa-continuous},
$x\in\mathbb{B}(\overline{x},\varepsilon)\cap[F(\overline{x})<F<F(\overline{x})+\varpi]$.
From \eqref{temp-KL12}, it follows that
\[
{\rm dist}(0,\partial F(x))\ge2q[c(2q\!-\!1)]^{-1}\big(F(x)\!-\!F(\overline{x})\big)^{\frac{1}{2q}}.
\]
By using the expression of $\widetilde{F}$ and noting that
$x\in{\rm int}\,\mathbb{B}(\overline{x},\widetilde{\varepsilon})\cap{\rm dom}\,g$,
it is immediate to have $\widehat{\partial}\widetilde{F}(x)=\partial\widetilde{F}(x)\subset\partial F(x)$.
Combining with the last inequality,	we obtain \eqref{aim-ineq21}, and then
the condition of Lemma \ref{lemma-localEB} is satisfied with $\gamma=\frac{2q-1}{2q}$ and $a=1/c$. From Lemma \ref{lemma-localEB}, for any $z\in\mathbb{B}(\overline{x},\delta/2)$, 
$\max(\widetilde{F}(z),0)\ge [c^{-1}{\rm dist}(z,S_{\widetilde{F}})]^{\frac{2q}{2q-1}}$ which, by the expressions of $\widetilde{F}$ and $S_{\widetilde{F}}$, is equivalent to saying that
\[
F(z)-F(\overline{x})\ge [c^{-1}{\rm dist}(z,S_{\widetilde{F}})]^{\frac{2q}{2q-1}}
\quad\ {\rm for\ all}\ z\in\mathbb{B}(\overline{x},\delta/2).
\]
Note that every point of $S_{\widetilde{F}}$ is a local minimizer of $F$. Clearly,
$S_{\widetilde{F}}\subset(\partial F)^{-1}(0)$. Then
\begin{equation}\label{Fgrowth}
	F(z)\ge F(\overline{x})+[c^{-1}{\rm dist}(z,(\partial F)^{-1}(0))]^{\frac{2q}{2q-1}}
	\quad\ \forall z\in\mathbb{B}(\overline{x},\delta/2).
\end{equation}
Now we argue that $\partial F$ is $(2q\!-\!1)$-subregular at $\overline{x}$ for $0$.
Pick any $x\in\!\mathbb{B}(\overline{x},{\delta}/{2})$. It suffices to consider 
$\partial F(x)\ne\emptyset$. If $x\in[F(\overline{x})<\!F<\!F(\overline{x})+\!\varpi]$,
by \eqref{temp-KL12} and \eqref{Fgrowth},
\begin{equation}\label{aim-ineq}
	{\rm dist}(0,\partial F(x))\ge
	2q[(2q\!-\!1)]^{-1}c^{\frac{2q}{1-2q}}\big[{\rm dist}(x,(\partial F)^{-1}(0))\big]^{\frac{1}{2q-1}}.
\end{equation}
If $x\!\notin\![F(\overline{x})<\!F<\!F(\overline{x})+\!\varpi]$, inequality \eqref{equa-continuous}
means that $F(x)=F(\overline{x})$. Along with \eqref{Fgrowth} and
${\rm dist}(x,(\partial F)^{-1}(0))=0$,  inequality \eqref{aim-ineq} holds.
Thus, inequality \eqref{aim-ineq} holds
for all $x\in\mathbb{B}(\overline{x},\frac{\delta}{2})$, and $\partial F$
is $(2q\!-\!1)$-subregular at $\overline{x}$ for $0$. 
\end{proof}

\section{Unit step-size results without error bound condition}\label{asec1}
\begin{lemma}\label{step-size-noEB}
	Let $\{x^k\}_{k\in\mathbb{N}}$ be the sequence generated by Algorithm \ref{IRPNM} with $\varrho\in(0,1/2)$.
	Consider any $\overline{x}\in\omega(x^0)$. Suppose that $\nabla^2\psi$ is strictly continuous at $A\overline{x}\!-\!b$ relative to $A({\rm dom}g)\!-\!b$ with modulus $L_{\psi}$, 
	then there exists $\overline{k}\in\mathbb{N}$ such that for all $x^k\in\mathbb{B}(\overline{x},{\varepsilon_0}/{2})$ with $k\ge\overline{k}$, 	$\alpha_k=1$, where $\varepsilon_0$ is the same as the one in Lemma \ref{xbark}.  
\end{lemma}
\begin{proof}
	From \eqref{dk-bound} we have $\frac{\|d^k\|}{\mu^k}\le\frac{(1+\|G_k\|)(1+\eta)}{a_2^2}r(x^k)^{1-2\varrho}$ 
	for each $k\in\mathbb{N}$. Recall that $\lim_{k\to\infty}r(x^k)=0$ 
	and $\lim_{k\to\infty}\|d^k\|=0$ by Proposition \ref{prop-xk} (iii). By the boundedness of $\{\|G_k\|\}$, there exists $\overline{k}\in\mathbb{N}$ such that for all $k\ge\overline{k}$, $\frac{\|d^k\|}{\mu^k}\le\frac{3(1-2\sigma)}{L_{\psi}\|A\|^3}$ and $\|d^k\|\le\frac{\varepsilon_0}{2}$. Fix any $k\ge\overline{k}$ with 
	$x^k\in\mathbb{B}(\overline{x},{\varepsilon_0}/{2})$.
	For such $x^k$ and any $x^k+st\beta^md^k$ with $s,t\in[0,1]$, we have $\|x^k+st\beta^md^k-\overline{x}\|\le\|x^k-\overline{x}\|+\|d^k\|\le\varepsilon_0$. 
	By invoking \eqref{Hessian-Lip}, we have
	\begin{align*}
		&f(x^k\!+\!\beta^md^k)\!-\!f(x^k)
		=\int_{0}^{1}\beta^m\langle[\nabla\!f(x^k\!+\!t\beta^md^k)\!-\!\nabla\!f(x^k)],d^k\rangle dt +\beta^m\langle\nabla\!f(x^k),d^k\rangle\\
		&=\int_{0}^{1}\int_{0}^{1}t\beta^{2m}\langle d^k,[\nabla^2\!f(x^k\!+\!st\beta^md^k)
		\!-\!\nabla^2\!f(x^k)]d^k\rangle dsdt\\
		&\quad+\int_{0}^{1}t\beta^{2m}\langle d^k,\nabla^2\!f(x^k)d^k\rangle dt
		+\beta^m\langle\nabla\!f(x^k),d^k\rangle\\
		&\le\frac{L_{\psi}\|A\|^3}{6}\beta^{3m}\|d^k\|^3
		+\frac{1}{2}\beta^{2m}\langle d^k,\nabla^2\!f(x^k)d^k\rangle
		+\beta^m\langle\nabla\!f(x^k),d^k\rangle\\
		&\le\frac{L_{\psi}\|A\|^3}{6}\beta^{3m}\|d^k\|^3
		+\frac{1}{2}\beta^{m}\langle d^k,[\nabla^2\!f(x^k)\!+\!\Lambda_kA^{\top}A]d^k\rangle
		+\beta^m\langle\nabla\!f(x^k),d^k\rangle
	\end{align*}
	where the last inequality is due to $\Lambda_k\ge 0$
	and $\nabla^2\!f(x^k)\!+\!\Lambda_kI\succeq 0$.
	In addition, from the convexity of the function $g$ and $\beta^m\in(0,1]$,
	it follows that
	\begin{align*}
		g(x^k\!+\!\beta^m d^k)-g(x^k)
		&\le\beta^m[g(y^k)-g(x^k)]=\beta^m[\ell_k(y^k)-\ell_k(x^k)
		-\langle\nabla\!f(x^k),d^k\rangle]\\
		&\le-\beta^m[\langle\nabla\!f(x^k),d^k\rangle+\frac{1}{2}\langle d^k,G_kd^k\rangle],
	\end{align*}
	where the last inequality is by \eqref{ineq-lk}.
	Adding the last two inequalities together yields
	\begin{align}\label{Fineq-noEB}
		F(x^k\!+\!\beta^md^k)-F(x^k)+\sigma\mu_k\beta^{m}\|d^k\|^2
		&\le\frac{L_{\psi}\|A\|^3}{6}\beta^{3m}\|d^k\|^3-(1/2-\sigma)\mu_k\beta^m\|d^k\|^2\nonumber\\
		&\le-\frac{1}{2}\beta^{m}\|d^k\|^3\Big(\frac{(1\!-\!2\sigma)\mu_k}{\|d^k\|}
		-\frac{L_{\psi}\|A\|^3}{3}\Big).
	\end{align}
	Recall that for all $k\ge\overline{k}$, $\frac{\|d^k\|}{\mu^k}\le\frac{3(1-2\sigma)}{L_{\psi}\|A\|^3}$.
	This means that \eqref{Fineq-noEB} holds with $m=0$ if $k\ge\overline{k}$ 
	and $x^k\in\mathbb{B}(\overline{x},{\varepsilon_0}/{2})$. Consequently, 
	the desired result follows.
\end{proof}

\section{Proof of Theorem \ref{converge}}\label{asec2}
Before providing the proof of Theorem \ref{converge}, we first show that the unit step-size occurs when $k$ is sufficiently large. 
\begin{lemma}\label{step-sizeS}
	Fix any $\overline{x}\in\omega(x^0)$ and $q>2\varrho$. Suppose that $\nabla^2\psi$ is strictly continuous at $A\overline{x}-b$ relative to $A({\rm dom}g)-b$ with modulus $L_{\psi}$, and that there exist $\varepsilon>0$ and $\kappa>0$ such that ${\rm dist}(x,\mathcal{S}^*)\le\kappa[r(x)]^q$ 
	for any $x\in\mathbb{B}(\overline{x},\varepsilon)$. 
	Then, there exists $\overline{k}\in\mathbb{N}$ such that for all $x^k\in\mathbb{B}(\overline{x},\varepsilon_1)$ with
	$k\ge\overline{k}$ and $\varepsilon_1=\min(\varepsilon,{\varepsilon_0}/{2})$, $\alpha_k=1$, where $\varepsilon_0$ is the same as the one in Lemma \ref{xbark}.  
\end{lemma}
\begin{proof}
	Since $\overline{x}\in\omega(x^0)$, then $\overline{x}\in\mathcal{S}^*$ by Proposition \ref{prop-xk} (iv). Recall that $d^k=y^k-x^k$ for each $k\in\mathbb{N}$. By invoking Lemmas \ref{lemma-yk}-\ref{xbark}, for all $x^k\in\mathbb{B}(\overline{x},\varepsilon_1)$,
	\begin{align*}
		\|d^k\|&\le\|y^k-\overline{x}^k\|+\|x^k-\overline{x}^k\|\\
		&\le\frac{\eta(1\!+\!\|G_k\|)}{a_2}[r(x^k)]^{1+\tau-\varrho}
		+\Big(\frac{L_{\psi}\|A\|^3{\rm dist}(x^k,\mathcal{S}^*)}{2\mu_k}
		+\frac{\Lambda_k\|A\|^2}{\mu_k}\!+\!2\Big){\rm dist}(x^k,\mathcal{S}^*).
	\end{align*} 
	From \eqref{rk-ineq1}, the last inequality and the error bound assumption,
	for all $x^k\in\mathbb{B}(\overline{x},\varepsilon_1)$,
	\begin{align}\label{dk-ineq0S}
		\|d^k\|
		&\le\frac{\eta(1\!+\!\|G_k\|)(2\!+\!L_{\nabla\!f})^{1+\tau-\varrho}}{a_2}[{\rm dist}(x^k,\mathcal{S}^*)]^{1+\tau-\varrho}
		+\frac{L_{\psi}\|A\|^3\kappa^{\frac{\varrho}{q}}}{2a_2}[{\rm dist}(x^k,\mathcal{S}^*)]^{2-\frac{\varrho}{q}}
		\nonumber\\
		&\qquad +\frac{\Lambda_k\|A\|^2\kappa^{\frac{\varrho}{q}}}{a_2}{\rm dist}(x^k,\mathcal{S}^*)^{1-\frac{\varrho}{q}}
		+2{\rm dist}(x^k,\mathcal{S}^*)\nonumber\\
		&\le\left[\frac{\eta(1\!+\!\|G_k\|)(2\!+\!L_{\nabla\!f})^{1+\tau-\varrho}}{a_2}
		[{\rm dist}(x^k,\mathcal{S}^*)]^{\tau-\varrho+\frac{\varrho}{q}}
		+\frac{L_{\psi}\|A\|^3\kappa^{\frac{\varrho}{q}}}{2a_2}{\rm dist}(x^k,\mathcal{S}^*)
		\right.\nonumber\\
		&\qquad\qquad\left. +\frac{\Lambda_k\|A\|^2\kappa^{\frac{\varrho}{q}}}{a_2}
		+2[{\rm dist}(x^k,\mathcal{S}^*)]^{\frac{\varrho}{q}}\right]
		{\rm dist}(x^k,\mathcal{S}^*)^{1-\frac{\varrho}{q}}.
	\end{align}
	Recall that $\{\Lambda_k\}$ is bounded and $\lim_{k\to\infty}{\rm dist}(x^k,\mathcal{S}^*)=0$ by Proposition \ref{prop-xk} (ii) and (iv).
	From the last inequality and $q>2\varrho$, there exists $\overline{k}\in\mathbb{N}$ 
	such that for all $k\ge\overline{k}$, when $x^k\in\mathbb{B}(\overline{x},\varepsilon_1)$, 
	$\frac{\|d^k\|}{(1-2\sigma)\mu_k}\le\frac{3}{L_{\psi}\|A\|^3}$ and
	$x^k+t d^k\in\mathbb{B}(\overline{x},\varepsilon_0)$ for all $t\in[0,1]$. 
	Then, by following the same arguments as those for Lemma \ref{step-size}, 
	the desired result follows. 
\end{proof}
In the following, we provide the proof of Theorem \ref{converge}.

\begin{proof}
	Let $\overline{k}\in\mathbb{N}$ be same as in Lemma \ref{step-sizeS}. Since $\lim_{k\to\infty}r(x^k)=0$ by Proposition \ref{prop-xk} (iii), $r(x^k)\le 1$ for all $k\ge\overline{k}$ (if necessary by increasing $\overline{k}$). Recall that $\lim_{k\to\infty}{\rm dist}(x^k,\mathcal{S}^*)=0$ by Proposition \ref{prop-xk} (iv) and $\{x^k\}_{k\in\mathbb{N}}$ is bounded by Proposition \ref{prop-xk} (ii). From \eqref{dk-ineq0S} and $\tau\!\ge\varrho$, there exists $c_1\!>0$ such that 
	\begin{equation}\label{dk-ineq1S}
		\|d^k\|\le c_1{\rm dist}(x^k,\mathcal{S}^*)^{1-\frac{\varrho}{q}}
		\quad{\rm for\ all}\ k\ge\overline{k}\ (\textrm{if necessary by increasing}\ \overline{k}).
	\end{equation}
	Together with Lemma \ref{step-sizeS}, it follows that for all $k\ge\overline{k}$, 
	when $x^k\in\mathbb{B}(\overline{x},\varepsilon_1)$,
	\begin{equation}\label{alphak-equaS}
		\alpha_k=1.
	\end{equation} 
	Next we argue that for all $k\ge\overline{k}$, 
	once $x^k\in\!\mathbb{B}(\overline{x},\varepsilon_1)$ and
	$x^{k+1}\!=x^k+d^k=y^k\in\mathbb{B}(\overline{x},\varepsilon_1)$,
	\begin{equation}\label{recursionS}
		{\rm dist}(x^{k+1},\mathcal{S}^*)=o({\rm dist}(x^{k},\mathcal{S}^*)).
	\end{equation} 
	Indeed, using ${\rm dist}(x^{k+1},\mathcal{S}^*)\!\le\!\kappa[r(x^{k+1})]^q$ and $r_k(y^k)\!\le\!\eta[r(x^k)]^{1+\tau}$ by \eqref{inexact-conda} and $r(x^k)\le 1$,
	\begin{align*}
		{\rm dist}(x^{k+1},\mathcal{S}^*)
		&\le\kappa[r(x^{k+1})]^q =\kappa\big[r(x^{k+1})-r_k(x^{k+1})+r_k(y^{k})\big]^q\nonumber\\
		&\le\kappa\big[r(x^{k+1})-r_k(x^{k+1})+\eta[r(x^{k})]^{1+\tau}\big]^q\nonumber\\
		&\le\kappa\big[|r(x^{k+1})-r_k(x^{k+1})|
		+\eta(2\!+\!L_{\nabla\!f})^{1+\tau}[{\rm dist}(x^k,\mathcal{S}^*)]^{1+\tau}\big]^q
	\end{align*}
	where the last inequality is using \eqref{rk-ineq1}. Note that $(1-t)x^k+tx^{k+1}\in\mathbb{B}(\overline{x},\varepsilon_0)$ for all $t\in[0,1]$. Using the expressions of $r$ and $r_k$ and inequality \eqref{dk-ineq1S} yields that
	\begin{align*}
		&|r(x^{k+1})-r_k(x^{k+1})|\le\|\nabla\!f(x^{k+1})-\nabla\!f(x^k)-G_k(x^{k+1}-x^k)\|\\
		&\le\Big\|\int_{0}^{1}\!\big[\nabla^2\!f(x^k\!+\!t(x^{k+1}\!-\!x^k))\!-\!\nabla^2\!f(x^k)\big]
		(x^{k+1}\!-\!x^k)dt\Big\|+(\Lambda_k\|A\|^2\!+\!\mu_k)\|x^{k+1}\!-\!x^k\|\\
		&\le\frac{L_{\psi}\|A\|^3}{2}\|x^{k+1}-x^k\|^2+(\Lambda_k\|A\|^2\!+\!\mu_k)\|x^{k+1}-x^k\|
		=\frac{L_{\psi}\|A\|^3}{2}\|d^k\|^2+(\Lambda_k\|A\|^2\!+\!\mu_k)\|d^k\| \\
		&\le \frac{L_{\psi}\|A\|^3c_1^2}{2}{\rm dist}(x^k,\mathcal{S}^*)^{2(1-\frac{\varrho}{q})}
		\!+c_1\Lambda_k\|A\|^2{\rm dist}(x^k,\mathcal{S}^*)^{1-\frac{\varrho}{q}}
		\!+c_1a_2(2\!+\!L_{\nabla\!f})^{\varrho}{\rm dist}(x^k,\mathcal{S}^*)^{1-\frac{\varrho}{q}+\varrho}.
	\end{align*}
	From the last two inequalities, we have 
	\begin{align}\label{dist-recursionS}
		&{\rm dist}(x^{k+1},\mathcal{S}^*)
		\le\kappa\Big[\frac{L_{\psi}\|A\|^3c_1^2}{2}{\rm dist}(x^k,\mathcal{S}^*)^{2(1-\frac{\varrho}{q})}
		+c_1a_2(2\!+\!L_{\nabla\!f})^{\varrho}[{\rm dist}(x^k,\mathcal{S}^*)]^{1-\frac{\varrho}{q}+\varrho}\nonumber\\
		&\qquad\qquad\qquad\quad +c_1\Lambda_k\|A\|^2{\rm dist}(x^k,\mathcal{S}^*)^{1-\frac{\varrho}{q}}		
		+\eta(2\!+\!L_{\nabla\!f})^{1+\tau}[{\rm dist}(x^k,\mathcal{S}^*)]^{1+\tau}\Big]^{q}\nonumber\\
		&=\kappa\Big[\frac{L_{\psi}\|A\|^3c_1^2}{2}{\rm dist}(x^k,\mathcal{S}^*)^{2-\frac{1+2\varrho}{q}}
		+c_1a_2(2\!+\!L_{\nabla\!f})^{\varrho}[{\rm dist}(x^k,\mathcal{S}^*)]^{(1-\frac{1}{q})(1+\varrho)}\nonumber\\
		&\quad\ 
		+c_1\Lambda_k\|A\|^2{\rm dist}(x^k,\mathcal{S}^*)^{1-\frac{1+\varrho}{q}}
		+\eta(2\!+\!L_{\nabla\!f})^{1+\tau}[{\rm dist}(x^k,\mathcal{S}^*)]^{1+\tau-\frac{1}{q}}\Big]^{q}
		{\rm dist}(x^k,\mathcal{S}^*).
	\end{align}
	Since $q>1+\varrho$, $\{\Lambda_k\}$ is bounded and $\lim_{k\to\infty}{\rm dist}(x^k,\mathcal{S}^*)=0$, the last inequality implies that the stated relation \eqref{recursionS} holds. Then, for any $\widetilde{\sigma}\in(0,1)$, there exists
	$0<\varepsilon_2<\varepsilon_1$ such that for all $k\ge\overline{k}$, 
	if $x^k\!\in\mathbb{B}(\overline{x},\varepsilon_{2})$ and
	$x^{k+1}\!=x^k+d^k=y^k\in\mathbb{B}(\overline{x},\varepsilon_{2})$,
	\begin{equation}\label{recursion1S}
		{\rm dist}(y^k,\mathcal{S}^*)\le\widetilde{\sigma}{\rm dist}(x^k,\mathcal{S}^*).
	\end{equation}
	
	Let $\overline{\varepsilon}\!:=\!\min\big\{(\frac{\varepsilon_2}{1+c_1})^{\frac{1}{1-\varrho/q}},
	\frac{\varepsilon_2}{2},
	[\frac{(1-\widetilde{\sigma}^{1-\varrho/q})\varepsilon_2}{2c_1}]^{\frac{1}{1-\varrho/q}}\big\}$. 
	Next we argue by induction that if some iterate 
	$x^{k_0}\in\mathbb{B}(\overline{x},\overline{\varepsilon})$ 
	with $k_0\ge\overline{k}$, then $\alpha_k=1$ and $x^{k+1}=y^k\in\mathbb{B}(\overline{x},\varepsilon_2)$
	for all $k\ge k_0$. Indeed, by letting $k=k_0$, remembering that
	$x^k\!\in\mathbb{B}(\overline{x},\overline{\varepsilon})$ and using
	\eqref{dk-ineq1S}, we get
	\[
	\|y^k\!-\!\overline{x}\|
	\le\|x^k\!-\!\overline{x}\|+\|y^k\!-\!x^k\|=\|x^k\!-\!\overline{x}\|+\|d^k\|
	\le\|x^k\!-\!\overline{x}\|+c_1\|x^k\!-\!\overline{x}\|^{1-\frac{\varrho}{q}}
	\le\varepsilon_2.
	\]
	Since $\overline{\varepsilon}<\varepsilon_1$, we get $\alpha_k=1$ by \eqref{alphak-equaS}, and hence $x^{k+1}=y^k\in\mathbb{B}(\overline{x},\varepsilon_2)$. Fix any $k>k_0$. 
	Assume that for all $k_0\le l\le k\!-\!1, \alpha_l=1,x^{l+1}=y^{l}\in\mathbb{B}(\overline{x},\varepsilon_2)$.
	By invoking \eqref{dk-ineq1S} and \eqref{recursion1S}, 
	\begin{align*}
		\|y^{k}\!-\!x^{k_0}\|&\!\le\!\sum_{l=k_0}^{k}\|d^l\|
		\!\le c_1\!\sum_{l=k_0}^{k}{\rm dist}(x^l,\mathcal{S}^*)^{1-\frac{\varrho}{q}}
		\!\le c_1\!\sum_{l=k_0}^{k}\widetilde{\sigma}^{(1-\frac{\varrho}{q})(l-k_0)}
		{\rm dist}(x^{k_0},\mathcal{S}^*)^{1-\frac{\varrho}{q}}\\
		&\le\!\frac{c_1}{1\!-\!\widetilde{\sigma}^{1-\frac{\varrho}{q}}}
		\|x^{k_0}\!-\!\overline{x}\|^{1-\frac{\varrho}{q}}
		\le\!\frac{\varepsilon_2}{2}.
	\end{align*}
	This implies that $\|y^k-\overline{x}\|\le\|y^k-x^{k_0}\|+\|x^{k_0}-\overline{x}\|
	\le\varepsilon_2$. Note that $x^{k}\in\mathbb{B}(\overline{x},\varepsilon_2)$.
	From \eqref{alphak-equaS}, $\alpha_k=1$ and
	$x^{k+1}\!=x^k+d^k=y^k\!\in\mathbb{B}(\overline{x},\varepsilon_2)$.
	The stated result holds.
	
	Since $\overline{x}\in\omega(x^0)$, we can find $k_0>\overline{k}$ 
	such that $x^{k_0}\in\mathbb{B}(\overline{x},\overline{\varepsilon})$.
	Then, as shown above, for all $k\ge k_0$, $\alpha_k=1$
	and $x^{k+1}\!=y^k\in\mathbb{B}(\overline{x},\varepsilon_{2})$.
	Since $\lim_{k\to\infty}{\rm dist}(x^k,\mathcal{S}^*)=0$, 
	for any $\epsilon>0$ there exists
	$\mathbb{N}\ni\overline{k}_0\ge k_0$ such that for all $k\ge\overline{k}_0$,
	${\rm dist}(x^{k},\mathcal{S}^*)\le\epsilon$.
	Fix any $k_1\ge k_2\ge\overline{k}_0$. From \eqref{recursion1S},
	it then follows that
	\begin{align}\label{xk-ineq2S}
		\|x^{k_1}\!-\!x^{k_2}\|&\le\sum_{j=k_2}^{k_1-1}\|x^{j+1}-x^j\|
		\le\sum_{j=k_2}^{k_1-1}\|d^j\|
		\le c_1\sum_{j=k_2}^{k_1-1}{\rm dist}(x^j,\mathcal{S}^*)^{1-\varrho/q}\nonumber\\
		&\le c_1\sum_{j=k_2}^{k_1-1}\widetilde{\sigma}^{(1-\varrho/q)(j-k_2)}
		{\rm dist}(x^{k_2},\mathcal{S}^*)^{1-\varrho/q} \nonumber\\
		&=\frac{c_1}{1-\widetilde{\sigma}^{1-\varrho/q}}{\rm dist}(x^{k_2},\mathcal{S}^*)^{1-\varrho/q}
		\le\frac{c_1\epsilon^{1-\varrho/q}}{1-\widetilde{\sigma}^{1-\varrho/q}}.
	\end{align}
	This shows that $\{x^k\}_{k\in\mathbb{N}}$ is a Cauchy sequence
	and converges to $\overline{x}\in\mathcal{S}^*$. 
	By passing the limit $k_1\to\infty$ to \eqref{xk-ineq2S} and using \eqref{dist-recursionS}, 
	we conclude that for any $k>\overline{k}_0$,
	\[
	\|x^{k+1}-\overline{x}\|
	\le\frac{c_1}{1-\widetilde{\sigma}^{1-\varrho/q}}{\rm dist}(x^{k+1},\mathcal{S}^*)
	\le O([{\rm dist}(x^{k},\mathcal{S}^*)]^{(q-\varrho)^2/q})
	\le O(\|x^k-\overline{x}\|^{(q-\varrho)^2/q}),
	\]
	and $(q-\varrho)^2/q>1$ provided that $q>\frac{2\varrho+1+\sqrt{4\varrho+1}}{2}$.
	That is, $\{x^k\}_{k\in\mathbb{N}}$ converges to $\overline{x}$ with
	the $Q$-superlinear rate of order $(q-\varrho)^2/q$.
\end{proof}

\end{document}